\documentclass[11pt]{article}
\usepackage{amssymb}
\usepackage{amsfonts}
\usepackage{amsmath}
\usepackage{array}
\usepackage{enumitem}
\usepackage{tikz}
\usepackage{multicol}
\usepackage{enumitem}
\usepackage{amsthm}
\usepackage{bm}
\usepackage{amssymb}
\usepackage{fullpage}
\usepackage{url}
\usepackage{verbatim}
\usepackage{mathtools}
\usepackage{mathrsfs}
\usepackage{diagbox}
\usepackage{amsmath}
\numberwithin{equation}{section}

\RequirePackage[top=1in,bottom=1in,left=1.4in,right=1.4in,includehead]{geometry}

\usepackage[colorlinks,
linkcolor=blue,
anchorcolor=blue,
citecolor=blue
]{hyperref}

\theoremstyle{plain}
\newtheorem{thm}{Theorem}[section]
\newtheorem{lem}[thm]{Lemma}
\newtheorem{prop}[thm]{Proposition}
\newtheorem{cor}[thm]{Corollary}
\newtheorem{conj}[thm]{Conjecture}
\newtheorem{ex}[thm]{Example}

\theoremstyle{definition}

\newcommand{\fmod}[1]{\ (\mathrm{mod}\ #1)}

\def\qed{\hfill \rule{4pt}{7pt}}

\def \dps {\displaystyle}

\allowdisplaybreaks[4]

\begin{document}
	\begin{center}
		{\Large \bf {Proof of Merca's stronger conjecture on truncated Jacobi triple product series}}
		\vskip 6mm
		{Xiangyu Ding$^a$,  Lisa Hui Sun$^b$
			\\[2mm]
			{\small Center for Combinatorics, LPMC, Nankai University, Tianjin 300071, P.R. China} \\[2mm]
			$^a$dingmath@mail.nankai.edu.cn, $^b$sunhui@nankai.edu.cn \\[2mm]}
	\end{center}
	
	\allowdisplaybreaks
	
	{\noindent \bf Abstract.}  In the study of theta series and partition functions, Andrews and Merca, Guo and Zeng  independently conjectured that  a truncated Jacobi triple product series  has nonnegative coefficients. This conjecture was proved analytically by Mao and combinatorially by Yee. In 2021, Merca proposed a stronger version of the conjecture, that  is, for positive integers $1\leq S<R$ with  $k\geq 1$, the coefficient of $q^n$ in
	the theta  series
	\[
	\frac{(-1)^{k} \sum_{j=k}^{\infty}(-1)^j q^{R j(j+1) / 2}\left(q^{-Sj}-q^{( j+1) S}\right)}{\left(q^S, q^{R-S}; q^R\right)_{\infty}}
	\]
	is nonnegative. Recently,  some very special cases of this conjecture have been proved and studied. For any given $R, S$ and $k$,  we take $s=S/(S,R), r=R/(S,R)$ which are coprime, equivalently. In this paper, we confirm Merca's stronger conjecture for sufficiently large $n$. Furthermore, for given $r, s$ and $k$, we provide a systematic method to determine an integer $N(r, s, k)$  such that Merca's stronger conjecture holds for $ n\geq N(r,s,k) $. More precisely, we decompose the infinite product in the denominator of the above theta series into two parts, one of which can be interpreted as the generating function of partitions with certain restricted parts and the other is a nonmodular infinite product.
	We derive the general upper and lower bounds for the coefficients of these two parts by using the partition theoretical method and the circle method, respectively. Further multiplying the partition part by the numerator of the theta series and considering the convolution with the nonmodular infinite product, we obtain the constant $N(r,s,k)$ and confirm Merca's stronger conjecture when $n\geq N(r,s,k)$. Moreover, we also show that when $k$ is sufficiently large, this conjecture holds directly for any $n\geq 0$.
	
	
	
	\allowdisplaybreaks
	
	\section{Introduction}
	
	In 2012, Andrews and Merca investigated a truncated form of Euler's pentagonal number theorem and obtained the following result \cite[Lemma 1.2]{andrews2012truncated}:
	\begin{align}\label{trun-euler}
		\frac{	(-1)^{k-1} }{(q ; q)_{\infty}} \sum_{j=-k}^{k-1}(-1)^j q^{j(3 j+1) / 2}=	(-1)^{k-1} + \sum_{n=1}^{\infty} \frac{q^{\binom{k}{2}+(k+1) n}}{(q ; q)_n}  {n-1\brack k-1}.
	\end{align}
	The truncated sum on the left hand side can be written in the form of
	\[
	\sum_{j=-k-\tau}^{k} (-1)^{j}q^{aj^{2}+bj}
	\]
	by replacing $k$ with $k+1$, where $a\in \mathbb{Q}^{+}, b\in \mathbb{Q} $ and $ \tau\in\{0,1\} $. Such truncated sums can be traced back to Shanks' version \cite[(2)]{shanks1951short} of  Euler's pentagonal number theorem,
	\[
	\sum_{j=-k}^{k}(-1)^j q^{j(3 j+1) / 2}=\sum_{n=0}^{k} \frac{(-1)^{n}(q ; q)_kq^{\binom{n}{2}+(k+1) n}}{(q ; q)_n}.
	\]
	Warnaar also studied the case of $ \sum_{j=-k-\tau}^{k}(-1)^j q^{j(5 j+c) / 2}$ in \cite[Theorem 1.1]{Warnaar2002} with $ c\in\{1,3\} $, for which Chapman \cite{Chapman2002} provided a combinatorial proof.
	
	One can see that the coefficient of $ q^{n} $ on the right hand side of \eqref{trun-euler}  is nonnegative, for which Andrews and Merca also gave a corresponding partition interpretation, which also leads to
	\begin{align*}
		(-1)^{k-1} \sum_{j=0}^{k-1}(-1)^j \big(p(n-j(3 j+1) / 2)-p(n-j(3 j+5) / 2-1)\big)\geq 0.
	\end{align*}	
	Later, Guo and Zeng studied a class of truncated forms of Gauss' triangular and 	square numbers series in \cite{GuoZeng13}, from which they also derive the recurrence relations for partition functions. 	
	Since then the truncated forms for theta series have attracted more and more attentions.
	
	In 2018,  Andrews and Merca  \cite{andrews2018truncated} obtained  the  unified truncated forms for Euler's pentagonal number theorem, Gauss's theta series for triangular and square numbers, which are further generalized by  Xia, Yee and Zhao in 2022 \cite{xia2022new}. We can easily notice that these  three classical theta series  all arise from  the Jacobi  triple product series. Actually, Andrews and Merca   \cite{andrews2012truncated}, and Guo and Zeng  \cite{GuoZeng13}  proposed the following conjecture for a truncated Jacobi triple product series,  independently.
	
	\begin{conj} \label{A-M-conj}
		For positive integers $k, R, S$ with  $1 \leq S<R / 2$, the coefficient of $q^n$ with $n \geq 1$ in
		\begin{align}\label{conj-weak}
			(-1)^{k-1} \frac{\sum_{j=0}^{k-1}(-1)^j q^{R j(j+1) / 2-S j}\left(1-q^{(2 j+1) S}\right)}{\left(q^S, q^{R-S}, q^R ; q^R\right)_{\infty}}
		\end{align}
		is nonnegative.
	\end{conj}
	
	The conjecture was proved analytically by Mao \cite{mao2015proofs} and combinatorially by Yee \cite{yee2015truncated} in 2015.  By applying the theory of Bailey pairs and Liu's expansion formula  \cite{Liuexpan} for $q$-series, Wang and Yee \cite{wang2019truncated} reproved this conjecture.
	Recently, by utilizing the iteration of Bailey's methods, we \cite{DingLisa}  obtained a general expression  of \eqref{conj-weak} in the form of multiple sums and the results can reduce to Andrews-Gordon and Bressoud type identities.
	
	In 2021, Merca \cite{Merca19, merca2022two} proposed a  stronger version of Conjecture \ref{A-M-conj} as follows.
	\begin{conj}\label{stronger-ja-1}
		For positive integers $1\leq S< R$ with $k \geq 1$, the theta series
		\begin{align}\label{conj-strong-1}
			(-1)^{k} \frac{\sum_{j=k}^{\infty}(-1)^j q^{R j(j+1) / 2}\left(q^{-Sj}-q^{( j+1) S}\right)}{\left(q^S, q^{R-S}; q^R\right)_{\infty}}
		\end{align}
		has nonnegative coefficients.
	\end{conj}
	Ballantine and Feigon \cite{ballantine2024truncated} proved combinatorially for the cases when $k\in\{1, 2, 3\} $.
	We also confirmed the cases when $ R=3S$ for $S\geq 1$ in \cite{DingLisa} by using the partition theoretical method. For the related works, see also Yao \cite{Yao6regular}, Ballantine and Feigon  \cite{ballantine2024truncated}, and Zhou \cite{Zhou}.

	Notice that  $S $ and $ R-S  $ are symmetric in our method. So without losing of generality, by replacing $ q^{(R,S)} $ with $ q $, we can assume that $1\leq S< R/2$ and $ R,\ S $ are coprime. In addition, the case of $R=2S$ is simpler and it will be proved as Corollary \ref{cor-r=2} in Section \ref{non-coprime}.  Then   Conjecture \ref{stronger-ja-1} can be restated equivalently as follows.
	\begin{conj}\label{stronger-ja-2}
		For positive integers $1\leq s< r/2$, $ (r,s)=1 $  and $k \geq 1$, the coefficient of $ q^{n} $ in
		\begin{align}\label{conj-strong-2}
			(-1)^{k} \frac{\sum_{j=k}^{\infty}(-1)^j q^{r j(j+1) / 2}\left(q^{-sj}-q^{( j+1) s}\right)}{\left(q^s, q^{r-s}; q^r\right)_{\infty}}
		\end{align}
		is nonnegative.
	\end{conj}
	
	In this paper, we are mainly devoted to determine a nonnegative integer constant $ N (r,s,k) $ respected to $r, s, k$ such that for such given $r, s$ and $k$, Merca's stronger conjecture holds for $ n \geq N (r,s,k) $.
	
	\begin{thm}\label{thm-large-n}
		For positive coprime integers $r$ and $s$ with  $1\leq s< r/2$  and $k \geq 1$, there exists a  constant $N (r,s,k)$ relative to $ r,s$ and $k $, such that  when $n\geq N (r,s,k)$,   the coefficient of $ q^{n} $ in
		\begin{align}\label{TJTTS_Merca}
			(-1)^{k} \frac{\sum_{j=k}^{\infty}(-1)^j q^{r j(j+1) / 2}\left(q^{-sj}-q^{( j+1) s}\right)}{\left(q^s, q^{r-s}; q^r\right)_{\infty}}
		\end{align}
		is nonnegative.
	\end{thm}
	Actually, we not only certify the existence of such constant $N $, but also provide a systematic method to calculate it precisely for given $r, s$ and $k$.
	
	To do so, we first decompose the denominator of the above theta series into two parts as follows,
	\begin{align}\label{two-part}
		\frac{1}{(q^s; q^r)_{\alpha}(q^{r-s}; q^r)_{\beta}}\cdot \frac{1}{\left(q^{\alpha r+s}, q^{(\beta+1)r-s}; q^r\right)_{\infty}},
	\end{align}
	where $ (\alpha,\beta)$ is chosen to be one of the elements in  $\{(2,2),(1,2),(2,1)\}  $ under different conditions. The denominator of the second part is a nonmodular infinite product.
	Firstly, by using the residue theorem  and the involutions between partitions,  we obtain the following theorem.
	\begin{thm}\label{main-thm}
		For positive integers $1\leq s< r/2$, $(s,r-s,r+s,2r-s)=1  $  and $k \geq 1$, the coefficient of $ q^{n} $ in
		\begin{equation}
			(-1)^{k} \frac{ \sum_{j=0}^{\infty}(-1)^{j}q^{r(j^{2}+(2k+1)j)/2-sj}(1-q^{2js+(2k+1)s})}{(1-q^{s})(1-q^{r-s})(1-q^{r+s})(1-q^{2r-s})}\label{mian-positive},
		\end{equation}
		is positive    when
		\begin{align}\label{n-Bound}
			n\geq\left(2 r^2 \sqrt[3]{{s}/{k}}+k\right) \left(4 r^3 \sqrt[3]{{s}/{k}}-r+2 s\right).
		\end{align}
	\end{thm}
	Consequently, when $ k $ is large enough to make the right hand side of \eqref{n-Bound} negative, that is to say  \eqref{mian-positive} has nonnegative coefficient for $ n\geq 0 $, which directly deduce to Conjecture \ref{stronger-ja-2} holding for $ n\geq 0 $. This leads to the following result.
	
	\begin{cor}\label{cor-main}
		When $ (s,r-s,r+s,2r-s)=1 $, $1\leq s< r/2 $ and
		\begin{align*}
			k\geq \frac{64r^{9}s}{(r-2s)^{3}},
		\end{align*}
		the theta series \eqref{mian-positive} has nonnegative coefficients for $ n\geq 0 $, and thereby  Conjecture \ref{stronger-ja-2} holds for such  $ k $.
	\end{cor}
	
	When $   (s,r-s,r+s,2r-s)\neq1 $, we obtain the following two results.
	\begin{thm}\label{subthm-2|s}
		If $ 2\mid  s$ , $1\leq s< r/2$ and $ k\geq  1$, then the coefficient  of $ q^{n} $ in
		\begin{align*}
			(-1)^{k} \frac{\sum_{j=0}^{\infty}(-1)^{j}q^{r(j^{2}+(2k+1)j)/2-sj}(1-q^{2js+(2k+1)s})}{(1-q^{s})(1-q^{r-s})(1-q^{r+s}) }
		\end{align*}
		is positive if $ n  \geq(z_0+k) (2 z_0r-r+2 s) $, where $z_0$ is the maximum zero of certain polynomials in $p$ as given in Section \ref{non-coprime}.
	\end{thm}
	\begin{thm}\label{subthm-2nmids}
		If $ 2\nmid  s$, $1\leq s< r/2$ and $ k\geq  1$, then the coefficient  of $ q^{n} $ in
		\begin{align*}
			(-1)^{k} \frac{\sum_{j=0}^{\infty}(-1)^{j}q^{r(j^{2}+(2k+1)j)/2-sj}(1-q^{2js+(2k+1)s})}{(1-q^{s})(1-q^{r-s})(1-q^{2r-s}) }
		\end{align*}
		is positive if $   n  \geq( z_0 +k) (2 z_0 r -r+2 s) $, where $z_0$ is the maximum zero of certain polynomials in $p$ as given in Section \ref{non-coprime}.
	\end{thm}
	
	Similar to Corollary \ref{cor-main}, from Theorem \ref{subthm-2|s} and Theorem \ref{subthm-2nmids} we also derive the following results,  the details can be seen in Section \ref{pf-large-n}.
	\begin{cor}\label{cor-3}
		When $ 2\mid s $, $1\leq s< r/2 $ and
		\begin{align*}
			k> \frac{4r^3s-2s(r-2s)}{r(r-2s)},
		\end{align*}
		the theta series in Theorem \ref{subthm-2|s} has nonnegative coefficients for $ n\geq0 $.
		
		When $ 2\nmid  s$, $1\leq s< r/2 $ and
		\begin{align*}
			k> \frac{8r^3s-2s(r-2s)}{r(r-2s)},
		\end{align*}
		the theta series in Theorem \ref{subthm-2nmids} has nonnegative coefficients for $ n\geq0 $. Moreover,  Conjecture \ref{stronger-ja-2} holds for the above $k$'s.
	\end{cor}
	
	Secondly, we observe that the second part on the right hand side of \eqref{two-part} is a nonmodular infinite product.  The work on the asymptotic properties of nonmodular infinite products was first utilized  by Lehner \cite{Lehner} and Livingood \cite{Livingdood} and developed by Grosswald \cite{Grosswald}. Recently, Chern used this method to prove Seo and Yee's conjecture \cite{Seo} related to the index of seaweed algebras and partitions \cite{Coll} for sufficiently large $n$, which was totally confirmed by Craig later in \cite{Craig}.
	
	For positive integers $ 1\leq a< M/2 $ and $ (a,M)=1 $, denote
	\begin{align}\label{def-Gq}
		G_{a,M} (q) = \frac{1}{(q^{2M+a },q^{3M-a};q^{M})_{\infty}} =\sum_{n=0}^{\infty}g_{a,M}(n)q
	\end{align}
	and
	\begin{align}\label{def-Phi-aM}
		\Phi_{a,M}=\log \frac{1}{(q^{a } ;q^{M})_{\infty}}, \quad \Psi_{a,M}=\log \frac{1}{(q^{2M+a } ;q^{M})_{\infty}}.
	\end{align}
	
	In 1959, Iseki gave an asymptotic formula for the coefficient of $  {G (q)}/{(q^{ a },q^{ M-a};q^{M})_{2}} $ as follows, see  \cite{Iseki}. Chern further generalised this result in \cite{Chern-1}.
	\begin{thm}[Iseki]
		Let
		\begin{align*}
			\frac{1}{ (q^{a},q^{M-a};q^{M})_{\infty}}=\sum_{n=0}^{\infty}g_{1}(n)q^{n}.
		\end{align*}
		When $ n $ is sufficiently large,
		\begin{align*}
			g_{1}(n) \sim \frac{1}{4\sin\frac{ a\pi}{M}} \left(3M\right)^{-\frac{1}{4}}n^{-\frac{3}{4}}\mathrm{exp}\left( 2\pi \sqrt{\frac{n}{3M}}\right).
		\end{align*}
	\end{thm}
	By using the circle method, we further extend Chern's idea to derive the following expression for the nonmodular infinite product  $ \Psi_{a, M}(q)+\Psi_{M-a, M}(q)$.
	\begin{thm}\label{Phi-aM}
		Let $X \geq 16$ be a sufficiently large positive number, and
		$$
		q:=e^{-\tau+\frac{2 \pi i h}{k_{1}}},
		$$
		where $1 \leq h \leq k_{1} \leq\lfloor\sqrt{2 \pi X}\rfloor=: N$ with $(h, k_{1})=1$ and $\tau:=X^{-1}+2 \pi i Y$ with $|Y| \leq \frac{1}{k_{1} N}$. Let $a,M$ be  the positive integers with $ 1\leq a< M/2 $ and $ (a,M)=1 $. Denote $b$ to be the unique integer between 1 and $(k_{1}, M)$ such that $b \equiv-h a\ (\bmod\ (k_{1}, M))$,
		then we have
		\begin{align*}
			\Psi _{a, M}(q)+\Psi _{M-a, M}(q)= & \frac{\pi^2}{\tau} \frac{(k_{1}, M)^2}{k_{1}^2 M}\left(\frac{2b(b-(k_{1},M))}{(k_{1}, M)^2 }+\frac{1}{3}\right)+2E_{1},
		\end{align*}
		where $2E_1$ is the error term such that
		$$
		|\Re(E_{1})| \ll_M X^{\frac{1}{2}} \log X.
		$$
	\end{thm}

	We obtain an asymptotic formula for the coefficient of $ G_{a,M}(q)$, for which the details will be given in Section \ref{pf-large-n}.
	\begin{thm}\label{approxiamtion-gn}
		When $ n $ is sufficiently large,
		\begin{align*}
			g_{a,M}(n)\sim \frac{1}{2\Delta_{M} \sin\frac{a\pi}{M}} \Big({\frac{\pi^{2}}{3Mn}}\Big)^{\frac{5}{2}}I_{-5}\Big(2\pi\sqrt{\frac{n}{3M}}\Big).
		\end{align*}
		where $ I_{v}(n) $ is the modified Bessel function of the first kind and $ \Delta_{M} :=a(M-a)(M+a)(2M-a)$.
	\end{thm}
	
	By replacing $ a,M $ by $ s,r $ in \eqref{def-Gq}, respectively, we also have the bounds for $  g_{s,r}(n) $.
	\begin{thm}\label{bound-gast}
		when $ n$ is sufficiently large,
		\begin{align*}
			&	g_{s,r}(n)\geq  \frac{0.99\pi^{4}}{2(3r)^{\frac{11}{4}}\Delta\sin (s\pi /r)} n^{-\frac{11}{4}}\mathrm{exp} \Big(2\pi\sqrt{\frac{n}{3r}}\Big)=:g^{d}(n),\\
			&
			g_{s,r}(n)\leq \frac{1.01\pi^{4}}{2(3r)^{\frac{11}{4}}\Delta\sin (s\pi /r)} n^{-\frac{11}{4}}\mathrm{exp} \Big(2\pi\sqrt{\frac{n}{3r}}\Big)=:g^{u}(n),
		\end{align*}
		where $ \Delta:=s(r-s)(r+s)(2r-s). $
	\end{thm}
	
	Finally, by analyzing the convolutions of the series as given in Theorems \ref{main-thm}--\ref{subthm-2nmids} and $G_{s,r}(q)$ with coefficients satisfying the boundary conditions in Theorem \ref{bound-gast}, respectively, we derive Theorem \ref{thm-large-n} and confirm Conjecture \ref{stronger-ja-2} for sufficiently large $n$.
	
	This paper is organized as follows. In Section \ref{preliminaries}, we shall introduce some basic definitions and notations on $q$-series and Bessel integrals. In Section \ref{coprime} and Section \ref{non-coprime}, we aim to prove Theorems \ref{main-thm}--\ref{subthm-2nmids} with coprime and non coprime parameters, respectively. In Section \ref{Nonmodular}, we discuss the nonmodular infinite product, together with the bounds on minor arcs of $ G_{a,M}(q)$ and its approximations  near the poles. Section \ref{pf-large-n} is mainly concerned with the derivation of the  asymptotic formula for $g_{a,M}(n)$ by using the circle method and the proof of Theorem \ref{thm-large-n}. In Section \ref{conclusion}, we give some remarks based on our method.  The appendix lists some verbose constant terms omitted  in Section \ref{coprime} and \ref{non-coprime}.
	
	\section{Preliminaries}\label{preliminaries}
	
	Throughout this paper, we adopt  standard notations and terminologies for $q$-series, see,  for example, \cite{Basic}. We assume that $|q|<1$.
	The $q$-shifted factorial is denoted by
	\[
	(a)_n=(a;q)_n=\begin{cases}
		1, & \text{\it if $n=0$}, \\[2mm]
		(1-a)(1-aq)\cdots(1-aq^{n-1}), & \text{\it if $n\geq 1$}.
	\end{cases}
	\]
	We also use the notation
	\[
	(a)_\infty=(a;q)_\infty=\prod_{n=0}^\infty (1-aq^n).
	\]
	There are more compact notations for the multiple $q$-shifted factorials
	\begin{align*}
		(a_1,a_2,\dots,a_m;q)_n&=(a_1;q)_n(a_2;q)_n \cdots(a_m;q)_n,\\
		(a_1,a_2,\dots,a_m;q)_{\infty}&=(a_1;q)_{\infty}(a_2;q)_{\infty}\cdots
		(a_m;q)_{\infty}.
	\end{align*}

	The well-known Jacobi triple product identity is given by \cite[II.28]{Basic}
	\begin{equation}\label{JTPI}
		\sum_{k=-\infty}^\infty (-1)^k q^{k\choose 2} a^k =(a,q/a,q;q)_\infty.
	\end{equation}
	
	Truncated theta series are closely related to integer partitions. Recall that a partition of a nonnegative integer $n$ is a finite nonincreasing sequence of positive integers $\lambda_1,\lambda_2,\ldots, \lambda_k$ such that $\sum_{i=1}^k \lambda_i=n$, see, for example, \cite{Partitions}. Let $p(n)$ denote  the number of integer partitions of $n$ with $p(0)=1$, which has the generating function
	\[
	\sum_{n=0}^\infty p(n)q^n=\frac{1}{(q;q)_\infty}.
	\] We then reveal the following important result given by Polya and Szeg\"{o}.
	
	\begin{lem}[{\cite[Problem 27.1]{Polya}}] \label{lem-coprime}
		Let $a_{1}, a_{2},\ldots, a_{\ell} $ be positive integers with $ (a_{i},a_{j})=1 $ when $ i\neq j $, then
		\begin{align}\label{lem-coprime-expression}
			\frac{1}{(1-q^{a_{1}})(1-q^{a_{2}})\cdots(1-q^{a_{\ell}})}=R\Big(\frac{1}{1-q}\Big)+\frac{S(q)}{1-q^{a_{1}a_{2}\cdots a_{\ell}}},
		\end{align}
		where $ R(x)$ and $ S(x) $ are polynomials of $x$ with degree  $\ell-1$ and smaller than $  a_{1} a_{2}\cdots a_{\ell} $, respectively.
	\end{lem}
	Let $ p_{\ell}(n) $ be the number of partitions with coprime numbers $a_{1}, a_{2},\ldots, a_{\ell} $ as parts, then the generating function of  $ p_{\ell}(n) $ can be stated as follows
	\begin{align}\label{genarting-function}
		\sum_{n=0}^{\infty}p_{\ell}(n)=\frac{1}{(1-q^{a_{1}})(1-q^{a_{2}})\cdots(1-q^{a_{\ell}})}.
	\end{align}
	By employing Lemma \ref{lem-coprime} with $\ell=2$, it leads to the following result. 	
	\begin{lem}[{\cite[Problem 26]{Polya}}]\label{lem-case-has-2parts}
		We have
		$$ p_{2}(n)= \left\lfloor\frac{n}{a_{1}a_{2}}\right\rfloor \text{ or }\   p_{2}(n)= \left\lfloor\frac{n}{a_{1}a_{2}} \right\rfloor+1,$$
		when $ 	\lfloor a \rfloor $ denotes the greatest integer no more than $ a. $
	\end{lem}

	Next, we recall some basic notation and properties which will be applied in this paper.  The symbol $(a_{1}, a_{2},\ldots, a_{\ell}) $ is used to denote the greatest common divisor of $a_{1}, a_{2},\ldots, a_{\ell} $. Let $ n\equiv a $  denote the least positive residue of $ n $ modulo $a$. For example, if $ n=25 $, $ a=7 $, then  $25\equiv7=4$. Let 	$ \textnormal{sup}\{f(n)\}_{0\leq n\leq M} $ and $ \textnormal{inf}\{f(n)\}_{0\leq n\leq M} $ denote supremum and infimum of the set $\left\{ f(n) \right\}$ for $ n=0,\ldots, M$, respectively. As usual, $ \Re(z) $ denotes the real part of the complex number $ z $. 	
	Let    $A \sim \mathbb{O}(x) $ stand for an expression $ A  $ satisfying $ |A| \leq x $. 	Denote by $ f(x)\ll g(x) $ the Vinogradov notation, which means that there exists a positive constant $ C $   satisfying $ |f(x)|\leq C|g(x)|$. If $ C $ depends on $y$, we write $ |f(x)|\ll_{y}  C g(x) $.
	
	Recall that the Gamma function $\Gamma(x)$  \cite{ArtinGamma} is given by \[ \Gamma(x)=\int_{0}^\infty t^{x-1}e^{-t} d t, \] where $\Re(x)>0$. It satisfies the reflection formula \cite[p. 250, eq. (3)]{Apostol}
	\begin{align}\label{Gamma-use}
		\Gamma(x)\Gamma(1-x) =\frac{\pi}{\sin (\pi x)}.
	\end{align}
	It is also known that $\Gamma(n)=(n-1)!$ for a positive integer $n\geq 1$ and so $\Gamma(1)=1$.
	
	The well-known Cauchy's integral formula states that
	\begin{align}
		f\left(z_0\right)=\frac{1}{2 \pi i} \oint_C \frac{f(z) d z}{z-z_0},
	\end{align}
	where  $C$ is a circle, oriented counterclockwise, and forms the boundary of   the closed disk  $  {\bigl \{}z:|z-z_{0}|\leq r{\bigr \}} $.	Moreover, the residue theorem \cite[Chapter 3]{Lang} can be stated as follows. When $f(z)$ is analytic except for the singularities $ z_{1},z_2,\ldots, z_{k} $, then
	\begin{align}\label{Cauchy}
		\frac{1}{2 \pi i} \oint_C f(z) d z=\sum_{i=1}^{k}\underset{z=z_i}{\operatorname{Res}} f(z),
	\end{align}
	where $\underset{z=z_i}{\operatorname{Res}} f(z) $ is the residue of $f(z)$ at the singularity $z_i$,
	
	Let $ \zeta(s) $ and $ \zeta(s,a) $ denote Riemann and Hurwitz zeta functions and $ \zeta^{\prime}(s,a) $ or $ \frac{d}{d s} \zeta (s,a) $ to denote the partial derivative of Hurwitz zeta function on $ s $ \cite{Adamchik1998}, where
	\begin{align*}
		\zeta(s) =\sum_{k=1}^{\infty}\frac{1}{k^{s}}\quad \text{and}\quad \zeta(s,a) =\sum_{k=0}^{\infty}\frac{1}{(k+a)^{s}}.
	\end{align*}
	Note that it is known that both $\zeta(s)$ and $\zeta(s,a)$ have a simple pole with residue one at $s=1$.  Moreover, $\zeta(0)=-1/2$,  $\zeta(2)=\pi^2/6$, and Hurwitz zeta function satisfies the following relations
	\begin{align}
		&  \zeta(0,a)=\frac{1}{2}-a, \label{special-1} \\
		& \frac{d}{d s} \zeta(s, a)|_{s=0}=\ln (\Gamma(a))-\frac{1}{2} \ln (2 \pi),\label{special-2} \\
		& \frac{d}{d s} \zeta(s,0)|_{s=0}=-\frac{1}{2} \ln (2 \pi), \label{special-3}
	\end{align}
	where first equation can be found in \cite{Apostol} and the last two ones lie in \cite{Bailey2007}.
	From \cite[Theorem 12.8]{Apostol}, there is a  functional equation  satisfied by $\zeta(s,a)$
	\begin{align}\label{Hurwitz}
		\zeta\Big(s, \frac{\lambda}{\kappa}\Big)= & 2 \Gamma(1-s)(2 \pi \kappa)^{s-1}\Big(\sin \frac{\pi s}{2} \sum_{1 \leq \nu \leq \kappa} \cos \frac{2 \pi \lambda \nu}{\kappa} \zeta\big(1-s, \frac{\nu}{\kappa}\big)  \notag\\
		&  +\cos \frac{\pi s}{2} \sum_{1 \leq \nu \leq \kappa} \sin \frac{2 \pi \lambda \nu}{\kappa} \zeta\big(1-s, \frac{\nu}{\kappa}\big)\Big).
	\end{align}
	
	Hurwitz zeta functions satisfy the following two summation formulas, which can be seen in \cite{Blagouchine}.
	\begin{lem}
		For any $\theta=1,2, \ldots, k$,
		\begin{align}
			&\sum_{1 \leq \alpha \leq k} \cos \frac{2 \pi \alpha \theta}{k} \zeta\left(0, \frac{\alpha}{k}\right)=-\frac{1}{2},\label{cos-zata-0}\\
			&	\sum_{1 \leq \alpha \leq k} \cos \frac{2 \pi \alpha \theta}{k} \zeta\left(2, \frac{\alpha}{k}\right)=\frac{\pi^2}{6}\left(6 \theta^2-6 k \theta+k^2\right).\label{cos-zata-2}
		\end{align}
	\end{lem}
	Based on \cite[Theorem 12.23]{Apostol} and some elementary calculations, the following result holds.
	\begin{lem}
		For $ s=\sigma+it $ and $ |t|\geq 5 $, we take $ N = 1 + \lfloor|t |\rfloor$ then have
		\begin{align}\label{abs-zeta-geq5}
			|\zeta(s, \alpha)|&\leq1+\frac{(N+\alpha)^{1-\sigma}}{1-\sigma}+\frac{(N+\alpha)^{1-\sigma}}{|s-1|}+\frac{1}{2}|s| \Big(1+\frac{20}{3}N^{\delta}+(1+N)^{\delta}\Big) \notag\\
			& \quad\  +\frac{1}{2}|s|\Big(1+\frac{(N+\alpha)^{-\sigma}}{-\sigma}\Big)+\frac{1}{2}|s||s+1| \frac{(N+\alpha)^{-\sigma-1}}{1+\sigma}+ |a^{-s}|.
		\end{align}
	\end{lem}
	
	For $ 0<\alpha<1 $, it can be seen that \cite[(3.5)]{Chern}
	\begin{align}\label{abs-zeta-leq5}
		|\zeta(-3/4+i t, \alpha)| \leq \frac{2 \Gamma\left(\frac{7}{4}\right) \zeta\left(\frac{7}{4}\right) \cosh \big(\frac{\pi|t|}{2}\big)}{(2 \pi)^{\frac{7}{4}}}.
	\end{align}
	
	The Mellin transformation \cite{Galambos} on a function $f(x)$ and its inversion form are defined as follows
	\begin{align}
		&\varphi(s)=\{ {\mathcal{M}} f\}(s)=\int_0^{\infty} x^{s-1} f(x) d x, \notag\\
		&f(x)=\{ {\mathcal{M}}^{-1} \varphi\}(x)=\frac{1}{2 \pi i} \int_{c-i \infty}^{c+i \infty} x^{-s} \varphi(s) d s=:\frac{1}{2 \pi i} \int_{(c)}  x^{-s} \varphi(s) d s.\label{Mellin-2}
	\end{align}
	
	The first modified Bessel function \cite[p.269]{Bessel} is given by
	\begin{align}\label{Bessel}
		I_v(z)&:=\frac{1}{2 \pi i} \int_{C} s^{-v-1} \mathrm{exp}\Big(  \frac{z}{2}\big(s+\frac{1}{s}\big) \Big)ds\notag \\
		&=\frac{e^z}{(2 \pi z)^{1 / 2}}\left(1+\delta_{1}\right)-i e^{-v \pi i} \frac{e^{-z}}{(2 \pi z)^{1 / 2}}\left(1+\gamma_{1}\right),
	\end{align}
	where $ C $ is the Hankel contour and
	\begin{align*}
		&|\gamma_{n}|\leq 	2 \exp \Big\{\Big|\Big(v^2-\frac{1}{4}\Big)   z^{-1}\Big|\Big\}\Big|A_n(v) z^{-n}\Big|,\\
		&|\delta_{n}|\leq 	2\pi \exp \Big\{\Big|\Big(v^2-\frac{1}{4}\Big) z^{-1}\Big|\Big\}\Big|A_n(v) z^{-n}\Big| ,\\
		&	A_s(v)=\frac{(4 v^2-1^2)(4 v^2-3^2) \cdots( 4v^2-(2 s-1)^2)}{s!8^s}.
	\end{align*}
	
	\section{Proof of the coprime case: Theorem \ref{main-thm}}\label{coprime}
	
	In this section, we are mainly concerned with the proof of Theorem \ref{main-thm}. Firstly, we consider the calculation of   $ p_{\ell}(n)$ for   $\ell=4$ and $ a_{i} $ being replaced by  $s,r-s,r+s,2r-s$ for $1\leq i\leq 4$, respectively, which follows that
	\[
	\sum_{n=0}^{\infty}p_{4}(n)q^{n}=\frac{1}{(1-q^{s})(1-q^{r-s})(1-q^{r+s})(1-q^{2r-s})}.
	\]
	Based on the residue theorem \eqref{Cauchy} and Lemma \ref{lem-coprime}, we obtain the following expression for $p_4(n)$.
	
	\begin{prop}\label{poly-part} For $p_4(n)$ defined as above, we have that
		\begin{align} \label{poly-parti}
			p_{4}(n)=P_{3}(n)+C_{4}+Q(n\equiv\Delta),
		\end{align}
		where \begin{align}\label{P3}
			P_{3}(n)=\frac{n^3+6 n^2 r+n \left(21 r^2+2 r s-2 s^2\right)/2}{6\Delta},
		\end{align}
		$ C_{4} $ is a constant related    to $ r$ and  $s $, and $Q(n)$ is a periodic function with period of  $ \Delta=s(r-s)(r+s)(2r-s )$.
	\end{prop}
	\begin{proof}
		When $ \ell=4 $, from \eqref{lem-coprime-expression} and (\ref{genarting-function}), it follows that
		\begin{equation} \label{gf-p4rt}
			\prod_{i=1}^{4}\frac{1}{1-q^{a_{i}}}=\sum_{i=1}^{4}\frac{c_{i}}{(1-q)^{5-i}}+\sum_{i=0}^{a_{1}a_{2}a_{3} a_{4}-1}\frac{b_{i}q^{i}}{1-q^{a_{1}a_{2}a_{3} a_{4}}},
		\end{equation}
		where $ c_{i} $ and $b_i$ are constants.
		By using the residue theorem \eqref{Cauchy},  we have that for $ i=1,2,3 $,
		\begin{align*}
			c_{i}=\frac{(-1)^{i-1}}{(i-1)!}\lim_{q\to1} \frac{\textnormal{d}^{i-1}}{\textnormal{d}q^{i-1}} \Big(\frac{(1-q)^{4}}{\prod_{i=1}^{4}(1-q^{a_{i}})}\Big).
		\end{align*}
		Thus we obtain that
		\[ \begin{aligned}
			&c_{1}=\frac{1}{a_{1}a_{2}a_{3}a_{4}}\\[3pt]		&c_{2}=\frac{{a_{1}+a_{2}+a_{3}+a_{4}-4}}{2a_{1}a_{2}a_{3}a_{4}}\\[3pt]
			&c_{3}=\frac{1}{12a_{1}a_{2}a_{3} a_{4}}\Big(\sum_{i=1}^{4}a_{i}(a_{i}-9)-\frac{3}{2}\sum_{i\neq j}a_{i}a_{j}+14\Big).
		\end{aligned}\]
		Taking  the  coefficient of $q^n$ in the first sum on the right hand side of \eqref{gf-p4rt} and  replacing $ a_{i} $  by  $s,r-s,r+s,2r-s$ for $1\leq i \leq 4 $, respectively, we see that
		\begin{align*}
			P_3(n)+C_4&=\sum_{i=1}^{3} c_{i}\binom{-5+i}{n}+c_4\\
			&=\frac{(n+1) \left(2 n^2+12 n r-2 n+21 r^2+2 r s-12 r-2 s^2+2\right)}{12 s (r-s) (2 r-s) (r+s)}+c_4,
		\end{align*}
		which leads to \eqref{P3} by simplifying.
		Moreover, by extracting the coefficients of $q^n$ in the following sum
		\[
		\sum_{i=0}^{a_{1}a_{2}a_{3} a_{4}-1}\frac{b_{i}q^{i}}{1-q^{a_{1}a_{2}a_{3} a_{4}}}=\sum_{i=0}^{a_{1}a_{2}a_{3} a_{4}-1}\sum_{k=0}^\infty b_iq^{i+(a_1a_2a_3a_4) k},
		\]
		it implies  the periodic part $Q(n)$ such that  $Q(n)=b_i$ if $n\equiv i \fmod{a_1a_2a_3a_4}$, which completes the proof.
	\end{proof}
	
	We further consider the constant term and the periodic term in \eqref{poly-parti}.
	Denote
	\begin{align}\label{constant-term}
		C_4(n)=	C_{4}+Q(n\equiv\Delta),
	\end{align}
	from which we can obtain the boundray conditions for $p_4(n)$.
	
	\begin{lem}\label{limit-C(n)}
		For any $n\geq 0$, we have
		\begin{align}\label{bound-p4}
			P_{3}(n)+C_4^{d}(\Delta-1)<p_{4}(n)<P_{3}(n)+C_4^{u}(\Delta-1),
		\end{align}
		where $  C_4^{d}(n)=(A^{d}n^{2}+B^{d}n+{  6\Delta} )/(6\Delta) $, $  C_4^{u}(n)=(A^{u}n^{2}+B^{u}n+{  6\Delta}) /(6\Delta) $,
		\[ \begin{aligned}
			&A^{d}=-6 r - 6 r^2 + 6 r s - 6 s^2,\\
			&B^{d}=-1 -  {33 r^2}/{2} + 7 r s + 12 r^3 s - 7 s^2 - 6 r^2 s^2	-12 r s^3 + 6 s^4,\\
			&A^{u}=-6 r + 6 r^2 + 6 r s - 6 s^2,\\
			&B^{u}=-1 -  {9 r^2}/{2} + 5 r s + 12 r^3 s - 5 s^2 - 6 r^2 s^2	-12 r s^3 + 6 s^4,
		\end{aligned} \]
		which satisfy that 		
		\begin{align}\label{ineq-Cn}
			C_4^{d}(n)<C_4(n)<C_4^{u}(n).
		\end{align}
	\end{lem}
	\begin{proof}
		Let $ p_{i,j}(n)$ denote the number of partitions with coprime numbers $a_{i}, a_{j}$ as parts, then it is easy to see that $p_4(n)=\sum_{m=0}^n p_{1,2}(m)p_{3,4}(n-m)$. From Lemma \ref{lem-case-has-2parts}, we have the boundary conditions for $ p_{i,j}(n) $ as follows
		\begin{align}\label{limit-p_ij}
			\frac{n}{a_{i} a_{j}}-1\leq p_{i,j}(n)\leq \frac{n}{a_{i} a_{j}}+1 .
		\end{align}
		By calculating the convolution of $ p_{1,2}(n) $ and $ p_{3,4}(n) $ and  replacing $ a_{i} $  by  $s,r-s,r+s,2r-s$ for $1\leq i \leq 4 $, respectively, we obtain that
		\[\begin{aligned}
			&	T^{d}(n) =\sum _{m=0}^n \left(\frac{m}{s (r-s)}-1\right) \left(\frac{n-m}{(r+s) (2 r-s)}-1\right) ,\\
			&	T^{u}(n) =\sum _{m=0}^n \left(\frac{m}{s (r-s)}+1\right) \left(\frac{n-m}{(r+s) (2 r-s)}+1\right),
		\end{aligned}\]
		such that for $ 0\leq n\leq \Delta-1 $,
		\[   T^{d}(n)\leq p_{4}(n)\leq T^{u}(n). \]
		By direct computation, we have
		\begin{align}
			&T^{d}(n)=\frac{(n+1) \left(n^2-6 n r^2-6 n r s+6 n s^2-n+6\Delta\right)}{6\Delta}\label{Td},\\
			&T^{u}(n)=\frac{(n+1) \left(n^2+6 n r^2+6 n r s-6 n s^2-n+6\Delta\right)}{6\Delta}\label{Tu},
		\end{align}
		and $\Delta=s(r-s)(r+s)(2r-s)$.
		
		By 	Proposition \ref{poly-part}, we know that for any given $n$,
		\begin{align}\label{generating-p4}
			p_{4}(n)=P_{3}(n)+C_4(n) ,
		\end{align}
		and for $ 0\leq n\leq \Delta-1 $,
		\[ 	T^{d}(n)\leq P_{3}(n)+C_4(n)   \leq T^{u}(n), \]
		then
		\[ T^{d}(n)-P_{3}(n)\leq C_4(n)   \leq T^{u}(n) -P_{3}(n).\]
		We further define
		\begin{align*}
			&C_4^{d}(n):= T^{d}(n)-P_{3}(n), \quad  C_4^{u}(n) := T^{u}(n) -P_{3}(n),
		\end{align*}
		respectively. Then $ C_4^{d}(n)$, $ C_4^{u}(n) $ and  (\ref{ineq-Cn}) can be obtained after combining (\ref{P3}), (\ref{Td}) and (\ref{Tu}).
		
		It can be seen that the quadratic function $ C_4^{u}(n) $ is monotonically increasing for $  0\leq n\leq \Delta-1  $, thus $ \textnormal{sup}\{C_4^{u}(n)\}_{0\leq n\leq \Delta-1}=C_4^{u}(\Delta-1)$.
		
		For $ C_4^{d}(n):= A^{d}n^{2}+B^{d}n+C_{0}$, a quadratic function respect to $ n $, we find its maximum point is less than  $(\Delta-1) /2$, that is
		\[\begin{aligned}
			&-B^{d}-A^{d}(\Delta-1)=12 r^5 s+r^4 \left(6 s^2+12 s\right)+r^3 \left(-30 s^3-6 s^2-12 s\right)+r^2 \\
			&(-12 s^3+6 s^2+\frac{21}{2})+r \left(18 s^5+6 s^4+12 s^3+s-6\right)-6 s^6-6 s^4-s^2+1>0,
		\end{aligned} \]
		which implies that
		$ \textnormal{inf}\{C_4^{d}(n)\}_{0\leq n\leq \Delta-1}=C_4^{d}(\Delta-1)$.
		The proof is complete.
	\end{proof}
	
	For convenience, denote
	\begin{align}\label{sym-J}
		J=C_4^{u}(\Delta-1)-C_4^{d}(\Delta-1)=C_4^{u}-C_4^{d}.
	\end{align}
	
	Now, we are ready to prove  Theorem \ref{main-thm}.
	
	{\noindent   \bf  Proof of Theorem \ref{main-thm}}. First, note that the numerator of \eqref{mian-positive} can be derived from (\ref{conj-strong-2}) by the following procedure
	\begin{align}
		& (-1)^{k}\sum_{j=k}^{\infty}(-1)^j q^{r j(j+1) / 2}\left(q^{-sj}-q^{( j+1) s}\right)\nonumber\\
		&=q^{rk(k+1)/2-sk}\sum_{j=0}^{\infty}(-1)^{j}q^{r(j^{2}+(2k+1)j)/2-sj}(1-q^{2js+(2k+1)s})\nonumber \\
		&=q^{rk(k+1)/2-sk}\sum_{j=0}^{\infty}\left(q^{t_{1,j}}-q^{t_{2,j}}-q^{t_{3,j}}+q^{t_{4,j}}\right), \label{numberator-c}
	\end{align}
	for which the above expression is obtained by considering the parities of $j$ and
	\begin{align*}
		&t_{1,j}=2 j^2 r+2 j k r+j r-2 j s,\\
		&t_{2,j}=2 j^2 r+2 j k r+j r+2 j s+2 k s+s,\\
		&t_{3,j}=2 j^2 r+2 j k r+3 j r-2 j s+k r+r-s,\\
		&t_{4,j}=2 j^2 r+2 j k r+3 j r+2 j s+k r+2 k s+r+2 s.
	\end{align*}
	Then assume  that \eqref{mian-positive} can be expanded as follows
	\begin{align}\label{main-transform}
		\sum_{n=0}^{\infty}	C(n)q^{n}=	\frac{\sum_{j=0}^{\infty}\left(q^{t_{1,j}}-q^{t_{2,j}}-q^{t_{3,j}}+q^{t_{4,j}}\right)}{(1-q^{s})(1-q^{r-s})(1-q^{r+s})(1-q^{2r-s})}.
	\end{align}
	By applying the generating function (\ref{generating-p4}) for $p_4(n)$, we can write the  coefficients of $q^n$ in \eqref{main-transform} as follows
	\begin{align}\label{p4-tj}
		C(n)=\sum_{j=0}^\infty   \Big( p_{4}(n-t_{1,j}) -p_{4}(n-t_{2,j}) -p_{4}(n-t_{3,j}) +p_{4}(n-t_{4,j}) \Big).
	\end{align}
	In fact there are finite number of  nonzero terms in \eqref{p4-tj} since $p_4(n)=0$ if $n<0$. We take the partial sum of  $C(n)$ as follows
	\[ \begin{aligned}
		C(n,p)=\sum_{j=0}^{p-1}   \Big( p_{4}(n-t_{1,j}) -p_{4}(n-t_{2,j}) -p_{4}(n-t_{3,j}) +p_{4}(n-t_{4,j}) \Big).
	\end{aligned} \]
	By applying the boundary conditions \eqref{bound-p4}  for $ p_{4}(n) $ and using the symbol \eqref{sym-J},  it is obvious to see that
	\begin{align}\label{C0}
		&\quad C(n,p)\notag \\
		&>\sum_{j=0}^{p-1}   \Big( P_{3}(n-t_{1,j})- P_{3}(n-t_{2,j})
		-P_{3}(n-t_{3,j} )+P_{3}(n-t_{4,j})-2J \Big) \notag \\
		&=\sum_{j=0}^{p-1}   \Big( P_{3}(n-t_{1,j}) -P_{3}(n-t_{2,j})-P_{3}(n-t_{3,j}) +P_{3}(n-t_{4,j}) \Big)-2pJ\notag \\
		&=:C_{0}(n,p).
	\end{align}

	Throughout the proof, we denote $c_{i,j}$ the constants relative to $ r,s,p,k $, for which the specific forms will not be fully demonstrated  for simplicity and one can refer to the appendix for details. Moreover, we also observe that $c_{0,2}, c_{1,3},c_{2,3},c_{3,3}$ are positive, which can be easily verified by replacing $r=ts$ in their corresponding expressions and noting that $t> 2$ since $1\leq s < \frac{r}{2}$.
	
	We claim that for
	\begin{align}\label{p-condition}
		p\geq 2r^{2}(s/k)^{1/3},
	\end{align}  and
	$ 	n\geq t_{4,p-1}, $
	$ C(n)>0. $ In other words, the coefficient  of $ q^{n} $ in \eqref{main-transform} is positive when $n\geq t_{4,p-1}$.
	
	By noting that $1\leq s < \frac{r}{2}$, it follows that $0\leq t_{1,j} <t_{2,j}<t_{3,j}<t_{4,j}<t_{1,j+1}$. So all the nonnegative integers fall into the 	union of the following intervals
	\[
	\bigcup_{p\geq 0}[t_{1,p},t_{1,p+1})=\bigcup_{p\geq 0} \Big[ \Big. \Big. 2 p^2 r+2 p k r+pr-2 p s, 2 (p+1)^2 r+2 (p+1) k r+(p+1) r-2 (p+1) s\Big).
	\]
	Thus to prove the claim, it is sufficient to show that it holds for $n\in [ 2 p^2 r+2 p k r+pr-2 p s, 2 (p+1)^2 r+2 (p+1) k r+(p+1) r-2 (p+1) s )$ for any given $p$ satisfying \eqref{p-condition}.
	Then the proof of the claim can be decomposed into the following four cases.
	
	{\noindent \bf  Case 0:} If $ t_{4,p-1}\leq n < t_{1,p}$, then $C(n)=C(n,p)>C_0(n,p)$.
	
	From Proposition \ref{poly-part} and \eqref{C0}, we have
	\[ \begin{aligned}
		&C_{0}(n,p)=\frac{1}{\Delta}\Big(-6 n^2 p s -
		n p s (30 r + 6 k r - 12 k^2 r
		\\
		&\qquad\qquad\quad - 36 k p r - 24 p^2 r - 6 s - 12 k s)+c_{0,1}\Big).
	\end{aligned} \]
	We find the  first derivation of $ C_{0}(n,p) $  on $n$ is positive,  since
	\begin{align}\label{Derivation-0}
		&\quad {6\Delta C_{0}^{\prime}(n,p)>6\Delta C_{0}^{\prime}(n,t_{1,p}) } \notag\\[3pt]
		&=p s\left(12 k^2 r  -6 k r  +12 k s -30 r  +6 s \right)+p^2s \left(12 k r -12 r  +24 s \right) >0.
	\end{align}
	Thus $ C_{0}(n,p) $ is monotonically increasing in this case and if $ p\geq 2r^{2}(s/k)^{1/3} $,
	\begin{align}\label{p-bound}
		6\Delta C_{0}(n,p) >6\Delta C_{0,d}(t_{4,p-1},p)=12 k p^4 r^2 s-96 p r^8 s^2+c_{0,2} >0.
	\end{align}
	
	{\noindent \bf  Case 1:} If $ t_{1,p}\leq n < t_{2,p}$, then we have
	\[ C(n) = C(n,p)+p_{4}(n-t_{1,p}). \]
	Therefore
	\[ \begin{aligned}
		C(n)	&>C_{0}(n,p)+P_{3}(n-t_{1,p})+C_{4}^{d}\\
		&=\Big(n^{3}+ \left(-6 k p r-6 p^2 r-3 p r+6 r\right)n^2+c_{1,1} n+c_{1,2}\Big)/(6\Delta)\\
		&=:C_{1,d}(n,p)/(6\Delta).
	\end{aligned}\]
	
	It can be seen that the minimum value of  $ {C^{\prime}_{1,d}}(n,p) $ is greater than $ 0 $ for any $ n $, that is
	\begin{align}\label{Derivation-1}
		&\Delta {C^{\prime}_{1,d}}(-(- 6 k p r - 6 p^2 r-3pr+6r)/3,p)=6ps^{2}(2k+2p+1) \notag\\
		&\quad+r s\left(12 k^2 p +12 k p^2 -6 k p -12 p^2 -6 p +1\right)-\frac{3 r^2}{2}-s^2>0.
	\end{align}
	Thus $ C_{1}(n,p) $ is monotonically increasing, and thereby
	\[{ 6\Delta C(n)>C_{1,d}(t_{1,p},p) } =12 k p^4 r^2 s-96 pr^8 s^2+c_{1,3}>0.\]

	{\noindent \bf  Case 2:} If $ t_{2,p}\leq n < t_{3,p}$, we obtain that
	\[ 	C(n)= C(n,p)+p_{4}(n-t_{1,p})-p_{4}(n-t_{2,p}), \]
	and then,
	\[ \begin{aligned}
		C(n)	&>C_{0}(n,p)+P_{3}(n-t_{1,p})-P_{3}(n-t_{2,p})-J\\
		&=\Big(n^2 (6 k s+6 p s+3 s) + c_{2,1}n +c_{2,2}\Big)/(6\Delta)\\
		&=:C_{2,d}(n,p)/(6\Delta).
	\end{aligned}\]
	
	Following the similar procedure as before, we see that $ C^{\prime}_{2,d}(n,p) $ is positive, since
	\begin{align}\label{Derivation-2}
		&{ C^{\prime}_{2,d}(n,p) \geq	C^{\prime}_{2,d} (t_{2,p},p)} =
		6p (3r -  k r + 2 k^2 r+ 3 s+ 6 k s)\notag \\
		&+12 r s + 24 k r s + 3 s^2 + 12 k s^2 + 12 k^2 s^2 +
		12	p^2s (-r +  k r + 2 s)>0.
	\end{align}
	Thus $ C_{2,d}(n,p) $ is monotonically increasing in this case and thereby,
	\[ \begin{aligned}
		{ 6\Delta C(n) \geq  C_{2,d}(t_{2,p},p)  }=12 k p^4 r^2 s-96 pr^8 s^2+c_{2,3}>0.
	\end{aligned} \]
	
	{\noindent \bf  Case 3:} If $ t_{3,p}\leq n < t_{4,p}$, it follows that
	\[ C(n)= C(n,p)+p_{4}(n-t_{1,p})-p_{4}(n-t_{2,p})-p_{4}(n-t_{3,p}). \]
	Thus,
	\[ \begin{aligned}
		C(n)&>C_{0}(n,p)+P_{3}(n-t_{1,p})-P_{3}(n-t_{2,p})-P_{3}(n-t_{3,p})-J-C_{4}^{u}\\
		&=\big(-n^{3}+ (-3 r + 3 k r + 9 p r + 6 k p r + 6 p^2 r + 6 k s)n^2\\		&\quad\quad+c_{3,1}n+c_{3,2}\big)/(6\Delta)\\
		&=:C_{3,d}(n,p)/(6\Delta).
	\end{aligned}\]
	We can see that the quadratic  function ${C^{\prime}_{3,d}}(n,p)$ is positive in this case since
	\begin{align}
		&C^{\prime}_{3,d}(t_{3,p},p)=p \left(12 k^2 r s+30 k r s-36 k s^2+42 r s-30 s^2\right)+12 k^2 r s-12 k^2 s^2\notag\\
		&+p^2 \left(12 k r s+12 r s-24 s^2\right)+42 k r s-24 k s^2-\frac{21 r^2}{2}+17 r s-8 s^2>0,\label{Derivation-3-1}\\
		&{C^{\prime}_{3,d}}(t_{4,p},p)=p \left(12 k^2 r s+30 k r s-12 k s^2-6 r s-42 s^2\right)+12 k^2 r s\notag\\
		&+p^2 \left(12 k r s+12 r s-24 s^2\right)+18 k r s-12 k s^2-\frac{21 r^2}{2}-19 r s-17 s^2>0.\label{Derivation-3-2}
	\end{align}
	It deduces $ C_{3,d}(n,p) $ is monotonically increasing in this case, and
	\[ 6\Delta C(n)>C_{3,d}(t_{3,p},p) =12 k p^4 r^2 s  - 96p r^8 s^2+c_{3,3}>0.\]
	
	Above all, it can be seen  that for a given $ p\geq 2r^{2}(s/k)^{1/3}  $ when
	\begin{align}\label{def-F}
		n\geq t_{4,p-1}  \geq  \left(2 r^2 \sqrt[3]{s/k}+k\right) \left(4 r^3 \sqrt[3]{s/k}-r+2 s\right)=:L_1(r,s,k),
	\end{align}
	that is the condition \eqref{n-Bound}, $C(n)>0$, and thereby the proof is complete.
	\qed
	
	Note that the boundary condition $ p\geq 2r^{2}(s/k)^{1/3}  $ may be lowered by more precise estimation. In fact, when $ r,s,k $ are given, we can take the  maximum zero of all $ C_{i,d}(t_{j,p},p) $ in the above proof  to be the lower bound for $ p $.
	
	\section{Proof of the non coprime cases}\label{non-coprime}
	In this section, we shall consider Merca's stronger conjecture with the non coprime  parmeters such that $ (s,r-s,r+s,2r-s )\neq 1$.
	More precisely, since $ (r,s)=1 $, it can  be easily seen that if $ 2\mid  s$ then $2\nmid r $ which leads to $ (s,r-s,r+s)=1 $, and if $ 2\nmid  s$ and $2\nmid r $ then $ (s,r-s,2r-s)=1 $.
	
	To prove Theorem \ref{subthm-2|s} and Theorem \ref{subthm-2nmids}, we shall alternatively consider the  partition function with the following generating function
	\begin{align}\label{p40}
		&\sum_{n=0}^{\infty}	p_{3,s}(n)q^{n}=\frac{1}{(1-q^{s})(1-q^{r-s})(1-q^{r+s}) }.
	\end{align}
	Denote $ \Delta_{1}=s(r-s)(r+s) $. Following the ideas given in the proof of Lemma \ref{limit-C(n)}, we derive the following result.
	
	\begin{lem}\label{lem-sub-2|s} For $n\geq 0$, we have
		\begin{align}\label{P3-detail}
			P_{2,s}(n)+C_{3,s}^{d}\leq	p_{3,s}(n)\leq P_{2,s}(n)+C_{3,s}^{u},
		\end{align}
		where
		\begin{align*}
			&P_{2,s}(n)=\frac{1}{2\Delta_{1}}(n^2+2 n r+n s+2 r+s-1),\\
			&  C_{3,s}^{u} =(r^3s \left(2 s -1\right)+2r^2 s\left(1-2 s^2\right)+r \left(-2 s^4+s^3-2 s-1\right)\\
			&\qquad\quad  +2s (s^4-2 s^2+2 s -1)+1) /2\Delta_{1} ,\\
			&  C_{3,s}^{d} =\left(r^3 \left(-2 s^2-s\right)+2 r^2 s^3+r \left(2 s^4+s^3+2 s-1\right)-2 s^5-2 s^2-s+1\right)/2\Delta_{1} .
		\end{align*}
	\end{lem}
	
	\begin{proof}	By  applying Lemma \ref{lem-coprime} and using the residue theorem \eqref{Cauchy}, it follows that
		\[
		p_{3,s}(n)=P_{2,s}(n)+C_{3,s}+Q(n\equiv \Delta_1),
		\]
		where $P_{2,s}(n)$ is as given above, $C_{3,s}$ is a constant relative to $r$ and $s$, and $Q_{s}(n)$ is a function with period of $\Delta_1$. Further employing  Lemma \ref{lem-case-has-2parts}, it follows that
		\begin{equation}\label{p3-1}
			\frac{1}{(1-q^s)(1-q^{r-s})}=\sum_{n=0}^\infty p_2(n) q^n,
		\end{equation}
		where $$p_{2}(n)= \left\lfloor\frac{n}{s(r-s)}\right\rfloor \text{ or }\   p_{2}(n)= \left\lfloor\frac{n}{s(r-s)} \right\rfloor+1. $$
		By convolving \eqref{p3-1} and $\frac{1}{1-q^{r+s}}$, we obtain
		\[
		p_{3,s}(n)=\sum_{m=0}^{\lfloor n/(r+s)\rfloor} p_2(n-m(r+s)).
		\]
		Substituting into the expression for $p_2(n)$, it leads to that
		\begin{align*}
			\sum _{m=0}^{n/(r+s)-1} \left(\frac{n-m(r+s)}{s (r-s)}-1\right) \leq p_{3,s}(n) \leq \sum _{m=0}^{ n/(r+s) } \left(\frac{n-m(r+s)}{s (r-s)}+1\right).
		\end{align*}
		Combining with the relation $C_{3,s}+Q_{ s}(n\equiv \Delta_1)=p_{3,s}(n)-P_{2,s}(n)$ and computing its  extreme values for $0\leq n \leq \Delta_1-1$, we obtain the boundary conditions \eqref{P3-detail} such that $C_{3,s}^{d}\leq C_{3,s}+Q_{ s}(n\equiv \Delta_1) \leq C_{3,s}^{u}$.
	\end{proof}

	Next, following the similar  procedures as in the proof of Theorem \ref{main-thm}, we give the main steps to prove Theorem \ref{subthm-2|s}.
	
	{\noindent{\bf Proof of Theorem  \ref{subthm-2|s}.} } In this proof for simplicity, we use $p_3(n)$ instead of $p_{3,s}(n)$ ($P_{2, s}(n),C_{3, s}^{u} ,C_{3, s}^{d},Q_{ s}(n)$ are likewise) when there is no confusion (namely when $s$ is fixed).  Consider the coefficient of $ q^{n} $ of following series
	\begin{align}\label{main-transform-3}
		\sum_{n=0}^{\infty}	D(n)q^{n}=	\frac{\sum_{j=0}^{\infty}\left(q^{t_{1,j}}-q^{t_{2,j}}-q^{t_{3,j}}+q^{t_{4,j}}\right)}{(1-q^{s})(1-q^{r-s})(1-q^{r+s})}.
	\end{align}
	Similarly, we first have
	\begin{align}\label{D(n)}
		&	D(n)=\sum_{j=0}^\infty   \big( p_{3}(n-t_{1,j}) -p_{3}(n-t_{2,j}) -p_{3}(n-t_{3,j}) +p_{3}(n-t_{4,j}) \big)\notag \\
		&>\sum_{j=0}^{p-1}   \big( P_{2}(n-t_{1,j})- P_{2}(n-t_{2,j})
		-P_{2}(n-t_{3,j} )+P_{2}(n-t_{4,j})-2(C_3^{u}-C_3^{d}) \big)\notag\\
		&=:D_{0}(n,p),
	\end{align}
	and we also denote the partial sum of the first $p$ terms in $D(n)$ by $D(n,p)$.
	
	Then, let us consider the following four cases, where $d_{i,j}  $ also denote some constants relative to $ r,s,p,k $, and one can refer to the appendix for details.
	
	{\noindent \bf  Case 0:} If $ t_{4,p-1}\leq n < t_{1,p}$, from \eqref{D(n)}, we obtain
	\[ \begin{aligned}
		&D_{0}(n,p)>\frac{1}{2\Delta_{1}}(-4 n p s+d_{0,1})=:\frac{D_{0,d}(n,p)}{ {2\Delta_{1}}}>\frac{D_{0,d}(t_{1,p},p)}{ {2\Delta_{1}}}\geq 0.
	\end{aligned} \]

	{\noindent \bf  Case 1:} If $ t_{1,p}\leq n < t_{2,p}$, we see that
	\[D(n)= D(n,p)+p_{3}(n-t_{1,p}). \]
	Then,
	\[ \begin{aligned}
		D(n)	&>D_{0}(n,p)+P_{2}(n-t_{1,p}) +C_{3}^d \\
		&=\big(n^{2}+ n (2 r - 2 p r - 4 k p r - 4 p^2 r + s)+d_{1,1}\big)/(2\Delta_{1})\\
		&=:D_{1,d}(n,p)/(2\Delta_{1})\\
		&\geq \text{min} \Big\{\frac{D_{1,d}(t_{1,p},p)}{2\Delta_{1}} ,\frac{D_{1,d}(-r + p r + 2 k p r + 2 p^2 r - s/2,p)}{2\Delta_{1}}\Big\}\geq 0,
	\end{aligned}\]
	where the last step holds since $ -r + p r + 2 k p r + 2 p^2 r - s/2< t_{2,p}$.
	
	{\noindent \bf  Case 2:} If $ t_{2,p}\leq n < t_{3,p}$, it follows that
	\[ 	D(n)= D(n,p)+p_{3}(n-t_{1,p})-p_{3}(n-t_{2,p}), \]
	and then,
	\[ \begin{aligned}
		D(n)	&>D_{0}(n,p)+P_{2}(n-t_{1,p})-P_{2}(n-t_{2,p})-(C_3^u-C_3^d)\\
		&=\Big(n (2 s + 4 k s + 4 p s)+  d_{2,1}  \Big)/(2\Delta_{1})\\
		&=:D_{2,d}(n,p) /(2\Delta_{1})\geq D_{2,d}(t_{2,p},p) /(2\Delta_{1})\geq 0.
	\end{aligned}\]
	
	{\noindent \bf  Case 3:} If $ t_{3,p}\leq n < t_{4,p}$, we obtain
	\[
	D(n)= D(n,p)+p_{3}(n-t_{1,p})-p_{3}(n-t_{2,p})-p_{3}(n-t_{3,p}).
	\]
	Thereby,
	\[ \begin{aligned}
		D(n)&>D_{0}(n,p)+P_{2}(n-t_{1,p})-P_{2}(n-t_{2,p})-P_{2}(n-t_{3,p})-(2C_3^u-C_3^d)\\
		&=\big(-n^{2}+ n (2 k r + 6 p r + 4 k p r + 4 p^2 r - s + 4 k s)+d_{3,1}\big)/(2\Delta_{1})\\
		&=:D_{3,d}(n,p)/(2\Delta_{1}) \\
		&\geq \text{min} \Big\{\frac{D_{3,d}(t_{3,p},p)}{2\Delta_{1}} , \frac{D_{3,d}(t_{4,p},p)}{2\Delta_{1}} \Big\}\geq 0.
	\end{aligned}\]
	
	Let $ z_0 $ be the maximum zero of $ D_{0,d}(t_{1,p},p),D_{1,d}(t_{1,p},p),D_{1,d}(-r + p r + 2 k p r + 2 p^2 r - s/2,p),D_{2,d}(t_{2,p},p),D_{3,d}(t_{3,p},p) ,D_{3,d}(t_{4,p},p) $ about $p$.
	Finally, we  deduce that when $p\geq z_0 $, the coefficient of $ q^{n} $  is nonnegative for
	\begin{align}\label{N2}
		n\geq t_{4,z_0-1}  \geq(z_0+k) (2 z_0r-r+2 s) :={L_2(r,s,k)}.
	\end{align}\qed
	
	We are going to consider the following partition function when $2\nmid s$,
	\begin{align}\label{En}
		\sum_{n=0}^{\infty}	E(n)q^{n}=	\frac{\sum_{j=0}^{\infty}\left(q^{t_{1,j}}-q^{t_{2,j}}-q^{t_{3,j}}+q^{t_{4,j}}\right)}{(1-q^{s})(1-q^{r-s})(1-q^{2r-s})}.
	\end{align}
	Denote $ \Delta_{2}=s(r-s)(2r-s) $.
	By replacing $s$ by $r-s$ in Lemma \ref{lem-sub-2|s}, we can formally obtain the following comparison results of \eqref{P3-detail}.
	
	\begin{lem} For $n\geq 0$, we have
		\begin{align*}
			P_{2,r-s}(n)+C_{3,,r-s}^{d}\leq	p_{3,r-s}(n)\leq P_{2,r-s}(n)+C_{3,,r-s}^{u},
		\end{align*}
		where
		\begin{align*}
			&P_{2,r-s}(n)=\frac{1}{2\Delta_{2}}(-1 + n^2 + 3 r + n (3 r - s) - s),\\
			&  C_{3,,r-s}^{u} =(r^3 (4 s^2-2 s)+r^2 (-10 s^3+3 s^2+4 s)+r (8 s^4-s^3-6 s^2-2 s-2)\\&\qquad-2 s^5+2 s^3+2 s^2+s+1)/2\Delta_{2} ,\\
			&  C_{3,,r-s}^{d} =(r^3 (-4 s^2-2 s)+r^2 (10 s^3+3 s^2)+r (-8 s^4-s^3+2 s-2)\\&\qquad+2 s^5-2 s^2+s+1)/2\Delta_{2} .
		\end{align*}
	\end{lem}
	
	{\noindent{\bf Proof of Theorem  \ref{subthm-2nmids}.} } For simplicity, we omit the subscript $r-s$ of $p_{3,r-s}(n)$, $P_{2,r-s}(n)$, $C_{3,r-s}^{u}$, $C_{3,r-s}^{d}$, $Q_{r- s}(n)$  when there is no confusion in this proof. Similar to the proof of Theorem \ref{subthm-2|s}, we first have
	\begin{align}\label{E(n)}
		&	E(n)=\sum_{j=0}^\infty   \Big( p_{3}(n-t_{1,j}) -p_{3}(n-t_{2,j}) -p_{3}(n-t_{3,j}) +p_{3}(n-t_{4,j}) \Big)\notag \\
		&>\sum_{j=0}^{p-1}   \Big( P_{2}(n-t_{1,j})- P_{2}(n-t_{2,j})
		-P_{2}(n-t_{3,j} )+P_{2}(n-t_{4,j})-2(C_3^{u}-C_3^{d}) \Big)\notag\\
		&=:E_{0}(n,p).
	\end{align}
	
	Then, let us consider the following four cases and we omit the details.
	
	{\noindent \bf  Case 0:} If $ t_{4,p-1}\leq n < t_{1,p}$,
	\[ \begin{aligned}
		&E_{0}(n,p)\geq\frac{E_{0,d}(t_{1,p},p)}{ {2\Delta_{2}}}.
	\end{aligned} \]

	{\noindent \bf  Case 1:} If $ t_{1,p}\leq n < t_{2,p}$,
	\[ \begin{aligned}
		E(n)  \geq \text{min} \Big \{\frac{E_{1,d}(t_{1,p},p)}{2\Delta_{2}} ,\frac{E_{1,d}(3 r - 2 p r - 4 k p r - 4 p^2 r - s,p)}{2\Delta_{2}}\Big\}.
	\end{aligned}\]

	{\noindent \bf  Case 2:} If $ t_{2,p}\leq n < t_{3,p}$,
	\[ \begin{aligned}
		E(n)	 \geq \frac{E_{2,d}(t_{2,p},p) }{2\Delta_{2}}.
	\end{aligned}\]
	
	{\noindent \bf  Case 3:} If $ t_{3,p}\leq n < t_{4,p}$,
	\[ \begin{aligned}
		E(n) \geq \text{min} \Big\{\frac{E_{3,d}(t_{3,p},p)}{2\Delta_{2}} , \frac{E_{3,d}(t_{4,p},p)}{2\Delta_{ 2}} \Big\}.
	\end{aligned}\]
	By computing the values, we also have that $E(n)\geq 0$ when
	\begin{align}\label{z-0}
		n\geq (z_0+k) (2 z_0r-r+2 s)=:{L_3(r,s,k)},
	\end{align}
	where $z_0$ is taken to be the maximum zero of the polynomials $ E_{0,d}(t_{1,p},p)$, $E_{1,d}(t_{1,p},p)$, $E_{1,d}(3 r - 2 p r - 4 k p r - 4 p^2 r - s,p)$,  $E_{2,d}(t_{2,p},p)$, $E_{3,d}(t_{3,p},p) $, $E_{3,d}(t_{4,p},p) $ as given in the appendix.
	\qed
	
	Now, we are going to investigate the value of $ k $ to guarantee the nonnegativity for the coefficient of $q^n$ for any given  $n\geq 0 $ when $2\mid s$. From the appendix, we can see $D_{0,d}(t_{1,p},p)$, $D_{1,d}(t_{1,p},p)$, $D_{1,d}(-r + p r + 2 k p r + 2 p^2 r - s/2,p)$, $D_{2,d}(t_{2,p},p)$, $D_{3,d}(t_{3,p},p)$, $D_{3,d}(t_{4,p},p)$ have a unified form as
	\begin{align*}
		A(r,s,k)p^{2}+B(r,s,k)p+C(r,s,k),
	\end{align*}
	where $ A(r,s,k), B(r,s,k), C(r,s,k) $ are polynomials in $ r,s$ and $k $. Here we give an example of $ r=9$ and $s=2 $ to estimate the lower bound of  $ k $ such that  \eqref{p40} always has nonnegative coefficient  and thus Conjecture \ref{stronger-ja-2} holds directly for such $k$.
	
	\begin{ex}\label{example-1}
		When $ r=9 $ and $ s=2 $, we have that
		\begin{align*}
			&D_{0,d}(t_{1,p},p)=(  72 k-40) p^2+( 72 k^2 - 20 k -8984) p ,\\
			&D_{1,d}(t_{1,p},p)=(  72 k-56) p^2+ (72 k^2 - 20 k -17780  ) p-3277,\\
			&D_{1,d}(-r + p r + 2 k p r + 2 p^2 r - s/2,p)=( 72 k-40) p^2+ (72 k^2 - 20 k -17860  ) p -3177 ,\\
			&D_{2,d}(t_{2,p},p)=( 72 k-40) p^2 + (  72 k^2 + 12 k-17684) p+ 16 k^2+ 96 k -8832  ,\\
			&D_{3,d}(t_{3,p},p) =(72 k+40) p^2 +( 72 k^2+ 132 k -17604) p +
			56 k^2+ 156 k  -12027 ,\\
			&D_{3,d}(t_{4,p},p)=(72 k+40) p^2  + (72 k^2+ 164 k-17780) p+ 72 k^2 + 92 k -12159 .
		\end{align*}
		Let $\{ z_{i}\} $ be the sets of the   nonnegative  roots of the above polynomials about $p$. Since $n\geq t_{4,p-1}  =(p+k) (2 pr-r+2 s) $, now we put
		\begin{align*}
			2z_{i}r-r+2s\leq 0, \quad\text{ for all }i,
		\end{align*}
		which is equivalent to the conditions such as the first two ones
		\begin{align*}
			&  \frac{-72k^{2}+20k-8984}{72k-40} \leq \frac{r-2s}{2r}=\frac{5}{18},\\
			&\frac{ -72k^{2}+20k+17780+\sqrt{ (72k^{2}-20k-17780)^{2}+4\cdot3277(72k-56)}}{2(72k-56)}\leq \frac{5}{18}.
		\end{align*}
		With the help of computer, we  obtain $ k\geq 19$ from the above inequalities, which is more precise that $k\geq 130$ by Corollary \ref{cor-3}. Thus when $ r=9 $, $ s=2 $ and $ k\geq 19 $, the coefficient of $ q ^{n}$ in \eqref{p40} is nonnegative for $ n\geq 0 $, and thereby Conjecture \ref{stronger-ja-2} holds.
	\end{ex}
	
	Following the similar procedures as above,  when $r=2$ and $s=1$ in \eqref{main-transform-3}, we obtain the following equivalent result with $R=2S$.
	
	\begin{cor}\label{cor-r=2}
		For the positive integers $R=2S$ and $k \geq 1$, the theta series
		\begin{align*}
			(-1)^{k} \frac{\sum_{j=k}^{\infty}(-1)^j q^{R j(j+1) / 2}\left(q^{-Sj}-q^{( j+1) S}\right)}{\left(q^S, q^{R-S}; q^R\right)_{\infty}}
		\end{align*}
		has nonnegative coefficients.
	\end{cor}
	\begin{proof}
		Set $r=2$ and $s=1$ in \eqref{main-transform-3}, we have the corresponding $ D_{0,d}(t_{1,p},p),D_{1,d}(t_{1,p},p),\\
		D_{1,d}(-r + p r + 2 k p r + 2 p^2 r - s/2,p),D_{2,d}(t_{2,p},p),D_{3,d}(t_{3,p},p) ,D_{3,d}(t_{4,p},p) $ as follows
		\begin{align*}
			&  E_{0,d}(t_{1,p},p)=(-40 + 8 k^2) p + 8 k p^2,\\
			&E_{1,d}(t_{1,p},p)=2 + (-40 + 8 k^2) p + 8 k p^2,\\
			&E_{1,d}(-r + p r + 2 k p r + 2 p^2 r - s/2,p)=-(17/4) + (-30 + 8 k^2) p + (-4 + 8 k) p^2,\\
			&E_{2,d}(t_{2,p},p)=-8 + 14 k + 4 k^2 + (-16 + 8 k + 8 k^2) p + 8 k p^2,\\
			&E_{3,d}(t_{3,p},p)=-14 + 14 k + 4 k^2 + (-16 + 8 k + 8 k^2) p + 8 k p^2,\\
			&E_{3,d}(t_{4,p},p)=-32 + 8 k + 8 k^2 + (-40 + 16 k + 8 k^2) p + 8 k p^2.
		\end{align*}
		Let $z_{i}$ for $\leq i \leq 6$ be the zeros of the above expressions and put $z_0=\text{max}\{z_i\}$, where $z_i$'s are only relate to $k$. Recall that $t_{4,p-1}  =(p+k) (2 pr-r+2 s) $ and the coefficients are nonnegative for $n\geq t_{4,z_0-1} $.  Also we have that for $k\geq1$ and all positive $z_i$
		\begin{align*}
			(z_i+k) (2 z_ir-r+2 s) \leq 10000.
		\end{align*}
		Moreover, with the help of computer, we can show that the coefficients of the first  $100 00$ terms are all positive, then the proof is complete.
	\end{proof}
	
	We remark the above case specially, since it shows that  the main result still holds when  $s=r/2$, which combined with the restriction $1\leq s < r/2$  will cover all the cases for $r, s$ under symmetry.  Moreover, the above  progress is also applicable for the denominator with four factors as given in \eqref{main-transform}, but for the sake of brevity, we  choose  a rough boundary condition \eqref{p-condition}  for the maximum zero there.

	\section{The nonmodular infinite products }\label{Nonmodular}
	
	In this section, based on Chern's method, we shall discuss the  asymptotic formula for the coefficient of the nonmodular infinite products $ \Psi_{a, M}(q)+\Psi_{M-a, M}(q) $ as defined in \eqref{def-Phi-aM}. From this formula, we obtain the explicit bound for $G_{a,M}(q)$ on the minor arcs and near the poles.
	
	\subsection{Proof of Theorem  \ref{Phi-aM}}
	
	In this subsection, we prove the expression for $ \Psi_{a, M}(q)+\Psi_{M-a, M}(q) $ as given in Theorem \ref{Phi-aM}.
	
	\noindent{\it Proof of Theorem  \ref{Phi-aM}.}
	We first use Chern's method to derive the main and error terms of $ \Psi _{a, M}(q)+\Psi_{M-a, M}(q)$. Recall that for the given $k_1$, we take
	\[
	q:=e^{-\tau+\frac{2 \pi i h}{k_{1}}},
	\]
	where $1 \leq h \leq k_{1} \leq\lfloor\sqrt{2 \pi X}\rfloor=: N$ with $(h, k_{1})=1$ and $\tau:=X^{-1}+2 \pi i Y$ with $|Y| \leq \frac{1}{k_{1} N}$. Denote
	\begin{align*}
		K:=k_{1}\frac{M}{(k_{1},M)} =k_{1}\frac{M}{M^{\ast}},
	\end{align*}
	where $M^*=(k_1, M)$.
	From the definition \eqref{def-Phi-aM}, one can see that
	\begin{align*}
		\Psi_{a,M}(q)&=\text{log}\Big( \frac{1}{(q^{a+2M };q^{M})_{\infty}}\Big)\\[3pt]
		&=\text{log}\Big( \frac{1}{(q^{a };q^{M})_{\infty}}\Big)-\text{log}\Big( \frac{1}{(1-q^{a })(1-q^{a+M})}\Big) \\
		&=\sum_{\substack{m \geq 1 \\ m \equiv a \bmod M}} \sum_{\ell \geq 1} \frac{q^{\ell m}}{\ell}-\Big(\sum_{\ell \geq 1} \frac{q^{a\ell }}{\ell}+\sum_{\ell \geq 1} \frac{q^{(a+M)\ell }}{\ell}\Big)\\
		&=:\Phi_{a, M}(q)-T_{a}\\
		&=\sum_{\substack{1 \leq \lambda \leq K \\ \lambda \equiv a \bmod M}} \sum_{\substack{1 \leq \mu \leq k_1}} e^{\frac{2 \pi i h \mu \lambda}{k_1}} \sum_{r, t \geq 0} \frac{1}{r k_1+\mu} e^{-(r k_1+\mu)(t K+\lambda) \tau}-T_{a},
	\end{align*}
	where the last equality follows by substituting   $ \ell $ and $ m $ as follows
	\begin{align*}
		&\ell=rk_1+\mu \quad\ \  \text{for}\quad 1\leq \mu \leq  k_1 ,\\
		&m=tK+\lambda \quad \text{for} \quad 1\leq \lambda \leq K,\ \lambda\equiv\ a\ (\bmod M).
	\end{align*}
	Applying the Mellin transform \eqref{Mellin-2} to   $\varphi(s)=\Gamma(s) $ and $f(x)=e^{-x}$ and substituting $x$ by $(rk_1+\mu)(tK+\lambda)\tau$, we have
	\begin{align*}
		\sum_{r,t\geq 0} \frac{e^{-(r k_1+\mu)(t K+\lambda) \tau}}{rk_1+\mu} & =\frac{1}{2 \pi i} \int_{(\frac{3}{2})} \sum_{r,t\geq 0} \frac{\Gamma(s)}{r k_1+\mu} \frac{d s}{(r k_1+\mu)^s(t K+\lambda)^s \tau^s}\\
		&=\frac{1}{2 \pi i} \int_{(\frac{3}{2})} \frac{\Gamma(s)}{\tau^s k_1^{s+1} K^s} \zeta\Big(s, \frac{\lambda}{K}\Big) \zeta\Big(1+s, \frac{\mu}{k_1}\Big) d s.
	\end{align*}
	By employing \eqref{Hurwitz}, \eqref{Gamma-use} and $e^{i\theta}=\cos{\theta}+i\sin{\theta}$, we obtain that
	\begin{align*}
		&\quad\ \Phi_{a, M}(q)=\Psi _{a, M}(q)+T_{a}\\
		&=\frac{1}{4 \pi i k_{1} K} \sum_{\substack{1 \leq \lambda \leq K \\ \lambda \equiv a \bmod M}} \sum_{\substack{1 \leq \mu \leq k_{1} \\ 1 \leq \nu \leq K}} \cos \frac{2 \pi \mu h\lambda}{k_{1}} \cos \frac{2 \pi \nu \lambda}{K} \int_{\substack{(\frac{3}{2})}} \frac{\zeta\big(1+s, \frac{\mu}{k_{1}}\big) \zeta\big(1-s, \frac{\nu}{K}\big)}{z^{ s} \cos \frac{\pi s}{2}} d s\\
		&+\frac{1}{4 \pi i k_{1} K} \sum_{\substack{1 \leq \lambda \leq K \\ \lambda \equiv a \bmod M}} \sum_{\substack{1 \leq \mu \leq k_{1} \\ 1 \leq \nu \leq K}} \cos \frac{2 \pi \mu h\lambda}{k_{1}} \sin \frac{2 \pi \nu \lambda}{K} \int_{\substack{(\frac{3}{2})}} \frac{\zeta\big(1+s, \frac{\mu}{k_{1}}\big) \zeta\big(1-s, \frac{\nu}{K}\big)}{z^{ s} \sin \frac{\pi s}{2}} d s\\
		&+\frac{1}{4 \pi   k_{1} K} \sum_{\substack{1 \leq \lambda \leq K \\ \lambda \equiv a \bmod M}} \sum_{\substack{1 \leq \mu \leq k_{1} \\ 1 \leq \nu \leq K}} \sin \frac{2 \pi \mu h\lambda}{k_{1}} \sin \frac{2 \pi \nu \lambda}{K} \int_{\substack{(\frac{3}{2})}} \frac{\zeta\big(1+s, \frac{\mu}{k_{1}}\big) \zeta\big(1-s, \frac{\nu}{K}\big)}{z^{ s} \sin \frac{\pi s}{2}} d s\\
		&+\frac{1}{4 \pi  k_{1} K} \sum_{\substack{1 \leq \lambda \leq K \\ \lambda \equiv a \bmod M}} \sum_{\substack{1 \leq \mu \leq k_{1} \\ 1 \leq \nu \leq K}} \sin \frac{2 \pi \mu h\lambda}{k_{1}} \cos \frac{2 \pi \nu \lambda}{K} \int_{\substack{(\frac{3}{2})}} \frac{\zeta\big(1+s, \frac{\mu}{k_{1}}\big) \zeta\big(1-s, \frac{\nu}{K}\big)}{z^{ s} \cos \frac{\pi s}{2}} d s\\
		&=:\gamma_{1,-s}-\gamma_{2,-s}-\gamma_{3,-s}+\gamma_{4,-s},
	\end{align*}
	where $ z:=\tau k_{1}/(2\pi ) $.
	As given in \cite[eq. (2.19)]{Chern},
	replacing $ s $ by $  -s $,
	$ h\lambda $ by $ -\rho $, reversing the direction of the integration path and shifting the path back to  $(\frac{3}{2})$ by using the residue theorem \eqref{Cauchy}, it can be seen that
	\begin{align*}
		\Phi_{a, M}(q)&=\Psi _{a, M}(q)+T_{a}\\	&=\gamma_{1,s}+\gamma_{2,s}+\gamma_{3,s}+\gamma_{4,s}-2\pi i\left(R_{1}+R_{2}+R_{3}+R_{4} \right)\\
		&=\gamma_{1}+\gamma_{2}+\gamma_{3}+\gamma_{4}-2\pi i\left(R_{1}+R_{2}+R_{3}+R_{4} \right),
	\end{align*}
	and $ R_{j} $ is the sum of the residues of the corresponding integrand  inside the strip $ -\frac{3}{2}<\Re(s)<\frac{3}{2} $. Thus
	\begin{align*}
		\Psi _{a, M}(q)		=\mathcal{\gamma}_{1}+\mathcal{\gamma}_{2}+\mathcal{\gamma}_{3 }+\mathcal{\gamma}_{4}-2\pi i\left(R_{1}+R_{2}+R_{3}+R_{4} \right)-T_{a}.
	\end{align*}
	We will give a brief proof which will show that the main terms for $\Psi_{a,M}(q)$ come form $ R_{1}, R_{4} $ while the others are the error terms.
	
	First, notice that
	\begin{align*}
		\bigg|\Re\bigg(\sum_{\ell\geq 1} \frac{q^{a\ell}}{\ell}\bigg)\bigg|&= \bigg|\Re\bigg(\sum_{\ell \geq 1} \frac{{e}^{-(X^{-1}+2\pi i (\frac{h}{k_{1}}-Y))a\ell}}{\ell}\bigg)\bigg|
		\\
		&\leq  \sum_{\ell \geq 1} \frac{\left|\Re\big({e}^{-(X^{-1}+2\pi i (\frac{h}{k_{1}}-Y))a\ell}\big)\right|}{\ell}
		\\
		&= \sum_{\ell \geq 1} \frac{e^{-\frac{a\ell}{X}}}{\ell}\leq \int_{\ell=2}^{\frac{X}{a}}\frac{e^{-\frac{a\ell}{X}}}{\ell}d\ell+\int_{\ell=\frac{X}{a}}^{\infty}\frac{e^{-\frac{a\ell}{X}}}{\ell}d\ell+1\\
		&\leq\int_{\ell=2}^{\frac{X}{a}}\frac{1}{\ell} d\ell+\int_{t=1}^{\infty}\frac{e^{-t}}{t}dt+1=\text{log} \frac{X}{a}+\frac{Gc}{e}+1-\text{log}2,
	\end{align*}
	where $ Gc  \leq 0.5963$ is the Euler-Gompertz constant \cite[A073003]{OEIS}, which is defined as
	\[
	\int_{1}^\infty \frac{e^{-t+1}}{t} dt = Gc.
	\]
	To deal with the integrals $\gamma_{i}$,  consider the following auxiliary function
	\begin{align*}
		\Omega_{a, M}(q):=\log \Bigg(\prod_{\substack{m \geq 1 \\ m \equiv-h a \bmod M^*}} \frac{1}{1-e^{\frac{2 \pi i \alpha a}{M}} q^m}\Bigg).
	\end{align*}
	Then for $h h^{\prime} \equiv-1 \ (\bmod\  k_1)  $ and $ q^{\ast}=e^{(\frac{2 \pi i \beta h^{\prime}}{k_{1}}-\frac{2 \pi}{K z} ) }$, \cite[pages 10--15]{Chern} gives the following results
	\begin{align}\label{Omega}
		\left|\Re\left(\Omega_{a, M}\left(q^*\right)\right)\right| \leq \frac{e^{-0.28 \pi^2 \frac{(k_1, M)}{M}}}{\big(1-e^{-0.28 \pi^2 \frac{(k_1, M)}{M}}\big)^2},
	\end{align}
	and
	\begin{align*}
		\Omega_{a, M}(q^{\ast})=\mathcal{\gamma}_{1}+\mathcal{\omega}_{2}+\mathcal{\gamma  }_{3}+\mathcal{\omega}_{4},\quad \Omega_{a, M}(q^{\ast})+\Omega_{M-a, M}(q^{\ast})=\mathcal{\gamma}_{1}+\mathcal{\gamma}_{3},
	\end{align*}
	where
	\begin{align*}
		\omega_{2}&=-\frac{1}{4 \pi i k_{1} K} \sum_{\substack{1 \leq \lambda \leq K \\ \lambda \equiv a \bmod M}} \sum_{\substack{1 \leq \mu \leq k_{1} \\ 1 \leq \nu \leq K}} \sin \frac{2 \pi \mu \rho}{k_{1}} \cos \frac{2 \pi \nu \lambda}{K} \int_{\substack{(\frac{3}{2})}} \frac{\zeta\big(1-s, \frac{\mu}{k_{1}}\big) \zeta\big(1+s, \frac{\nu}{K}\big)}{z^{-s} \sin \frac{\pi s}{2}} d s,\\
		\omega_{4}&=-\frac{1}{4 \pi   k_{1} K} \sum_{\substack{1 \leq \lambda \leq K \\ \lambda \equiv a \bmod M}} \sum_{\substack{1 \leq \mu \leq k_{1} \\ 1 \leq \nu \leq K}} \cos \frac{2 \pi \mu \rho}{k_{1}} \sin \frac{2 \pi \nu \lambda}{K} \int_{\substack{(\frac{3}{2})}} \frac{\zeta\big(1-s, \frac{\mu}{k_{1}}\big) \zeta\big(1+s, \frac{\nu}{K}\big)}{z^{-s} \cos \frac{\pi s}{2}} d s.
	\end{align*}
	Furthermore, for $ i=2,4 $,
	\begin{align}\label{gamma-omega}
		\text{max}\{|\mathcal{\gamma}_{i}+\mathcal{\omega}_{i}|,|\mathcal{\gamma}_{i}-\mathcal{\omega}_{i}|\}:=|\mathcal{\gamma}_{i}\pm\mathcal{\omega}_{i}|\leq  \frac{6.51M^{\frac{3}{2}}}{(k_{1}, M)^{\frac{5}{2}}} X^{\frac{1}{2}}.
	\end{align}
	Then by combining \eqref{Omega} and \eqref{gamma-omega}, we can see that
	\begin{align*}   	&\quad\left|\Re\left(\mathcal{\gamma}_{1}+\mathcal{\gamma}_{2}+\mathcal{\gamma  }_{3}+\mathcal{\gamma}_{4}\right)\right| \leq \left|\Re (\mathcal{\gamma}_{1}+\mathcal{\gamma}_{3})\right|+\left|\Re(\mathcal{\gamma  }_{2}+\mathcal{\gamma}_{4} )\right|\\
		&\leq \left|\Re (\mathcal{\gamma}_{1}+\mathcal{\gamma}_{3})\right|+\left|\Re(\mathcal{\gamma  }_{2}+\mathcal{\gamma}_{4} )\right|+\left|  (\mathcal{\gamma}_{2}\pm\mathcal{\omega}_{2})\right|+\left| (\mathcal{\gamma  }_{4}\pm\mathcal{\omega}_{4} )\right|\\
		&\leq \left|\Re\left(\Omega_{a, M}\left(q^*\right)\right)\right| + 2
		\left|\Re(\mathcal{\gamma  }_{1}+\mathcal{\gamma}_{3} )\right|+\left|  (\mathcal{\gamma}_{2}\pm\mathcal{\omega}_{2})\right|+\left| (\mathcal{\gamma  }_{4}\pm\mathcal{\omega}_{4} )\right|\\
		&\leq 2\left|\Re\left(\Omega_{a, M}\left(q^*\right)\right)\right| +\left|\Re\left(\Omega_{M-a, M}\left(q^*\right)\right)\right| +\left|  (\mathcal{\gamma}_{2}\pm\mathcal{\omega}_{2})\right|+\left| (\mathcal{\gamma  }_{4}\pm\mathcal{\omega}_{4} )\right|\\
		&\leq \frac{3 e^{-0.28 \pi^2 \frac{(k_{1}, M)}{M}}}{\big(1-e^{-0.28 \pi^2 \frac{(k_{1}, M)}{M}}\big)^2}+13.02 \frac{M^{\frac{3}{2}}}{(k_{1}, M)^{\frac{5}{2}}} X^{\frac{1}{2}}.
	\end{align*}
	Based on the above results, we deduce that
	\begin{align}		&\left|\Re\left(\mathcal{\gamma}_{1}+\mathcal{\gamma}_{2}+\mathcal{\gamma}_{3}+\mathcal{\gamma}_{4}-T_{a}\right)\right|
		\leq \left|\Re\left(\mathcal{\gamma}_{1}+\mathcal{\gamma}_{2}+\mathcal{\gamma}_{3}+\mathcal{\gamma}_{4}\right)\right|+\left|\Re(T_{a})\right|\notag \\[3pt]
		&\leq \frac{3e^{-0.28\pi^{2} \frac{(k_{1},M)}{M}}}{(1-e^{-0.28\pi^{2} \frac{(k_{1},M)}{M}})^{2}}+13.02\frac{M^{\frac{3}{2}}}{(k_{1},M)^{\frac{5}{2}}}X^{\frac{1}{2}}+\text{log} \frac{X^{2}}{a(M+a)}+\frac{2Gc}{e}+2\text{log} \frac{e}{2}\notag \\[3pt]
		&\leq \frac{3e^{-0.28\pi^{2} \frac{1}{M}}}{(1-e^{-0.28\pi^{2} \frac{1}{M}})^{2}}+13.02 {M^{\frac{3}{2}}} X^{\frac{1}{2}}+\text{log} \frac{X^{2}}{a(M+a)}+\frac{2Gc}{e}+2\text{log} \frac{e}{2}.\label{ReTa}
	\end{align}
	
	Next, we shall give the accurate value or the estimation of each $ R_{i} $ for $ 1\leq i\leq 4 $.  By the definition of $\mathcal{\gamma}_{1}$, the  residues of the integrand in $ \mathcal{\gamma}_{1}$ are given as follows
	\begin{align*}
		\mathcal{R}_{1} & :=\sum_{|\Re(s)|<\frac{3}{2}} \underset{s}{\operatorname{Res}}  \frac{\zeta\big(1-s, \frac{\mu}{  k_1}\big) \zeta\big(1+s, \frac{\nu}{K}\big)}{z^{-s} \cos \frac{\pi s}{2}}=\underset{s=-1}{\mathrm{Res} } +\underset{s=1}{\mathrm{Res} } +\underset{s=0}{\mathrm{Res} } \\
		& =\frac{2 \zeta\big(2, \frac{\mu}{k_1}\big) \zeta\big(0, \frac{\nu}{K}\big)}{\pi z}-\frac{2 z \zeta\big(0, \frac{\mu}{k_1}\big) \zeta\big(2, \frac{\nu}{K}\big)}{\pi}-\log z\\		
		&\quad +\Big(-\frac{\Gamma^{\prime}}{\Gamma}\Big(\frac{\mu}{k_1}\Big)+\frac{\Gamma^{\prime}}{\Gamma}\Big(\frac{\nu}{K}\Big)\Big)=:\mathcal{R}_{11}+\mathcal{R}_{12}+\mathcal{R}_{13}+\mathcal{R}_{14}.
	\end{align*}
	Thus
	\begin{align*}
		R_{1}=\frac{1}{4 \pi i k K} \sum_{\substack{1 \leq \lambda \leq K \\ \lambda \equiv a \bmod M}} \sum_{\substack{1 \leq \mu \leq k_1 \\ 1 \leq \nu \leq K}} \cos \frac{2 \pi \mu \rho}{k_1} \cos \frac{2 \pi \nu \lambda}{K} \cdot \mathcal{R}_{1}.
	\end{align*}
	Here for brevity, we only give two parts ${R}_{11}$ and $ {R}_{12} $ of $ R_{1} $ in details as illustration.
	Furthermore, with \eqref{cos-zata-0} and \eqref{cos-zata-2}, it can be deduced that
	\begin{align*}
		R_{11}&=\frac{1}{4 \pi i k K} \sum_{\substack{1 \leq \lambda \leq K \\ \lambda \equiv a \bmod M}} \sum_{\substack{1 \leq \mu \leq k_1 \\ 1 \leq \nu \leq K}} \cos \frac{2 \pi \mu \rho}{k_1} \cos \frac{2 \pi \nu \lambda}{K} \cdot \mathcal{R}_{11}\\
		&=\frac{1}{z} \frac{1}{2 i \pi^2 k_1 K} \sum_{\substack{1 \leq \lambda \leq K \\ \lambda \equiv a \bmod M}} \sum_{1 \leq \mu \leq k_1} \cos \frac{2 \pi \mu \rho}{k_1} \zeta\Big(2, \frac{\mu}{k_1}\Big) \sum_{1 \leq \nu \leq K} \cos \frac{2 \pi \nu \lambda}{K} \zeta\Big(0, \frac{\nu}{K}\Big)\\
		& =\frac{1}{z} \frac{1}{2 i \pi^2 k_1 K} \sum_{\substack{1 \leq \rho \leq k_1 \\
				\rho \equiv b \bmod M^*}} \frac{\pi^2}{6}\left(6 \rho^2-6 k \rho+k_1^2\right) \cdot\Big(-\frac{1}{2}\Big) \\
		& =\frac{1}{z} \frac{1}{2 i \pi^2 k_1 K} \sum_{j=0}^{k_1/M^{\ast}-1} \frac{\pi^2}{6}\left(6 (b+jM^{\ast})^2-6 k_1 (b+jM^{\ast})+k_1^2\right) \cdot\Big(-\frac{1}{2}\Big) \\
		& =-\frac{2 \pi}{\tau k_1} \frac{1}{24 i k_1 M}\left(6 b^2-6 b M^*+(M^*)^2\right),
	\end{align*}
	where the summation on $\lambda$ is changed to be the summation on $\rho$ by noting that  $\rho$ is the unique integer between 1 and $k_1$ such that $\rho \equiv -h\lambda\  (\bmod \ k_1)$ and $b$ is the unique integer between $1$ and $M^*$ such that $b\equiv -ha\ (\bmod\  M^*)$.
	
	For $ \mathcal{R}_{12} $, by \eqref{cos-zata-0} and \eqref{cos-zata-2}, it can be seen that
	\begin{align*}
		R_{12}	&= -\frac{z}{2 i \pi^2 k_1 K} \sum_{\substack{1 \leq \lambda \leq K \\ \lambda \equiv a \bmod M}} \sum_{1 \leq \mu \leq k_1} \cos \frac{2 \pi \mu \rho}{k_1} \zeta\Big(0, \frac{\mu}{k_1}\Big) \sum_{1 \leq \nu \leq K} \cos \frac{2 \pi \nu \lambda}{K} \zeta\Big(2, \frac{\nu}{K}\Big)\\
		&=-\frac{\tau}{4 i \pi^3   K} \sum_{\substack{1 \leq \lambda\leq K \\
				\lambda \equiv a \bmod M }} \frac{\pi^2}{6}\left(6 \lambda^2-6 {  k_1} \lambda+ {  k_1}^2\right)  {  \Big(-\frac{1}{2}\Big)}\\
		&=  \frac{\tau}{ {   48 i \pi  M}}\left(6 a^2-6 a M+M^2\right).
	\end{align*}
	Thus,
	\begin{align*}
		|\Re (-2\pi i R_{12} )|=|\Re ( \tau) |   \frac{| 6 a^2-6 a M+ M^2 |}{24M}  \leq \frac{1}{X}\cdot \frac{M^2}{24M}\leq \frac{M}{24X}.
	\end{align*}
	
	Similarly as above, for all the $ R_{i} $, we write it as a linear combination of $ R_{ij} $  where
	\begin{align*}
		R_{i}=\sum_{j=1}^{i_{0}}R_{ij}\quad \text{for } 1 \leq i \leq 4.
	\end{align*}
	Let us consider the following three cases to determine the main and the error terms:
	
	{Case 1}: $ a \not\equiv0 \  (\bmod\  M) $ and $ b \not\equiv0  \ (\bmod\  M^{\ast}) $.
	
	{Case 2}: $ a \not\equiv0 \   (\bmod\  M) $ and $ b \equiv0   \ (\bmod\  M^{\ast}) $.
	
	{Case 3}: $ a  \equiv0 \   (\bmod\  M) $ and $ b  \equiv0  \ (\bmod\  M^{\ast}) $.	
	
	Note that $b \equiv -ha \   (\bmod\  M^*)$, thus if $ a  \equiv0 \   (\bmod\  M) $,  then  $b \equiv0\ (\bmod\ M^{\ast})$, and it holds if and only if $M^{\ast}=1 $. Furthermore, we also can ignore Case 3 by noting that $ 0<a<M/2 $.
	
	Following the similar procedures as deducing $R_{11}$ and $R_{12}$, one can directly deduce   $ R_{ij} $  as follows, see \cite[Section 4]{Chern} for details.
	For $ R_{1j} $, we have that
	\begin{align}
		&-2\pi i R_{11}=\frac{1}{\tau}\frac{\pi^{2}}{6k_{1}^{2}M}(6b^{2}-6bM^{\ast}+{M^{\ast}}^{2}), \label{R11}\\[3pt]
		&\left|\Re (-2\pi i R_{12})\right| \leq \frac{M}{{24}}X^{-1},\label{R12}\\[3pt]
		&\left|\Re (-2\pi i R_{13})\right| =0,\label{R13}\\[3pt]
		&\left|\Re (-2\pi i R_{14})\right| \leq  \left\{
		\begin{array}{ll}
			0& \text{ Case 1}, \\[3pt]
			\text{log}\sqrt{2}& \text{ Case 2} .
		\end{array}
		\right.\label{R14}
	\end{align}
	For $ R_{2j} $, we have
	\begin{align}
		&\left|\Re (-2\pi i R_{21})\right| \leq \frac{\text{log}2}{2}\sqrt{2\pi X} (\text{log} \sqrt{2\pi X}+\gamma)+\frac{1}{2}(\text{log} \sqrt{2\pi X}+\gamma),\notag\\
		&\qquad \qquad\qquad \ \leq 0.44X^{\frac{1}{2}}\text{log}X+1.3X^{\frac{1}{2}}+0.25\text{log}X+0.75,\label{R21}\\
		&\left|\Re (-2\pi i R_{22})\right| \leq \frac{1}{2} \text{log}(2\pi X)\leq\frac{1}{2}\text{log}X+0.92,\label{R22}\\
		&\left|\Re (-2\pi i R_{23})\right|\notag\\
		& \leq  \left\{
		\begin{array}{ll}
			0,& \text{ Case 1}, \\[3pt]
			\big(\gamma+2\text{log}(2\pi)+\text{log}\sqrt{2\pi X}+\text{log}M+M-2\text{log}\Gamma (M)   \big)/2,& \text{ Case 2} .
		\end{array}
		\right.\notag\\
		&\leq \frac{1}{4}\text{log}X+\frac{M}{2}+\frac{\text{log}M}{2}+\text{log}\Gamma(\frac{1}{M})+2.59.\label{R23}
	\end{align}
	Here $\gamma$ is the Euler-Mascheroni constant,
	\[
	\gamma=\lim_{n\to \infty}\big( 1+\frac{1}{2}+\cdots+\frac{1}{n} -\log n\big) \leq 0.5773,
	\]
	see, for example \cite{Havil}.
	For $ R_{3} $, we have
	\begin{align}
		\left|\Re (-2\pi i R_{3})\right|=0. \label{R3}
	\end{align}
	For $ R_{4} $, denote
	$$
	b^*:= \begin{cases}(k_{1}, M)-b, & \text { if } b \neq(k_{1}, M), \\ (k_{1}, M), & \text { if } b=(k_{1}, M),\end{cases}
	$$ then we have
	\begin{align}
		& -2\pi i R_{41}= \left\{
		\begin{array}{ll}
			0,& \text{ if } b=M^{\ast}, \\[3pt]
			-\zeta^{\prime}\left(-1, \frac{b}{(k_{1}, M)}\right)+\zeta^{\prime}\Big(-1, \frac{b^*}{(k_{1}, M)}\Big),& \text{ if }b\neq M^{\ast} .
		\end{array}
		\right.\label{R41}\\[3pt]
		&\left|\Re (-2\pi i R_{42})\right| \leq \frac{1}{6}\frac{M}{M^{\ast}} (\text{log}\sqrt{2\pi X}+\gamma)\leq \frac{1}{12}\text{log}X+0.25\frac{M}{M^{\ast}}.\label{R42}
	\end{align}
	
	The main terms come from $ R_{11} $ and $ R_{41} $. To avoid confusion, let $ -2\pi i R_{11}(a) $,  $ -2\pi i R_{41}(a) $ denote the main terms of $ \Psi_{a, M}  $ and the corresponding $b(a)  $ satisfying $b(a) \equiv-h a\ (\bmod\ M^*)$.
	It can be seen  that
	\begin{align*}
		b(M-a) \equiv -h(M-a) \equiv ha\  (\bmod M^*),
	\end{align*}
	then $ b(M-a)=M^{\ast }-b(a) =b^{\ast}(a)$ and $ b(a)=b^{\ast}(M-a)$. It leads to that the main term of $ \Psi _{a, M}(q)+\Psi _{M-a, M}(q) $ which is given as follows
	\begin{align*}
		& -2\pi i (R_{11}(a) + R_{41}(a)+ R_{11}(M-a) + R_{41}(M-a))\\[3pt]
		&=\frac{1}{\tau}\frac{\pi^{2}}{6k_{1}^{2}M}(6b^{2}-6bM^{\ast}+{M^{\ast}}^{2})+
		\frac{1}{\tau}\frac{\pi^{2}}{6k_{1}^{2}M}(6(M^{\ast}-b)^{2}-6(M^{\ast}-b)M^{\ast}+{M^{\ast}}^{2})\\[3pt]
		&-\zeta^{\prime}\Big(-1, \frac{b}{(k_{1}, M)}\Big)+\zeta^{\prime}\Big(-1, \frac{b^*}{(k_{1}, M)}\Big)-\zeta^{\prime}\Big(-1, \frac{b^*}{(k_{1}, M)}\Big)+\zeta^{\prime}\Big(-1, \frac{b}{(k_{1}, M)}\Big)\\[3pt]
		&=\frac{1}{\tau}\frac{(k_{1},M)^{2}\pi^{2}}{k_{1}^{2}M}\Big(\frac{2b(b-(k_{1},M))}{(k_{1},M)^{2}}+\frac{1}{3}\Big).
	\end{align*}
	Combining \eqref{ReTa}--\eqref{R42} together, we complete the proof.
	\qed
	
	\subsection{The bound for $G_{a,M}(q)$ on the minor arcs}
	
	We next aim to discuss the values of $G_{a,M}(q)$ near the minor arcs, i.e. when $q$ is away from $\pm1$. By using the above asymptotic results for $\Psi _{a, M}(q) +\Psi _{M-a, M}(q)$, we obtain the following bound for $G_{a,M}(q)$.
	
	\begin{thm}\label{lem-Gq-bound}
		Let $ X\geq 540M^{5} $. Then
		\begin{align}\label{Gq-bound}
			|G_{a,M}(q)|\leq \mathrm{exp}{\Big( \frac{\pi^{2}X}{2\sqrt{3}M}\Big)},
		\end{align}
		for  $ q=\mathrm{exp}(-\tau +{2\pi i h}/{k_{1}}) $, where $ \tau= X^{-1}+2\pi iY  $ and $ |Y|\geq  1/{2\pi X} $.
	\end{thm}	
	
	\begin{proof}
		Let  $\mathcal{M}_{G}(b)$ denote the main term of $ \Psi _{a, M}(q) +\Psi _{M-a, M}(q)$, then by Theorem \ref{Phi-aM}, we have that
		\begin{align*}
			\mathcal{M}_{G}(b)&=  \frac{\pi^2}{\tau} \frac{(k_{1}, M)^2}{k_{1}^2 M}\Big(\frac{2b(b-(k_{1},M))}{(k_{1}, M)^2 }+\frac{1}{3}\Big),\\[3pt]
			2 E_{1}&=\log G_{a,M}(q)-\mathcal{M}_{G}(b).
		\end{align*}
		
		When $  X\geq 540M^{5} $, by applying \eqref{ReTa}--\eqref{R42}, we have
		\begin{align}\label{2E1}
			2|\Re ( E_{1})|&\leq26.04 M^{1.5} X^{0.5}+1.5 M+\frac{6 e^{2.77/M}}{\left(e^{2.77/M}-1\right)^2}+0.5 \log M+2.6 X^{0.5}\notag \\
			&+\left(0.88 X^{0.5}+6.17\right) \log X+10.89\leq\frac{\pi^{2}X}{2\sqrt{3}M}-\frac{\pi^{2}X}{6M}.
		\end{align}
		Since	    $ |Y|\geq \frac{1}{2\pi X} $, it leads to that
		\begin{align}\label{re-tau}
			\Re(\tau^{-1}) \leq \frac{X}{2},
		\end{align}
		where
		\begin{align*}
			{\tau}^{-1}=\frac{X^{-1}}{X^{-2}+4\pi^{2}Y^{2}}-i\frac{2\pi Y}{X^{-2}+\pi^{2}Y^{2}}.
		\end{align*}
		Noting that $1\leq b\leq(k_{1},M)$ and $(k_1, M)|k_1$, it follows that
		\begin{align}\label{(k_{1},M)}
			\Re(\mathcal{M}_{G}(b))&\leq \Re(\mathcal{M}_{G}((k_{1},M) ))=  \frac{\pi^{2}(k_{1},M)^{2}X }{3k_{1}^{2}M }\Re(\tau^{-1}) \leq   \frac{\pi^{2} X }{6M }.
		\end{align}
		Then combining \eqref{2E1}, \eqref{(k_{1},M)} with the relation
		\begin{align*}
			\log |G_{a,M}(q)|= \Re (\log G_{a,M}(q))\leq \Re(\mathcal{M}_{G})+2|\Re ( E_{1})|,
		\end{align*}
		we complete the proof.
	\end{proof}
	
	\subsection{Approximations of $G_{a,M}(q)$ near the poles}
	In this subsection, we shall discuss the approximations of $G_{a,M}(q)$ near the dominant poles. It is easy to see that $G_{a,M}(q)$ has dominant poles at $q = \pm 1$. Let $q=e^{-\tau+2 \pi i h/k_1}$,  then the approximations of $ \log G(e^{-\tau}) $ and $ \log G(-e^{-\tau})  $ will be given.  Moreover, noting that $(a,M)=1$, for	$\log G(-e^{-\tau}) $, the discussion will be divided into the following three cases according to the parities of $a$ and $M$:
	
	{ {(Case 1)}:} $ a \equiv1 \ (\bmod\ 2) $ and $ M \equiv0  \ (\bmod\ 2)$.
	
	{{(Case 2)}:} $ a \equiv0  \ (\bmod\ 2) $ and $ M \equiv1  \ (\bmod\ 2) $.
	
	{{(Case 3)}:} $ a \equiv1 \ (\bmod\ 2) $ and $ M \equiv1 \ (\bmod\ 2) $.
	
	\begin{thm}\label{approximation-near-pole}
		If $ \tau=X^{-1}+2\pi iY $ with $ |Y|\leq \frac{1}{2\pi X} $, then we have
		\begin{align}\label{caseall-+}
			&	\log G(e^{-\tau})=\frac{\pi^{2}}{3M\tau}+4\log \tau -\log 2\sin \frac{a\pi}{M} -\log \Delta_{M}+E_{+},
			\\[3pt]
			&\log G(-e^{-\tau})=\left\{
			\begin{array}{ll}
				\dps \frac{\pi^{2}}{6M\tau}-4\log 2+E_{1,-}, & \hbox{\rm (Case 1),} \\[3pt]
				\dps \frac{\pi^{2}}{4M\tau}+\log \frac{2}{\sin\frac{a\pi}{2M}}+\log a(2M-a)+E_{2,-}, & \hbox{\rm (Case 2),} \\[3pt]
				\dps \frac{\pi^{2}}{4M\tau}+\log \frac{2}{\sin\frac{a\pi}{2M}}+\log (M-a)(M+a)+E_{3,-}, & \hbox{\rm (Case 3),}  \label{Ge-tau}
			\end{array}
			\right.
		\end{align}
		where
		$ E_{+} $ and $ E_{i,-} $ for $ i=1,2,3 $  satisfy that
		\begin{align*}
			&|E_{+}|\leq7.02M^{\frac{3}{4}}X^{-\frac{3}{4}},\quad\  |E_{i,-}|\leq23.71M^{\frac{3}{4}}X^{-\frac{3}{4}}.
		\end{align*}
	\end{thm}
	
	\begin{proof}
		By using the definition of $G_{a,M}(q)$  \eqref{def-Gq} and applying the Mellin transform \eqref{Mellin-2}, we obtain that
		\begin{align*}
			&\log G(e^{-\tau})=\sum_{m\geq 2} \Big(\log  \frac{1}{1-e^{-(mM+a)\tau}}+\log \frac{1}{1-e^{-(mM+M-a)\tau}}\Big)\\
			&=\sum_{m\geq 2}\sum_{\ell\geq 1}\Big(\frac{e^{-(mM+a)\ell\tau }}{\ell}+\frac{e^{-(mM+M-a)\ell\tau }}{\ell}\Big)\\
			&=\frac{1}{2\pi i} \int_{(\frac{3}{2})} \frac{1}{\tau^{s}}\Gamma(s) \sum_{m\geq 0}\sum_{\ell\geq 1}\ell^{-s-1} \Big(\frac{1}{(mM+a)^{s}} + \frac{1}{(mM+M-a)^{s}}  \Big)ds\\
			&\quad -\frac{1}{2\pi i} \int_{(\frac{3}{2})} \frac{1}{\tau^{s}}\Gamma(s) \sum_{m=0}^{1}\sum_{\ell\geq 1}\ell^{-s-1} \Big(\frac{1}{(mM+a)^{s}} + \frac{1}{(mM+M-a)^{s}}  \Big)ds\\
			&=\frac{1}{2\pi i} \int_{(\frac{3}{2})} \frac{1}{(M\tau)^{s}}\Gamma(s) \zeta(s+1)\Big(\zeta\big(s,\frac{a}{M}\Big)+\zeta\big(s,\frac{M-a}{M}\big)\Big)ds\\
			&\quad -\frac{1}{2\pi i} \int_{(\frac{3}{2})} \frac{1}{\tau ^{s}}\Gamma(s) \zeta(s+1) \Big(\frac{1}{a^{s}}+ \frac{1}{(M-a)^{s}}+\frac{1}{(M+a)^{s}}+\frac{1}{(2M-a)^{s}}\Big)ds.
		\end{align*}
		Let
		\begin{align*}
			&G_{1}:= \frac{1}{(M\tau)^{s}}\Gamma(s) \zeta(s+1)\Big(\zeta\big(s,\frac{a}{M}\Big)+\zeta\Big(s,\frac{M-a}{M}\big)\Big),\\
			&G_{2}:= \frac{1}{\tau ^{s}}\Gamma(s) \zeta(s+1) \Big(\frac{1}{a^{s}}+ \frac{1}{(M-a)^{s}}+\frac{1}{(M+a)^{s}}+\frac{1}{(2M-a)^{s}}\Big).
		\end{align*}
		Then by changing the path of integration to $ \left(-\frac{3}{4}\right) $ and considering the residues in strip $ \left(-\frac{3}{4}\right)<\Re(s)< \left(\frac{3}{2}\right)$, it follows that
		\begin{align*}
			\log G(e^{-\tau})=\sum_{-\frac{3}{4}<\Re (s)<\frac{3}{2}}\big(\underset{s}{\mathrm{Res}}G_{1}+\underset{s}{\mathrm{Res}}G_{2} \big)+\frac{1}{2\pi i}\int_{(-\frac{3}{4})} \left(G_{1} +G_{2} \right)ds.
		\end{align*}
		For $ G_{1} $, since $ s=1 $ is the singularity of order $ 1 $ and $ s=0 $ is the singularity of order $ 2 $, then combined with \eqref{special-1}-\eqref{special-3}, we have
		\begin{align*}
			\underset{s=1}{\mathrm{Res}\,}G_{1} =\frac{1}{M\tau} \Gamma(1)\zeta(2)(1+1)=\frac{\pi^{2}}{3M\tau},
		\end{align*}
		and
		\begin{align*}
			\underset{s=0}{\mathrm{Res}}G_{1} &=- \Big(\frac{1}{2}-\frac{a}{M}+\frac{1}{2}-\frac{M-a}{M}\Big)\log M\tau+\zeta^{\prime}\Big(0,\frac{a}{M}\Big)+\zeta^{\prime}\Big(0,\frac{M-a}{M}\Big)\\
			&=-\log 2\sin \frac{a\pi}{M}.
		\end{align*}
		Similarly, since $ s=0 $ is the singularity of order $ 2 $ or $G_2$, we obtain
		\begin{align*}
			\underset{s=0}{\mathrm{Res}}G_{2} &=-4\log \tau+\log \frac{1}{a}+\log \frac{1}{M-a}+\log \frac{1}{M+a}+\log \frac{1}{2M-a}\\
			&=-4\log \tau -\log a(M-a)(M+a)(2M-a).
		\end{align*}
		Furthermore,
		\begin{align*}
			|\tau ^{-s}|=\left|\mathrm{exp}\left(\log |\tau| +\Im(s) \mathrm{Arg}(\tau)\right)\right|\leq |\tau|^{\frac{3}{4}}e^{\frac{|\Im(s)|\pi}{4}},
		\end{align*}
		since $  |Y|\leq \frac{1}{2\pi X} . $
	\end{proof}
	By using  \eqref{abs-zeta-geq5} and \eqref{abs-zeta-leq5}, it leads to that
	\begin{align*}
		&  \Big|\frac{1}{2\pi i}\int_{(-\frac{3}{4})} G_{1}ds\Big|  \leq \frac{|\tau|^{\frac{3}{4}}}{2\pi } \int_{-\infty}^{\infty}M^{\frac{3}{4}} e^{\frac{|t|\pi}{4}} \Big|\Gamma\big(-\frac{3}{4}+it\big)\Big|\Big|\zeta\big(\frac{1}{4}+it\big)\Big|\\
		&\quad \times \left(\Big|\zeta\big(-\frac{3}{4}+it,\frac{a}{M}\big)\Big|+\Big|\zeta\big(-\frac{3}{4}+it,\frac{M-a}{M}\big)\Big|\right)dt\\
		&\leq \frac{|\tau|^{\frac{3}{4}}}{\pi } \left(\int_{-\infty}^{-5}+\int_{5}^{\infty}\right)M^{\frac{3}{4}} e^{\frac{|t|\pi}{4}} \Big|\Gamma\big(-\frac{3}{4}+it\big)\Big| \Big|\zeta\big(\frac{1}{4}+it\big)\Big| \Big|\zeta\big(-\frac{3}{4}+it,1\big)\Big|dt\\
		&\quad +\frac{|\tau|^{\frac{3}{4}}}{ \pi }\int_{-5}^{5} M^{\frac{3}{4}} e^{\frac{|t|\pi}{4}} \Big|\Gamma\big(-\frac{3}{4}+it\big)\Big| \Big|\zeta\big(\frac{1}{4}+it\big)\Big|\frac{2 \Gamma\left(\frac{7}{4}\right) \zeta\left(\frac{7}{4}\right) \cosh \left(\frac{\pi|t|}{2}\right)}{(2 \pi)^{\frac{7}{4}}}dt\\
		&\leq1.837\left|M\tau\right|^{\frac{3}{4}}.
	\end{align*}
	Moreover,
	\begin{align*}
		&  \Big|  \frac{1}{2\pi i}\int_{(-\frac{3}{4})} G_{2}  ds    \Big|\leq \frac{|\tau|^{\frac{3}{4}}}{2\pi } \int_{-\infty}^{\infty}M^{\frac{3}{4}} e^{\frac{|t|\pi}{4}} \Big|\Gamma\big(-\frac{3}{4}+it\big)\Big|\Big|\zeta\big(\frac{1}{4}+it\big)\Big|\\
		&\quad\times  \Big|\Big(\frac{1}{a^{-\frac{3}{4}+it}}+ \frac{1}{(M-a)^{-\frac{3}{4}+it}}+\frac{1}{(M+a)^{-\frac{3}{4}+it}}+\frac{1}{(2M-a)^{-\frac{3}{4}+it}}\Big)\Big|dt\\
		&\leq \frac{|\tau|^{\frac{3}{4}}}{2\pi } \int_{-\infty}^{\infty} 4.33M^{\frac{3}{4}} e^{\frac{|t|\pi}{4}} \Big|\Gamma\big(-\frac{3}{4}+it\big)\Big| \Big|\zeta\big(\frac{1}{4}+it\big)\Big| dt\\
		&\leq 3.13\left|M\tau\right|^{\frac{3}{4}}.
	\end{align*}
	Since \begin{align*}
		|\tau |\leq \sqrt{X^{-2}+4\pi^{2}Y^{2}}\leq\sqrt{2}X^{-1},
	\end{align*}
	it follows that
	\begin{align*}
		|E_{+}|\leq \left|  G_{1}  \right|+\left|   G_{2}  \right|\leq {7.03} M^{\frac{3}{4}}X^{-\frac{3}{4}}.
	\end{align*}
	
	Next, we evaluate $\log G(-e^{-\tau}) $ according to the different cases for the parities of $a$ and $M$. For $ \log G(-e^{-\tau}) $ in (Case 1), we have
	\begin{align*}
		&\quad \log G(-e^{-\tau})\\
		&=\frac{1}{2\pi i} \int_{(\frac{3}{2})} \frac{1}{\tau^{s}}\Gamma(s) \sum_{m\geq 0}\sum_{\ell\geq 1}\frac{(-1)^{\ell}}{\ell^{s+1}} \Big(\frac{1}{(mM+a)^{s}} + \frac{1}{(mM+M-a)^{s}}  \Big)ds\\
		&\quad -\frac{1}{2\pi i} \int_{(\frac{3}{2})} \frac{1}{\tau^{s}}\Gamma(s) \sum_{m=0}^{1}\sum_{\ell\geq 1}\frac{(-1)^{\ell}}{\ell^{s+1}} \Big(\frac{1}{(mM+a)^{s}} + \frac{1}{(mM+M-a)^{s}}  \Big)ds\\
		&=\frac{1}{2\pi i} \int_{(\frac{3}{2})} \frac{(1-2^{-s})}{(M\tau)^{s}}\Gamma(s)  \zeta(s+1)\Big(\zeta(s,\frac{a}{M})+\zeta(s,\frac{M-a}{M})\Big)ds\\
		&\quad -\frac{1}{2\pi i} \int_{(\frac{3}{2})} \frac{(1-2^{-s})}{\tau ^{s}}\Gamma(s) \zeta(s+1) \Big(\frac{1}{a^{s}}+ \frac{1}{(M-a)^{s}}+\frac{1}{(M+a)^{s}}+\frac{1}{(2M-a)^{s}}\Big)ds.
	\end{align*}
	Define
	\begin{align*}
		&G_{1}^{1}:=\frac{(1-2^{-s})}{(M\tau)^{s}}\Gamma(s)  \zeta(s+1)\Big(\zeta(s,\frac{a}{M})+\zeta(s,\frac{M-a}{M})\Big),\\
		&G_{2}^{1}:=- \frac{(1-2^{-s})}{\tau ^{s}}\Gamma(s) \zeta(s+1) \Big(\frac{1}{a^{s}}+ \frac{1}{(M-a)^{s}}+\frac{1}{(M+a)^{s}}+\frac{1}{(2M-a)^{s}}\Big).
	\end{align*}
	Then by changing the path of integration to $ \left(-\frac{3}{4}\right) $ and considering the residues in strip $ \left(-\frac{3}{4}\right)<\Re(s)< \left(\frac{3}{2}\right)$, it can be seen
	\begin{align*}	&\quad\underset{s=1}{\mathrm{Res}}G_{1}^{1}+\underset{s=0}{\mathrm{Res}}G_{1}^{1}+\underset{s=0}{\mathrm{Res}}G_{2}^{1}\\ &=\frac{\pi^{2}}{6M\tau}+\Big(\zeta\big(0,\frac{a}{M}\big)+\zeta\big(0,\frac{M-a}{M}\big)\Big)  \log \frac{1}{2}-4\log 2=\frac{\pi^{2}}{6M\tau}-4\log 2.
	\end{align*}
	For the error term, we have
	\begin{align*}
		&\quad|E_{1,-}|\leq \Big|\frac{1}{2\pi i}\int_{(-\frac{3}{4})} \left(G_{1}^{1} +G_{2}^{1}\right)ds\Big|\\
		&\leq \left( 1+2^{\frac{3}{4}}\right)\bigg(\frac{|\tau|^{\frac{3}{4}}}{2\pi } \int_{-\infty}^{\infty}2M^{\frac{3}{4}} e^{\frac{|t|\pi}{4}} \Big|\Gamma\big(-\frac{3}{4}+it\big)\Big|\Big|\zeta\big(\frac{1}{4}+it\big)\Big|\Big|\zeta\big(-\frac{3}{4}+it,1\big)\Big|dt\\
		&\quad +\frac{|\tau|^{\frac{3}{4}}}{2\pi } \int_{-\infty}^{\infty} 4.33M^{\frac{3}{4}} e^{\frac{|t|\pi}{4}} \Big|\Gamma\big(-\frac{3}{4}+it\big)\Big| \Big|\zeta\big(\frac{1}{4}+it\big)\Big|\bigg)dt
		\leq 11.4M^{\frac{3}{4}}X^{-\frac{3}{4}}.
	\end{align*}
	
	For brevity, we omit the details for (Case 2) and (Case 3), and thereby the proof is complete.
	\qed	
	
	\section{The asymptotic formula and Theorem \ref{thm-large-n}}\label{pf-large-n}
	
	The circle method was first used by Hardy and Ramanujan \cite{H-R} in 1918 to investigate the asymptotic growth of the partition function $p(n)$, which was later refined to deal with the coefficients of the $q$-expansions of modular forms \cite{Rademacher}.  Recently, Chern \cite{Chern} extended this method to study the nonmodular infinite products.  The key idea of this method is to write the function under consideration into an integral over a circle, then transforming the integral into something more tractable to  derive the asymptotic formula for the coefficient of the function.

	To be more precisely, throughout this section, we assume that
	\begin{align*}
		X=\sqrt{\frac{3Mn}{\pi^{2}}}.
	\end{align*}
	From Lemma \ref{lem-Gq-bound}, we set $ X\geq 540 M^{5} $, and thereby the bound for $n$ is given as follows
	\begin{align}\label{X-n-bound}
		n\geq(4.63M)^{9}.
	\end{align}

	\subsection{Proof of Theorem  \ref{approxiamtion-gn}}
	In  this section, we will employ the circle method to  estimate $g_{a,M}(n) $ for sufficiently large $n$.
	Based on Theorem \ref{lem-Gq-bound} and Theorem \ref{approximation-near-pole}, we are ready to prove Theorem  \ref{approxiamtion-gn}.
	
	\noindent{\it Proof of Theorem  \ref{approxiamtion-gn}.}
	With the residue theorem \eqref{Cauchy}, we have that
	\begin{align*}
		g_{a,M}(n) &=\frac{1}{2\pi i } {\int}_{|q|=e^{-\frac{1}{X}}}\frac{G_{a,M}(q)}{q^{n+1}} d q\\
		&=e^{\frac{n}{X}} \int_{-\frac{1}{2 \pi X}}^{1-\frac{1}{2 \pi X}} G\Big(e^{-\left(X^{-1}+2 \pi i t\right)}\Big) e^{2 \pi i n t} d t .
	\end{align*}
	Let $ I $ denote the interval $ \left[-\frac{1}{2\pi X},1-\frac{1}{2\pi X}\right] $ and denote
	\begin{align*}
		&I_{1}:=\left[-\frac{1}{2\pi X}, \frac{1}{2\pi X}\right],\\[3pt]
		&I_{2}:=\left[\frac{1}{2}-\frac{1}{2\pi X},\frac{1}{2}+\frac{1}{2\pi X}\right],\\[3pt]
		&I_{3}:=I-I_{1}-I_{2}.
	\end{align*}
	Then
	\begin{align}\label{gn}
		g_{a,M}(n)=\sum_{i=1}^{3}e^{\frac{n}{X}} \int_{I_{i}} G\left(e^{-\left(X^{-1}+2 \pi i t\right)}\right) e^{2 \pi i n t} d t =:\sum_{i=1}^{3}g_{i}(n).
	\end{align}
	
	First, by applying \eqref{caseall-+}, we have
	\begin{align*}
		g_{1}(n)&=e^{\frac{n}{X}} \int_{-\frac{1}{2 \pi X}}^{ -\frac{1}{2 \pi X}} G\Big(e^{-\left(X^{-1}+2 \pi i t\right)}\Big) e^{2 \pi i n t} d t\\[3pt]
		&=\frac{1}{2\pi i} \int_{\frac{1}{X}-i\frac{1}{X}}^{\frac{1}{X}+i\frac{1}{X}} e^{n\tau} G(e^{-\tau})d\tau\\[3pt]
		&=\frac{\exp(E_{+})}{2\Delta_{M}\sin\frac{a\pi}{M}}\frac{1}{2\pi i} \int_{\frac{1}{X}-i\frac{1}{X}}^{\frac{1}{X}+i\frac{1}{X}} \tau^{4} \mathrm{exp}\Big(  \frac{\pi^{2}}{3M\tau} +n \tau \Big)d\tau.
	\end{align*}
	Then  we rewrite the above integral into the following form
	\begin{align*}
		&\quad\frac{1}{2\pi i} \int_{\frac{1}{X}-i\frac{1}{X}}^{\frac{1}{X}+i\frac{1}{X}} \tau^{4} \mathrm{exp}\Big(  \frac{\pi^{2}}{3M\tau} +n \tau \Big)d\tau\\[3pt]
		&=\frac{1}{2\pi i}\Big(\int_{C}- \int_{-\infty-i\frac{1}{X}}^{\frac{1}{X}-i\frac{1}{X}}-\int_{\frac{1}{X}+i\frac{1}{X}}^{-\infty+i\frac{1}{X}}\Big)  \tau^{4} \mathrm{exp}\Big(  \frac{\pi^{2}}{3M\tau} +n \tau \Big)d\tau\\[3pt]
		&=:g_{11}(n)+g_{12}(n)+g_{13}(n),
	\end{align*}
	where the Hankel contour $ C $ is given as
	\begin{align*}
		C:\quad 	\left(-\infty-i X^{-1}\right) \rightarrow\left(X^{-1}-i X^{-1}\right) \rightarrow\left(X^{-1}+i X^{-1}\right) \rightarrow\left(-\infty+i X^{-1}\right).
	\end{align*}
	By the definition of Bessel function \eqref{Bessel}, we have
	\begin{align}\label{g11}
		g_{11}(n)&=\Big({\frac{\pi^{2}}{3Mn}}\Big)^{ \frac{5}{2}}\frac{1}{2\pi i}\int_{C^{\ast}} s^{4} \mathrm{exp}\Big(  \sqrt{\frac{\pi^{2}n}{3M}}\big(s+\frac{1}{s}\big) \Big)ds\notag\\[3pt]
		&=\Big({\frac{\pi^{2}}{3Mn}}\Big)^{\frac{5}{2}}I_{-5}\Big(2\pi\sqrt{\frac{n}{3M}}\Big),
	\end{align}
	which contributes to the  main term  for $g_{a,M}(n)$.
	
	For $ g_{12}(n) $, by letting $ \tau=s-iX^{-1} $ and noting that \eqref{re-tau}, we obtain
	\begin{align}\label{a/tau}
		\left|e^{\frac{a}{\tau}} \right| \leq e^{\frac{aX}{2}}, \quad\left|e^{n \tau}\right|=e^{n s},
	\end{align}
	and
	\begin{align*}
		|\tau^4|=s^{4}+X^{-4}-6s^{2}X^{-2}\leq s^{4}+X^{-4} .
	\end{align*}
	Then it can be seen that
	\begin{align}\label{g12}
		|g_{12}(n)| & \leq \frac{1}{2\pi} \int_{-\infty}^{\frac{1}{X} } |\tau^{4}| \Big|\mathrm{exp}\Big(  \frac{\pi^{2}}{3M\tau} +n \tau \Big)\Big| d\tau  \notag\\
		&\leq \frac{1}{2\pi}   e^{\frac{\pi^{2}}{3M} \frac{X}{2}} \int_{-\infty}^{\frac{1}{X}} e^{ns}\left(s^{4}+X^{-4}\right) ds\notag\\[3pt]
		&\leq\frac{1}{2\pi}e^{\frac{\pi^{2}}{6M} X}  \Big(\int_{-\infty}^{0} e^{ns} s^{4} ds
		+		\int_{-\frac{1}{X}}^{\frac{1}{X}} e^{ns} X^{-4} ds\Big)\notag\\[3pt]
		&\leq\frac{1}{ 2 \pi}e^{\frac{\pi^{2}}{6M} X} \Big( \frac{\Gamma(5)}{n^{5}}+\frac{e^{\frac{n}{X}}}{nX^{4}}\Big)\notag\\[3pt]
		&=\frac{\Gamma(5)}{n^{5}\pi} \mathrm{exp}\Big(\frac{\pi}{2} \sqrt{\frac{n}{3M}}\Big)+\frac{\pi^{3}}{9M^{2}n^{3}}\mathrm{exp}\Big(\frac{3\pi}{2} \sqrt{\frac{n}{3M}}\Big)=:g_{12}^{u}(n).
	\end{align}
	Similarly, we obtain
	\begin{align}\label{g13}
		\quad|g_{13}(n)|\leq\frac{ \Gamma(5)}{n^{5}\pi} \mathrm{exp}\Big(\frac{\pi}{2} \sqrt{\frac{n}{3M}}\Big)+\frac{\pi^{3}}{9M^{2}n^{3}}\mathrm{exp}\Big(\frac{3\pi}{2} \sqrt{\frac{n}{3M}}\Big)=:g_{13}^{u}(n).
	\end{align}
	
	Next, according to the three cases for the parities of $a$ and $M$ as given in Theorem \ref{approximation-near-pole}, we will give the  estimation for $ g_{2}(n) $.
	For (Case 1), that is, $ a \equiv1 \ (\bmod\ 2) $ and $ M \equiv0  \ (\bmod\ 2)$,   by setting  $ t\to t-\frac{1}{2} $ and using Theorem \ref{approximation-near-pole}, it implies that
	\begin{align*}
		g_{2}(n)&=(-1)^{n}e^{\frac{n}{X}} \int_{-\frac{1}{2 \pi X}}^{ -\frac{1}{2 \pi X}} G\Big(-e^{-\left(X^{-1}+2 \pi i t\right)}\Big) e^{2 \pi i n t} d t\\[3pt]
		&=\frac{(-1)^{n}}{2\pi i} \int_{\frac{1}{X}-i\frac{1}{X}}^{\frac{1}{X}+i\frac{1}{X}} e^{n\tau} G(e^{-\tau})d\tau\\[3pt]
		&=\frac{\exp(E_{1,-})}{16}\frac{(-1)^{n}}{2\pi i} \int_{\frac{1}{X}-i\frac{1}{X}}^{\frac{1}{X}+i\frac{1}{X}}   \mathrm{exp}\Big(  \frac{\pi^{2}}{6M\tau} +n \tau \Big)d\tau.
	\end{align*}
	Then it can be divided as the way dealing with $ g_{1}(n) $
	\begin{align*}
		&\quad	\frac{1}{2\pi i} \int_{\frac{1}{X}-i\frac{1}{X}}^{\frac{1}{X}+i\frac{1}{X}}   \mathrm{exp}\Big(  \frac{\pi^{2}}{6M\tau} +n \tau \Big)d\tau\\[3pt]
		&=\frac{1}{2\pi i}\Big(\int_{C}- \int_{-\infty-i\frac{1}{X}}^{\frac{1}{X}-i\frac{1}{X}}-\int_{\frac{1}{X}+i\frac{1}{X}}^{-\infty+i\frac{1}{X}}\Big)  \mathrm{exp}\Big(  \frac{\pi^{2}}{6M\tau} +n \tau \Big)d\tau\\[3pt]
		&=:g_{21}(n)+g_{22}(n)+g_{23}(n).
	\end{align*}
	It directly follows that
	\begin{align}\label{g21}
		g_{21}(n)=\Big({\frac{\pi^{2}}{6Mn}}\Big)^{\frac{1}{2}}I_{-1}\Big(\sqrt{2}\pi\sqrt{\frac{n}{3M}}\Big) .
	\end{align}
	For $ g_{2 2}(n) $, by \eqref{a/tau}, let $ \tau=x-iX^{-1} $ and we have
	\begin{align}\label{g22}
		&\quad |	g_{22}(n)| =\Big|\frac{1}{2\pi i} \int_{\infty-i\frac{1}{X}}^{\frac{1}{X}-i\frac{1}{X}} \mathrm{exp}\Big(  \frac{\pi^{2}}{6M\tau} +n \tau \Big)d\tau\Big|\notag\\[3pt]
		&\leq \frac{1}{2\pi}e^{\frac{\pi^{2} X}{12M}} \int_{-\infty}^{\frac{1}{X}} e^{nx}dx= \frac{1}{2\pi}e^{\frac{\pi^{2} X}{12M}}\cdot \frac{e^{\frac{n}{X}}}{n}\notag\\[3pt]
		&=\frac{1}{2\pi n}\mathrm{exp}\Big( \frac{5\pi}{4}\sqrt{\frac{n}{3M}}\Big) =:g_{22}^{u}(n).
	\end{align}
	Simlilarly, we also have
	\begin{align}\label{g23}
		|	g_{2 3}(n)|\leq\frac{1}{2\pi n}\mathrm{exp}\Big( \frac{5\pi}{4}\sqrt{\frac{n}{3M}}\Big)=:g_{23}^{u}(n).
	\end{align}
	
	From Theorem \ref{Phi-aM} and Theorem \ref{lem-Gq-bound},  we can see that $I_3$ is covered by $\{\frac{h}{k_1}-|Y| \}$ where  $\frac{1}{\sqrt{2\pi X}}\leq \frac{h}{k_1}\leq 1$ and $\frac{1}{2\pi X}\leq |Y|\leq \frac{1}{k_1N}\leq \frac{1}{\sqrt{2\pi X}}$. Then for $ g_{3}(n) $, by applying Theorem \ref{lem-Gq-bound}, we have
	\begin{align}\label{g3}
		&\quad|	g_{ {3}}(n)|=\Big|e^{\frac{n}{X}} \int_{I_{3}} G\left(e^{-\left(X^{-1}+2 \pi i t\right)}\right) e^{2 \pi i n t} d t \Big|\notag\\[3pt]
		&\leq e^{\frac{n}{X}} \int_{I_{3}} e^{ \frac{\pi^{2}X}{2\sqrt{3}M}}d t\leq e^{\frac{n}{X}} \cdot e^{ \frac{\pi^{2}X}{2\sqrt{3}M}}\notag\\[3pt]
		&=\mathrm{exp}\Big(\frac{\sqrt{3}+2}{2}\pi\sqrt{\frac{n}{3M}}\Big)=:	g_{31}^{u}(n).
	\end{align}
	
	Above all, when $ a \equiv1 \ (\bmod \ 2) $ and $ M \equiv0 \  (\bmod\  2) $ ((Case 1) as given in Theorem \ref{approximation-near-pole}), we obtain  the main term by substituting \eqref{g11} into \eqref{gn}. In this case, we can see that   \eqref{g12}--\eqref{g3} are all error terms, thus we obtain the asymptotic estimation for $ g_{a,M}(n) $ as given in Theorem \ref{approxiamtion-gn}.
	
	For the other two cases of $a$ and $M$ as  discussed  in Theorem \ref{approximation-near-pole}, the only difference is that the error terms are from $g_2(n)$. By using the second case given by \eqref{Ge-tau}, we obtain that in these two cases
	\begin{align*}
		&g_{21}(n)=\Big({\frac{\pi^{2}}{4Mn}}\Big)^{\frac{1}{2}}I_{-1}\Big(\sqrt{3}\pi\sqrt{\frac{n}{3M}}\Big),\\
		&	\text{max}\{|	g_{2 2}(n)| ,|g_{2 3}(n)|\}\leq\frac{1}{2\pi n}\mathrm{exp}\Big( \frac{11\pi}{8}\sqrt{\frac{n}{3M}}\Big).
	\end{align*}
	This completes the proof.
	\qed

	\subsection{Merca's stronger conjecture for sufficiently large $n$}
	
	Note that from Corollary \ref{cor-r=2},  Merca's stronger conjecture \ref{stronger-ja-1} has been certified for the cases $R=2S$. Moreover, the authors \cite{DingLisa} have proved the cases when $ R= 3S $. Thus now it is sufficient for us to consider the cases when $ R\geq 4 $, and then $ M\geq 4 $ correspondingly. Firstly, let us deduce the bound for $ g_{a,M}(n)  $ as given in Theorem \ref{bound-gast} when $ n\geq(4.63M)^{9} $.
	
	{\noindent \it Proof of Theorem \ref{bound-gast}}.
	From \eqref{Bessel}, denote
	\begin{align*}
		f_{v}:=\frac{(v^{2}-\frac{1}{4})}{t}.
	\end{align*}
	Thus for real number number $  t$, we have
	\begin{align}
		&I_{v}(t)\leq \frac{e^{t}}{\sqrt{2\pi t}}\left(1+\pi f_{v} e^{f_{v}}\right)+ \frac{e^{-t}}{\sqrt{2\pi t}}\left(1+f_{v}e^{f_{v}}\right) ,\label{Ivtu}\\
		&I_{v}(t)\geq \frac{e^{t}}{\sqrt{2\pi t}}\left(1-\pi f_{v} e^{f_{v}}\right)- \frac{e^{-t}}{\sqrt{2\pi t}}\left(1+f_{v}e^{f_{v}}\right) .\label{Ivtd}
	\end{align}
	With $n\geq(4.63M)^{9}$, it can be seen that
	\begin{align}\label{f-5}
		f_{-5}:=\frac{(v^{2}-\frac{1}{4})}{n}
		\leq0.000001,
	\end{align}
	and
	\begin{align}\label{e-n}
		\frac{e^{-n}}{\sqrt{2\pi n}} (1+f_{-5}e^{f_{-5}}) \leq 	\frac{2e^{-n}}{\sqrt{2\pi n}} \leq 0.000001	\frac{e^{n}}{\sqrt{2\pi n}} .
	\end{align}
	Combining \eqref{Ivtu}--\eqref{e-n}, it follows that
	\begin{align}
		&I_{-5} (n)\leq \frac{e^{n}}{\sqrt{2\pi n}}\left(1+f_{5}\pi e^{f_{5}}+0.000001\right)<\frac{1.0001e^{n}}{\sqrt{2\pi n}} =:g_{11}^{u}(n),\label{Ivtu-1}\\
		&I_{-5} (n)\geq \frac{e^{t}}{\sqrt{2\pi n}}\left(1-f_{5}\pi e^{f_{5}}-0.000001\right)>\frac{0.9999e^{n}}{\sqrt{2\pi n}} =:g_{11}^{d}(n).\label{Ivtd-1}
	\end{align}
	
	From Theorem \ref{approximation-near-pole}, we have that
	\begin{align*}
		7.02 M^{\frac{3}{4}}X^{-\frac{3}{4}}\leq0.001,
	\end{align*}
	thus by Taylor's expansion,
	\[0.999\leq 1-|E_{+}|\leq e^{ E_{+} }  \leq 1+2|E_{+}|\leq 1.002.\]
	
	By using  \eqref{g11}, \eqref{g12}, \eqref{g21},  \eqref{g22}, \eqref{g3}, \eqref{Ivtu-1} and \eqref{Ivtd-1}, it leads to that
	\begin{align*}
		&g_{a,M}(n) \leq(1.002-1)a_{11}g_{11}^{u}(n)+ \sum _{i,j}  a_{ij}g_{i j}^{u}(n)  ,\label{g-u-bound}\\
		&g_{a,M}(n) \geq  0.999 a_{11}g_{11}^{d}(n)+ a_{11}g_{11}^{u}(n)- \sum _{i,j} a_{ij}g_{i j}^{u}(n)   ,
	\end{align*}
	where the coefficients $ a_{ij} $ can be found in the process to evaluate $ g_{a,M}(n) $  in \eqref{gn} with the results of  Theorem \ref{approximation-near-pole}.
	By the definition of $ g_{12}^{u} $, $g_{22}^{u}$ and $ g_{31}^{u} $, when $ 	n\geq (4.63 M)^{9}\geq (4.63\times 4)^9  $, it implies that
	\[  a_{22}g_{22}^{u}(n)<a_{12}g_{12}^{u}(n)<a_{31}g_{31}^{u}(n) \leq \frac{1}{2\Delta_{M}\sin \frac{a\pi}{M}}  \Big({\frac{\pi^{2}}{3Mn}}\Big)^{\frac{5}{2}}\frac{0.000001 e^{n_{0}}}{\sqrt{2\pi n_{0}}},\]
	where $ n_{0}=2\pi\sqrt{\frac{n}{3M}} $.
	We can therefore deduce that
	\begin{align*}
		\Big({\frac{\pi^{2}}{3Mn}}\Big)^{\frac{5}{2}}\frac{0.99}{2\Delta_{M}\sin \frac{a\pi}{M}}\frac{e^{ n_{0}}}{\sqrt{2\pi  n_{0}}}\leq g_{a,M}(n)\leq\Big({\frac{\pi^{2}}{3Mn}}\Big)^{\frac{5}{2}} \frac{1.01}{2\Delta_{M}\sin \frac{a\pi}{M}}\frac{e^{ n_{0}}}{\sqrt{2\pi  n_{0}}} .
	\end{align*}
	Finally, the proof is complete by replacing $ a,M $ by $ s,r $, respectively.
	\qed
	
	Now, we are ready to prove the stronger version of Merca's conjecture on the truncated Jacobi triple theta series \eqref{TJTTS_Merca} for sufficiently large $n$.
	
	\noindent{\it Proof of Theorem  \ref{thm-large-n}.}  Firstly, we consider the case when  $ (s,r-s,r+s,2r-s)=1 $. From \eqref{p-condition}, let
	\begin{align}\label{p-bound-non-k}
		p\geq 2r^{2}s^{1/3},
	\end{align}
	and
	\begin{align}\label{def-F1}
		F_{1}:=(2r^{2}s^{1/3}+k) (2\cdot 2r^{2}s^{1/3}\cdot r-r+2 s).
	\end{align}
	Redefine the right hand side of \eqref{main-transform} as follows,
	\begin{align*}
		&\sum_{n=0}^{\infty} b(n)q^{n}:=\quad\frac{\sum_{j=0}^{\infty}\left(q^{t_{1,j}}-q^{t_{2,j}}-q^{t_{3,j}}+q^{t_{4,j}}\right)}{(q^s,q^{r-s}; q^r)_{2}}.
	\end{align*}
	Then from the results in Section \ref{coprime}, we see that
	\begin{align}\label{bn}
		b(n)\geq  b^{\ast} (n),
	\end{align}
	where $ b^{\ast}(n) $ is the combination of $   {D_{i,d}}(n,p )/(6\Delta) $ for $i=1,2,3,4 $,
	i.e., for any nonnegative integer $ p $,
	\begin{align*}
		b^{\ast}(n):=\left\{
		\begin{array}{ll}
			\dps {D_{1,d}}(n,p )/(6\Delta), & t_{4,p-1}\leq n<t_{1,p} ,\\
			\dps{D_{2,d}}(n,p )/(6\Delta), & t_{1,p}\leq n<t_{2,p} ,\\
			\dps{D_{3,d}}(n,p )/(6\Delta), &  t_{2,p}\leq n<t_{3,p} ,\\
			\dps{D_{4,d}}(n,p )/(6\Delta), &  t_{3,p}\leq n<t_{4,p}.
		\end{array}
		\right.
	\end{align*}
	Applying \eqref{bn} and noting the fact that $ g_{s,r}(n) \geq 0 $ as given in \eqref{def-Gq}, then
	\begin{align}\label{Hn}		\sum_{n=0}^{\infty}H(n)q^{n}&:=\frac{\sum_{j=0}^{\infty}\left(q^{t_{1,j}}-q^{t_{2,j}}-q^{t_{3,j}}+q^{t_{4,j}}\right)}{(q^{ s},q^{ r-s};q^r)_{\infty}}\notag\\[3pt]
		&=\frac{1}{(q^{2r+s},q^{3r-s};q^r)_{\infty}} \frac{\sum_{j=0}^{\infty}\left(q^{t_{1,j}}-q^{t_{2,j}}-q^{t_{3,j}}+q^{t_{4,j}}\right)}{(q^s,q^{r-s};q^r)_{2}  }\notag\\[3pt]
		&=\sum_{n=0}^{\infty} g_{s,r}(n)q^{n}\times\Big(\sum_{n=0}^{F_{1}-1} b (n)q^{n}+\sum_{n=F_{1}}^{\infty} b (n)q^{n}\Big)\notag\\
		&=\sum_{n=0}^{\infty}\sum_{m=0}^{n} b  (m) g_{s,r}(n-m) q^{n}.
	\end{align}
	We have
	\begin{align*}
		H(n)\geq H_{1}(n),
	\end{align*}
	where
	\begin{align*}
		H_{1}(n) = \sum_{m=0}^{F_{1}-1} b^{\ast} (m) g_{s,r}(n-m)+\sum_{m=F_{1}}^{n} b^{\ast}(m) g_{s,r}(n-m).
	\end{align*}
	
	From the analysis as given in Section \ref{coprime}, it can be seen that  $  b^{\ast}(n) $ are  monotonically increasing for $ n\geq F_{1} $. So that we can use a linear  function to bound $ b^{\ast}(n) $, more specifically, we assume that
	\begin{align*}
		b^{\ast} (n)\geq a_{1} n+a_{0},\quad \text{ for } n\geq F_{1} ,
	\end{align*}
	where $a_{1} $ and $a_{0}  $ are constants.
	With $p\geq 2r^{2}s^{1/3}$, \eqref{Derivation-0}, \eqref{Derivation-1}--\eqref{Derivation-3-2}, we can  bound $ a_{1} $ as follows
	\begin{align}\label{a1-bound}
		\text{min}\{D^{\prime}_{i,d}\}_{1\leq i\leq 4} \geq p^2 \left(12 k r s-12 r s+12 s^2\right)/(6\Delta)=: a_{1}.
	\end{align}
	Since \eqref{p-bound} shows that $ D_{0,d}(t_{4,p-1},p)>0 $, we let  $a_{0}= - a_{1}F_{1} $. Then
	\begin{align}\label{b2-bound}
		b^{\ast} (n)\geq a_{1 } n-a_{1 }F_{1}.
	\end{align}
	
	For $ 0\leq n\leq F_{1}-1 $, by the monotonicity of each case in Section \ref{coprime} we have
	\begin{align}\label{b1-bound}
		&\quad|b^{\ast}(n)|\leq -\sum_{j=0}^{p -1}\left( C^{d}-  C^{u}\right)=p J\notag\\
		&=\frac{p}{6\Delta}(48 r^8 s^2 - 180 r^6 s^4 + 12 s^6 - 60 r^2 s^8 - 12 s^{10} +
		r^5 (-24 s + 60 s^5)\notag\\
		& \quad+ r^4 (-12 s^2 + 204 s^6) +
		r^3 (60 s^3 - 120 s^7) + r (-36 s^5 + 60 s^9))\notag\\
		&<48 pr^8 s^2/(6\Delta)=:b_{1}.
	\end{align}

	When $ n\geq\text{max}\{(4.63r)^{9},F_{1}\} $, applying Theorem \ref{bound-gast}, \eqref{b2-bound} and \eqref{b1-bound}, then
	\begin{align*}
		H_{1}(n)&= \sum_{m=0}^{F_{1}-1} b^{\ast} (m)g_{s,r}(n-m)+\sum_{m=F_{1}}^{n} b^{\ast}(m)g_{s,r}(n-m)\\
		&\geq \sum_{m=0}^{F_{1}-1} b ^{\ast
		}(m)g_{s,r}(n-m)+\sum_{m=F_{1}}^{n} \left(a_{1 } m+a_{0 }\right)g_{s,r}(n-m)\\
		&>-F_{1}b_{1}g^{u}(n)+\sum_{m=0}^{n-F_{1}} \left(a_{1 } (n-m)+a_{0}\right)g^{d}(m).
	\end{align*}
	By using the variable substitution
	\begin{align}\label{def-n1}
		n=n_{1}+F_{1} :=k_{2}n_{1},
	\end{align}
	and noting the fact
	\begin{align*}
		\int x^n e^{a x} d x=\frac{1}{a} x^n e^{a x}-\frac{n}{a} \int x^{n-1} e^{a x} d x,
	\end{align*}
	it leads to that   for
	\begin{align}\label{n1-bound-1}
		n_{1} \geq  2\cdot (4.63r)^{9},
	\end{align}
	the following relations hold
	\begin{align*}
		&\quad-F_{1}b_{1}g^{u}(n)+\sum_{m=0}^{n-F_{1}} \left(a_{1 } (n-m)+a_{0}\right)g^{d}(m)\\[3pt]
		&>\int_{m=\frac{n_{1}}{2}}^{n_{1}} \left(a_{1 } ( n_{1}-m)\right)g^{d}(m)dm-F_{1}b_{1}g^{u}(k_{2}n_{1})\\[3pt]
		&>a_{1}n_{1}\int_{m=\frac{n_{1}}{2}}^{n_{1}} g^{d}(m)dm-a_{1 }\int_{m=\frac{n_{1}}{2}}^{n_{1}} mg^{d}(m)dm-F_{1}b_{1}g^{u}(k_{2}n_{1})\\[3pt]
		&>{\frac{0.99\pi^{4}a_{1} }{2(3r)^{\frac{11}{4}}\Delta\sin ({s\pi}/{r})\times 2 }}   \bigg(\frac{\sqrt{3r}}{2\pi}n_{1}^{-\frac{5}{4}} +
		\frac{9(\sqrt{3r})^{2}}{2(2\pi)^{2}} n_{1}^{-\frac{7}{4}}  \Big. \\
		&\quad \Big.- \frac{\sqrt{3r}}{2\pi}n_{1}^{-\frac{5}{4}} -
		\frac{5(\sqrt{3r})^{2}}{2(2\pi)^{2}} n_{1}^{-\frac{7}{4}}    \bigg)\mathrm{exp} \Big(2\pi\sqrt{\frac{ n_{1}}{3r}}\Big) -1.01F_{1}b_{1}g^{u}(n)\\
		&>{\frac{\pi^{4} n_{1}^{-\frac{11}{4}}}{2(3r)^{\frac{11}{4}}\Delta\sin (s\pi /r)} }\bigg( 0.99  \frac{3ra_{1}}{2\pi^{2}}n_{1}\mathrm{exp} \Big(2\pi\sqrt{\frac{ n_{1}}{3r}}\Big)
		-1.01\frac{F_{1}b_{1}}{k_{2}^{\frac{11}{4}} }\mathrm{exp} \Big(2\pi\sqrt{\frac{k_{2} n_{1} }{3r} }\Big)   \bigg)\\
		&=:H_{2}(n).
	\end{align*}
	When $n_{1} \geq 2\cdot (4.63r)^{9}$, the relation $e^{n_1/2}<2e^{n_1} $ deduces the penultimate step.
	
	Since $ k_{2}>1 $, to ensure $ H_{2}(n)\geq0 $, it's sufficient to show that
	\begin{align}\label{inequality}
		\frac{ {3r}a_{1 } }{2\pi^{2} F_{1}b_{1}}n_{1}\geq \frac{1.01}{0.99} \mathrm{exp}\left(\frac{2\pi}{\sqrt{3r}} \sqrt{n_{1}} (\sqrt{k_{2}}-1)\right).
	\end{align}
	If
	\begin{align}\label{n1-bound-F}
		n_{1}\geq 4 F^{2}_{1},
	\end{align}
	thus by \eqref{def-n1},
	\begin{align*}
		\sqrt{n_{1}} (\sqrt{k_{2}}-1)<\frac{1}{4}.
	\end{align*}
	By noting that $r\geq 4$, the  right hand side of \eqref{inequality} is less than a constant, which is
	\begin{align}\label{exp2pi}
		\frac{1.01}{0.99}\mathrm{exp}\Big(\frac{2\pi}{\sqrt{3r}} \sqrt{n_{1}} (\sqrt{k_{2}}-1)\Big)<1.021\mathrm{exp}\Big(\frac{ \pi}{4\sqrt{3}}  \Big).
	\end{align}
	
	Moreover, by substituting \eqref{p-bound-non-k}, \eqref{def-F1}, \eqref{a1-bound}, \eqref{b1-bound} and \eqref{n1-bound-F} into the left hand side of \eqref{inequality}, it can be seen that \begin{align}\label{F-constant}
		\frac{ {3r}a_{1 } }{2\pi^{2} F_{1}b_1}n_{1}&\geq \frac{3r}{2\pi^{2}}\times\frac{ p^2 \left(12 k r s-12 r s+12 s^2\right) }{48 p r^8 s^2 }\cdot4F_{1}\notag\\
		&  >\frac{3r }{2\pi^{2}}\times \frac{12p s^{2} }{48   r^{8}s^{2}}\times\frac{31}{4} \big(2r^{2}s^{\frac{1}{3}}\big)^{2}r\geq \frac{93s}{4\pi^{2}} >1.021\mathrm{exp} \Big(\frac{ \pi}{ 4\sqrt{3 }}  \Big).
	\end{align}
	Thus \eqref{inequality} holds.

	By combining \eqref{n1-bound-1} and \eqref{n1-bound-F}, we can conclude that  when $ (s,r-s,r+s,2r-s)=1 $, the coefficient of $q^n$ in \eqref{TJTTS_Merca} is nonnegative for
	\begin{align}\label{n-equality}
		n&\geq \text{max}\left\{ 4F^{2}_{1}  ,2\cdot (4.63r)^{9} \right\}+F_{1}+rk(k+1)/2-sk=: N_1 (r,s,k)
	\end{align}

	If $ (s,r-s,r+s,2r-s)\neq1 $, following the similar procedures, we can also obtain such a constant $N (r,s,k)$ such that the coefficient $H (n)$ in \eqref{Hn} is nonnegative when $n\geq N (r,s,k)$ and we will give a brief proof. When $ 2\mid  s$ then $2\nmid r $, $r\geq 5$ and $r\geq 2s+1$, we  first define
	\begin{align*}
		\sum_{n=0}^{\infty}J_{a,M}(n)q^n =\frac{1}{(q^{2M+a },q^{2M-a};q^{M})_{\infty}} .
	\end{align*}
	Then by \eqref{main-transform-3} and \eqref{Hn}, we can divide $H(n)$ as follows
	\begin{align}\label{JD}	\sum_{n=0}^{\infty}H(n)q^{n}&=\frac{\sum_{j=0}^{\infty}\left(q^{t_{1,j}}-q^{t_{2,j}}-q^{t_{3,j}}+q^{t_{4,j}}\right)}{(q^{ s},q^{ r-s};q^r)_{\infty}}\notag\\[3pt]
		&=\sum_{n=0}^{\infty}J_{s,r}(n)q^{n}\sum_{n=0}^{\infty}	D(n)q^{n}.
	\end{align}
	
	We shall next give the corresponding $p,F_{1},a_1$ and $b_{1}$ in the case $ 2\mid  s$.  With the six expressions of $ D_{0,d}(t_{1,p},p)$, $D_{1,d}(t_{1,p},p)$, $D_{1,d}(-r + p r + 2 k p r + 2 p^2 r - s/2,p)$, $D_{2,d}(t_{2,p},p)$, $D_{3,d}(t_{3,p},p) $, $D_{3,d}(t_{4,p},p)$ in the appendix, we find a bound for $p$ by letting
	$$ p^2 (r (4 k s-4 s)+8 s^2)\geq p\cdot 8r^3s^2$$
	and   the remaining terms are positive in the six expressions. Then we have
	\begin{equation}\label{thm17-p}
		p\geq {2r^3s}/({r(k-1)+2s}).
	\end{equation}
	Then for \begin{align*}
		n\geq ({2r^3s}/({r(k-1)+k})+r)(2\cdot {2r^3s}/({r(k-1)+2s}) \cdot r-r+2s),
	\end{align*} Merca's conjecture holds in this case and Corollary \ref{cor-3} holds immediately by further letting
	\begin{align*}
		2\cdot{2r^3s}/({r(k-1)+2s})\leq( r-2s).
	\end{align*}
	We next use a stronger bound for $p$ that
	\begin{align}\label{psbound}
		p\geq r^{4} .
	\end{align}
	With $2\Delta_{1}D_{0,d}(n,p):=2\Delta_{1}D_{0}(n,p)$, $2\Delta_{1}D_{1,d}(n,p) $, $2\Delta_{1}D_{2,d} (n,p)$, $2\Delta_{1}D_{3,d} (n,p)$ in the appendix, we have
	\begin{align*}
		D(n)\geq D_{i,d}(n,p)/2\Delta_1\geq  a_{2 } n-a_{2 }F_{2},\text{ for } 0\leq i\leq 3,
	\end{align*}
	where $a_2=4ps /(2r\Delta_1) $, $F_{2}:=(r^{4}+k) (2\cdot r^{4}\cdot r-r+2 s)$ and one can verify the condition by Mathmatica.
	
	Similar to \eqref{b1-bound}, in this case, the corresponding bound for $ b $ can be expressed as follows
	\begin{align*}
		D(n)\leq 2p(C_3^{u}-C_3^{d})\leq \frac{8pr^3s^2}{2\Delta_1}=:b_2.
	\end{align*}
	By the similar discussion with Section \ref{Nonmodular}, we omit the details and the bounds for $J_{s,r}(n)$ can be given by
	\begin{align*}
		0.99 J(n)\leq 	J_{s,r}(n)\leq 1.01 J(n),
	\end{align*}
	where
	\begin{align*}
		J(n):=\frac{1}{2\Delta_1\sin (s\pi /r)}\Big(\frac{\pi^2}{3rn}\Big)^2\mathrm{exp} \Big(2\pi\sqrt{\frac{n}{3r}}\Big).
	\end{align*}
	Finally, following the similar  steps  as dealing with the convolution of   $C(n)$ and $g_{s,r}(n)$,  and from \eqref{JD}, we obtain that
	\begin{align*}
		H(n)>- 1.01F_{2}b_{2}J(n)+0.99\sum_{m=0}^{n-F_{2}} \left(a_{1 } (n-m)-a_{1}F_{1}\right)J(m):=H^{(2)}_1(n).
	\end{align*}
	Similarly with \eqref{inequality}, we find that when $ n\geq N_2(r,s,k) $, $H^{(2)}_1(n)>0$  where
	\[  N_2(r,s,k):=\text{max}\left\{ 4F^{2}_{2},2\cdot (4.63r)^{9} \right\}+F_{2}+rk(k+1)/2-sk.\]
	
	We omit the details of the case of $ 2\nmid s $ and the corresponding parameters are $p\geq 64r^4/15$,  $F_{3}:=(64r^4/15+k) (2\cdot64r^4/15\cdot r-r+2 s)$, $a_3=2ps/(2r\Delta_2)$, $b_3=16pr^3s^2/(2\Delta_2)$, and thereby
	\begin{align*}
		N_{3}(r,s,k)=\text{max}\left\{ 4F^{2}_{3}  ,2\cdot (4.63r)^{9} \right\}+F_{3}+rk(k+1)/2-sk,
	\end{align*}
	such that  when $ n\geq N_3(r,s,k)$, the coefficient of $q^n$ in \eqref{TJTTS_Merca} is nonnegative. Above all, the proof is complete.
	\qed
	
	In summary, we state the systematic method to prove Merca's conjecture \ref{stronger-ja-2} by determining the constant $N(r,s,k)$ for given $r, s$ and $k$ such that when $n\geq N(r,s,k)$ the conjecture holds.
	
	{\noindent} \rule{14.5cm}{0.08em}
	\centerline{Algorithm for the proof of Merca's stronger conjecture \ref{stronger-ja-2}}
	
	{\noindent} \rule[6pt]{14.5cm}{0.08em}
	
	{\noindent}{\bf Case 1:} If $ (s,r-s,r+s, 2r-s)=1$, $N(r,s,k)$ can be determined as follows:
	
	{\bf Step 1.1:} Compute
	\[
	{L_1(r,s,k)=\left(2 r^2 \sqrt[3]{s/k}+k\right) \left(4 r^3 \sqrt[3]{s/k}-r+2 s\right)}
	\]
	as given in \eqref{def-F}. Using computer  programs, if we can show that for  $0\leq n \leq L_1(r,s,k)$ the coefficients of $q^n$ in the theta series \eqref{main-transform} are nonnegative, then Conjecture \ref{stronger-ja-2} holds directly as a corollary.  Alternatively, if $k$ is large enough and satisfies the condition in Corollary \ref{cor-main},  Merca's conjecture holds directly, too. Then return $N(r,s,k)=0$.  Otherwise, go to Step 1.2.
	
	{\bf Step 1.2:} By computing the convolution of the series for $C(n)$ and $g_{s,r}(n)$ as shown in \eqref{Hn}, return
	\[
	N(r,s,k)=N_{1}(r,s,k)= \text{max}\left\{ 4F^{2}_{1}  ,2\cdot (4.63r)^{9} \right\}+F_{1}+rk(k+1)/2-sk,
	\]
	where
	{	 $F_{1} $ is given by \eqref{def-F1}.
		
		{\noindent}{\bf Case 2:} If
		$ (s,r-s,r+s,2r-s)\neq 1 $ and $2\mid  s$,   the procedures are as follows:
		
		{\bf Step 2.1:} For given $r,s,k$, let $z_0$ be the maximum zero of  $ D_{0,d}(t_{1,p},p)$, $D_{1,d}(t_{1,p},p)$, $D_{1,d}(-r + p r + 2 k p r + 2 p^2 r - s/2,p)$, $D_{2,d}(t_{2,p},p)$, $D_{3,d}(t_{3,p},p) $, $D_{3,d}(t_{4,p},p)$, which are polynomials relative to $p$ of degree $2$. Compute
		\[
		{L_2(r,s,k)=(z_0+k) (2 z_0 r-r+2 s).}
		\]
		If we can show the theta series \eqref{main-transform-3} has nonnegative terms for $0\leq n \leq L_2(r,s,k)$, then Conjecture \ref{stronger-ja-2} holds directly as a corollary. Else if $k$ satisfies the condition in Corollary \ref{cor-3},  Merca's conjecture holds immediately for such $k$. Then return $N(r,s,k)=0$. Otherwise, go to Step 2.2.

		{\bf Step 2.2:} By studying the convolution of the series for $J_{s,r}(n)$ and $D(n)$ as given in \eqref{JD}, return
		\begin{align*}
			N(r,s,k)=N_{2}(r,s,k)=\text{max}\left\{ 4F^{2}_{2}  ,2\cdot (4.63r)^{9} \right\}+F_{2}+rk(k+1)/2-sk,
		\end{align*}
		where $F_{2}=(r^{4}+k) (2\cdot r^4\cdot r-r+2 s)$.
		
		{\noindent}{\bf Case 3:} If $ (s,r-s,r+s,2r-s)\neq 1 $, and $2\nmid  s$, the procedure is similar to Case 2 and we just need to consider  $E(n)$ and $J_{r-s,r}(n)$.
		
		{\bf Step 3.1:} For given $r,s,k$, let $z_0$ be the maximum zero of the polynomials  $ E_{0,d}(t_{1,p},p)$, $E_{1,d}(t_{1,p},p)$, $E_{1,d}(-r + p r + 2 k p r + 2 p^2 r - s/2,p)$, $E_{2,d}(t_{2,p},p)$, $E_{3,d}(t_{3,p},p) $, $E_{3,d}(t_{4,p},p)$ on $p$. Compute
		\[
		{L_3(r,s,k)=(z_0+k) (2 z_0 r-r+2 s).}
		\]
		Moreover, if we can show the theta series \eqref{main-transform-3} has nonnegative terms for $0\leq n \leq L_3(r,s,k)$ or the $k$ satisfies the condition in Corollary \ref{cor-3}, then Conjecture \ref{stronger-ja-2} holds directly as a corollary. Then return $N(r,s,k)=0$. Otherwise, go to Step 3.2.
		
		{\bf Step 3.2:} By studying the convolution of the series for $J_{r-s,r}(n)$ and $E(n)$ similar to \eqref{JD}, return
		\begin{align*}
			N(r,s,k)=N_{3}(r,s,k)=\text{max}\left\{ 4F^{2}_{3}  ,2\cdot (4.63r)^{9} \right\}+F_{3}+rk(k+1)/2-sk,
		\end{align*}
		where $F_{3}=(64r^4/15+k) (2\cdot64r^4/15\cdot r-r+2 s)$.
		
		{\noindent} \rule[6pt]{14.5cm}{0.08em}
		
		We also remark that for the bound $L_1(r,s,k)$ as given in Case 1, we can also improve it following the similar analysis  as discussed in Case 2 and Case 3, in which we need to consider the maximum zero for certain polynomials.
		
		Now, we give an example to illustrate the process to determine $N (r,s,k)$ when $ (s,r-s,r+s,2r-s)=1 $.
		
		\begin{ex}
			When $ r=12,s=1 ,k=1 $, since $ (s,r-s,r+s,2r-s)=1 $, we can follow the above discussion to derive $N(r,s,k) $.  By \eqref{main-transform},  we obtain that
			\begin{align*}		&\quad\frac{q^{rk(k+1)/2-sk}\sum_{j=0}^{\infty}\left(q^{t_{1,j}}-q^{t_{2,j}}-q^{t_{3,j}}+q^{t_{4,j}}\right)}{(1-q^{s})(1-q^{r-s})(1-q^{r+s})(1-q^{2r-s})}\\
				&=\frac{q^{11}}{(q,q^{11};q^{12})_{2}}  \sum_{j=0}^{\infty}(-1)^j  q^{  6 \left(j^2+2 j  +j\right)-j }\left(1-q^{2 j+3}\right):=q^{11}\sum_{n\geq0} b(n)q^{n}.
			\end{align*}
			By applying Mathematica, one can see that  $ b(49)=-1 $ while $ b(n)\geq 0 $ for $ n\neq 49 $. By \eqref{p-bound-non-k} and \eqref{def-F1}, we have $ p=2\times12^{3}=3456 $ and $ F_{1}= (p+k) (2 pr-r+2 s)=286702838$. Furthermore, since
			\[
			N_1(r,s,k)=\max \left\{ 4F^{2}_{1} ,2\times(4.63r)^{9}\right\}+F_1+rk(k+1)/2-sk \geq 3.29\times 10^{17},
			\]
			we claim that Merca's conjecture \eqref{conj-strong-1} is true for $ n\geq 3.29\times 10^{17}$.
		\end{ex}

		Next, we discuss some specific examples of the bounds for $ n $  with respect to $ r,s,k $. The bounds as given in Table 1 are the corresponding $ N_{1}(r,s,k) $ as defined in \eqref{def-F}, which are large enough to make $C(n)\geq 0$ when $n\geq L_1(r,s,k)$.
		\begin{table}[!ht]\label{Tab-1}
			\centering
			\caption{ $L_{1}(r,s,k)$ for $ (s,r-s,r+s,2r-s)=1 $}
			\begin{tabular}{|l|l|l|l|l|l|}
				\hline
				\diagbox{$(r,s)$}{$ k $} & $1 $ & $  10^{2}$ & $  10^{4} $ & $ 10^{6} $ & $ 10^{8} $ \\ \hline
				$ (4,1) $ & 8382   & 5682  &98839  &$ 5.6\times 10^5 $  & 0   \\ \hline
				$ (10,3) $ &  $ 1.67 \times 10^{6}$  &  $ 2 \times 10^{5}$   &  $ 2.65\times 10^6 $& $ 5.37\times 10^7 $ &  $ 8.43\times 10^8 $ \\ \hline
				$ (20,3) $ & $ 5.33 \times 10^7 $   &  $ 3.5 \times 10^{6}$   & $ 2.14\times 10^7 $ & $ 4.48\times 10^8 $ &  $ 8.54\times 10^9 $  \\ \hline
				$ (50,21) $ &  $ 1.90\times 10^{10} $  &  $ 9.13 \times 10^{8}$   & $ 7.7\times 10^8 $ &$ 1.38\times 10^{10} $  &  $ 3.33\times 10^{11} $  \\ \hline
				$ (100,19) $ &$ 2.26 \times 10^{11}$    &  $ 2.67 \times 10^{10}$   & $ 6.18\times 10^9 $ & $ 5.70\times 10^{11} $ &  $ 2.29\times 10^{12} $  \\ \hline
				$ (1000,333) $ &$5.97 \times 10^{15}$    &  $ 1.78 \times 10^{16}$   & $ 8.41\times 10^{14} $ & $ 1.07\times 10^{11} $ &  $ 3.16\times 10^{14} $  \\ \hline
			\end{tabular}
		\end{table}
		
		Note that when $k$ is large, one can obtain a sharper bound than $L_1(r,s,k)$ by applying the method of computing the roots used in obtaining $ L_2(r,s,k) $ or $ L_3(r,s,k) $. Moreover, with the help of computer, if we can show the coefficients of \eqref{main-transform}, \eqref{main-transform-3} and \eqref{En}  are nonnegative for more than  $ L_i(r,s,k) $  terms respectively, then Conjecture \ref{stronger-ja-2} holds directly as a corollary. Otherwise, we have to consult the cases as in the proof of Theorem \ref{thm-large-n}.
		Actually, it will take days for our personal computer to expand terms on the scale of $10^{4}\sim 10^{6}$.
		
		In Table 2 and 3, we list $ L_2(r,s,k) $ and $L_3(r,s,k) $ as given in \eqref{N2} and \eqref{z-0} which are the starting point of the coefficients of $ D(n) $ and $E(n)$ that are nonnegative,  respectively.
		\begin{table}[!ht]\label{Tab-2}
			\centering
			\caption{$  L_2(r,s,k)  $ for $ 2\mid s  $}
			\begin{tabular}{|l|l|l|l|l|l|}
				\hline
				\diagbox{$(r,s)$}{$ k $} & $1 $ & $  10 $ & $  10^{2} $ & $ 10^{4} $ & $ 10^{6} $ \\ \hline
				$ (5,2) $ &175910  & 12 &0 & 0  & 0   \\ \hline
				$ (9,2) $ & $2.22\times 10^{7}$  & 7769  &0  & 0  & 0   \\ \hline
				$ (11,4) $ & $4.90\times 10^{7}$ & $60430$   & 0& 0 & 0 \\
				\hline
				$ (21,4) $ & $8.89\times 10^{9}$ & $ 3.67\times 10^{6}$   & 0& 0 & 0 \\
				\hline
				$ (53,20) $ &  $ 2.68\times 10^{12}$  &  $ 4.34\times 10^{9}$ & $ 3.22\times 10^{7}$& 0 & 0\\ \hline
			\end{tabular}
		\end{table}
		\begin{table}[!ht]\label{Tab-3}
			\centering
			\caption{$ L_3(r,s,k)  $ for $ 2\nmid s  $}
			\begin{tabular}{|l|l|l|l|l|l|}
				\hline
				\diagbox{$(r,s)$}{$ k $} & $1 $ & $  10 $ & $  10^{2} $ & $ 10^{4} $ & $ 10^{6} $ \\ \hline
				$ (2,1) $ &216  & 0  &0 & 0  & 0   \\ \hline
				$ (3,1) $ & 11835   & 0  &0  & 0  & 0   \\ \hline
				$ (10,3) $ & $5.79\times 10^{7}$ & $ 48525$   & 0& 0 & 0 \\
				\hline
				$ (20,3) $ & $1.87\times 10^{10}$ & $ 4.86\times 10^{6}$   & 0& 0 & 0 \\
				\hline
				$ (50,21) $ &  $ 1.77\times 10^{12}$  &  $ 3.51\times 10^{9}$ &  $ 2.60\times 10^{7}$& 0 & 0\\ \hline
			\end{tabular}
		\end{table}

		Remark that the first two lines of Table 3 are just equivalent to the special cases of  $R=2S$ and $R=3S$ \cite{DingLisa} of Conjecture \ref{conj-strong-1}, which can be fully confirmed with the help of computer programs, respectively.
		
		\section{Concluding remarks}\label{conclusion}

		For the partition functions  as given in Theorem \ref{main-thm}, we can further consider the following more general form
		\begin{align}\label{gen}
			(-1)^{k} \frac{\sum_{j=k}^{\infty}(-1)^j q^{r j(j+1) / 2}\left(q^{-sj}-q^{( j+1) s}\right)}{\left(q^s;q^r)_{\ell_1} (q^{r-s}; q^r\right)_{\ell_2} },
		\end{align}
		where $r$ and $s$ are coprime positive integers such that $1\leq s< r/2$, $k \geq 1,\ \ell_{1}\geq 0,\ \ell_{2}\geq 0 $ and denote $\ell=\ell_1+\ell_2$. Then by dividing these  $\ell$ integers $ \{ir+s\}_{0\leq i< \ell_{1} } $ and $ \{(j+1)r-s\}_{0\leq j< \ell_{2}} $  into  groups of coprime numbers and following the similar discussions based on Lemma \ref{lem-coprime} and the circle method, one can try to show the nonnegativity of the coefficient of $q^n$.
		
		Moreover, employing our method, one may study the following truncated quintuple product identity, which was also investigate by Chan, Ho and Mao in \cite{Maoquint}:
		\begin{align*}
			\frac{\sum_{j=-k}^{k}(-1)^j q^{R j(3j+1) }\left(q^{3jS}-q^{-(3 j+1) S}\right)}{\left(q^{-S}, q^{R+S}, q^R ; q^R\right)_{\infty} \left(q^{R-2S}, q^{R+2S}; q^{2R}\right)_{\infty}}.
		\end{align*}
		
		At last, we proposed the following generalized truncated version for Conjecture \ref{A-M-conj}.
		
		\begin{conj}\label{new-conj}
			For positive integers $1\leq S< R$ with $k>  \ell \geq1$ and $ k\geq 4 $, there exists  an integer $N$ such that the theta series
			\begin{align}
				\frac{(-1)^{ k-1} \left(\sum_{j=-k+1}^{k} -\sum_{j=-\ell+1}^{\ell}\right)(-1)^j q^{R j(j+1) / 2+S j}    }{\left(q^S, q^{R-S}, q^R ; q^R\right)_{\infty}}
			\end{align}
			has nonnegative coefficients when $ n\geq N$.
		\end{conj}
		The above conjecture is equivalent to the following inequality for the partition function $P_{R,S}(n) $, where $ P_{R,S} (n)$ denotes the number of partition of $ n$ with parts congruent to $ \ \pm S$ or $R$ modulo $R$.
		\begin{conj}
			Under the same conditions as given  in Conjecture \ref{new-conj},  we have that for $n\geq N$
			\begin{align*}
				(-1)^{k-1}\bigg(\sum_{j=-k+1}^{-\ell}+\sum_{j=\ell+1}^{k}\bigg)   (-1)^{j}  P_{R,S}\Big(n-\left(  R j(j+1) / 2+S j  \right)\Big)  \geq 0.
			\end{align*}
		\end{conj}
		For Conjecture \ref{new-conj}, if $k-\ell\equiv 1\ (\bmod\ 2)$, then  $N=0$. At this time, the conjecture is a direct corollary of Conjecture \ref{A-M-conj} and we have the following theorem.
		\begin{thm}
			When $k-\ell\equiv 1\ (\bmod\ 2)$, we have
			\begin{align*}
				(-1)^{\ell}\bigg(\sum_{j=-k+1}^{-\ell+1}+\sum_{j=\ell}^{k}\bigg)   (-1)^{j}  P_{R,S}\big(n-(  R j(j+1) / 2+S j)\big)  \geq 0.
			\end{align*}
		\end{thm}
		
		\section*{Acknowledgments}
		This work is supported by the National Natural Science Foundation of China (Grant No. 12071235), and the Fundamental Research Funds for the Central Universities.
		
		\section*{Declarations}
		\textbf{Competing interests} The authors declare no competing interests.

		\section*{Appendix}\label{Appendix}
		
		In this section, we list the specific expressions for the auxiliary functions that appear in the proof of our main results in Section 3 and Section 4.
		
		{\noindent $\bullet$ } Firstly, for  $ 6\Delta C_{0}(t_{4,p-1},p) $, $ 6\Delta C_{1}(t_{1,p},p) $, $  6\Delta C_{2}(t_{2,p},p) $ and $ 6\Delta C_{3}(t_{3,p},p) $  used in   the discussion of the proof of Theorem  \ref{main-thm}, their expressions are given as follows:
		
		{\small {\noindent \bf  Case 0:}
			\begin{align*}
				&6\Delta C_{0}(t_{4,p-1},p) =\boldsymbol{12 k p^4 r^2 s} +
				p^3 (-32 s^3 + r^2 (-12 k s + 24 k^2 s) + r (24 s^2 + 24 k s^2))\\& +
				p^2 (12 s^3 - 36 k s^3 + r^2 (30 s + 33 k s - 18 k^2 s + 12 k^3 s) +
				r (-66 s^2 + 18 k s^2 + 36 k^2 s^2)) \\&+
				p (\boldsymbol{-96 r^8 s^2 }+ 2 s^3 + 6 k s^3 - 12 k^2 s^3 + 360 r^6 s^4 -
				24 s^6 + 24 s^{10} - 2 r^5 (-24 s + 60 s^5) \\&-
				2 r^4 (-12 s^2 + 204 s^6) - 2 r^3 (60 s^3 - 120 s^7) +
				r^2 (-36 s + 15 k s + 33 k^2 s - 6 k^3 s  \\&+ 120 s^8)+
				r (13 s^2 - 33 k s^2 + 12 k^3 s^2 - 2 (-36 s^5 + 60 s^9))).
			\end{align*}
			{\noindent \bf  Case 1:}
			\begin{align*}
				&6\Delta C_{1}(t_{1,p},p)=   \boldsymbol{12 k p^4 r^2 s}+p^3 (r^2 (24 k^2 s+12 k s)+r (24 s^2-24 k s^2)-32 s^3)\\&+p^2 (r (-36 k^2 s^2+18 k s^2+66 s^2)+r^2 (12 k^3 s+18 k^2 s+33 k s-30 s)-36 k s^3-12 s^3)\\&+p (r (-12 k^3 s^2+33 k s^2+13 s^2-2 (60 s^9-36 s^5))-12 k^2 s^3+r^2 (6 k^3 s+33 k^2 s\\&-15 k s+120 s^8-36 s)-6 k s^3-\boldsymbol{96 r^8 s^2}+360 r^6 s^4-2 r^5 (60 s^5-24 s)-2 r^4 (204 s^6\\&-12 s^2)-2 r^3 (60 s^3-120 s^7)+24 s^{10}-24 s^6+2 s^3)-24 r^8 s^2+24 r^7 s^2+r^6 (90 s^4\\&-24 s^3-24 s^2)+r^5 (-30 s^5-42 s^4+24 s^3+33 s)+r^4 (-102 s^6+48 s^5+42 s^4\\&-(5 s^2)/2-24 s)+r^3 (60 s^7+12 s^6-48 s^5-54 s^3+12 s^2+2 s)+r^2 (30 s^8-24 s^7\\&-12 s^6+(19 s^4)/2+24 s^3-s^2-21/2)+r (-30 s^9+6 s^8+24 s^7+21 s^5-12 s^4-2 s^3\\&-s+6)+6 s^{10}-6 s^8-7 s^6+s^4+s^2-1.
			\end{align*}
			{\noindent \bf  Case 2:}
			\begin{align*}
				&6\Delta C_{2}(t_{2,p},p)=\boldsymbol{12 k p^4 r^2 s}+p^3 (r^2 (24 k^2 s+12 k s)+r (24 k s^2-24 s^2)+32 s^3)\\&+r (24 k^2 s^2+26 k s^2-60 s^9+36 s^5+7 s^2)+12 k^2 s^3+p^2 (r (36 k^2 s^2-18 k s^2+30 s^2)\\&+r^2 (12 k^3 s+18 k^2 s+33 k s-30 s)+60 k s^3+36 s^3)+p (r (12 k^3 s^2+63 k s^2+35 s^2\\&-2 (60 s^9-36 s^5))+36 k^2 s^3+r^2 (6 k^3 s+33 k^2 s-15 k s+120 s^8+6 s)+42 k s^3\\&-\boldsymbol{96 r^8 s^2}+360 r^6 s^4-2 r^5 (60 s^5-24 s)-2 r^4 (204 s^6-12 s^2)-2 r^3 (60 s^3-120 s^7)\\&+24 s^{10}-24 s^6+10 s^3)++r^2 (21 k s+60 s^8+(21 s)/2)+4 k s^3-48 r^8 s^2+180 r^6 s^4\\&+r^5 (24 s-60 s^5)+r^4 (12 s^2-204 s^6)+r^3 (120 s^7-60 s^3)+12 s^{10}-12 s^6+8 k^3 s^3.
			\end{align*}
			{\noindent \bf  Case 3:}
			\begin{align*}
				&6\Delta C_{3}(t_{3,p},p)= \boldsymbol{12 k p^4 r^2 s}+p^3 (r^2 (24 k^2 s+36 k s)+r (-24 k s^2-24 s^2)+32 s^3)\\&+p^2 (r (-36 k^2 s^2-90 k s^2-102 s^2)+r^2 (12 k^3 s+54 k^2 s+69 k s+30 s)+60 k s^3\\&+60 s^3)+p (36 k^2 s^3+r^2 (18 k^3 s+69 k^2 s+93 k s+120 s^8+66 s)+r (-12 k^3 s^2\\&-72 k^2 s^2-171 k s^2-120 s^9+72 s^5-97 s^2)+78 k s^3-\boldsymbol{96 r^8 s^2}+360 r^6 s^4\\&+r^5 (48 s-120 s^5)+r^4 (24 s^2-408 s^6)+r^3 (240 s^7-120 s^3)+24 s^{10}-24 s^6\\&+34 s^3)+r^2 (6 k^3 s+39 k^2 s+69 k s+90 s^8-24 s^7-12 s^6+(19 s^4)/2+24 s^3-s^2\\&+51 s/2-21/2)+r (-12 k^3 s^2-60 k^2 s^2-79 k s^2-90 s^9+6 s^8+24 s^7+57 s^5\\&-12 s^4-2 s^3-26 s^2-s+6)++22 k s^3-72 r^8 s^2+24 r^7 s^2+r^6 (270 s^4-24 s^3\\&-24 s^2)+r^5 (-90 s^5-42 s^4+24 s^3+57 s)+r^4 (-306 s^6+48 s^5+42 s^4+(19 s^2)/2\\&-24 s)+r^3 (180 s^7+12 s^6-48 s^5-114 s^3+12 s^2+2 s)+18 s^{10}-6 s^8\\&-19 s^6+s^4+8 k^3 s^3+24 k^2 s^3+6 s^3+s^2-1.
		\end{align*}}
		
		{\noindent $\bullet$ } Then, $D_{0}(n,p)/(2\Delta_{1})$, $D_{1,d}(n,p)/(2\Delta_{1})$, $D_{2,d}(n,p)/(2\Delta_{1})$, $D_{3,d}(n,p)/(2\Delta_{1})$, $ D_{0,d}(t_{1,p},p)$, $D_{1,d}(t_{1,p},p)$, $D_{1,d}(-r + p r + 2 k p r + 2 p^2 r - s/2,p)$,  $D_{2,d}(t_{2,p},p)$, $D_{3,d}(t_{3,p},p) $, $D_{3,d}(t_{4,p},p) $ used in the proof of   Theorem  \ref{subthm-2|s}  in Section \ref{non-coprime} are given as follows:
		{\small \begin{align*}
				&2\Delta_{1}D_{0}(n,p) =-4 n p s - 6 p r s - 2 k p r s + 4 k^2 p r s + 12 k p^2 r s +
				8 p^3 r s + 4 k p s^2\\
				& -
				2 p (4 s^2 + 4 r^3 s^2 - 2 s^3 + 4 s^5 + r^2 (2 s - 4 s^3) +
				r (-4 s - 4 s^4));
				\\
				&2\Delta_{1}D_{1,d}(n,p) = n^2 -2 + 2 r + p^2 r^2 + 4 k p^2 r^2 + 4 k^2 p^2 r^2 + 4 p^3 r^2 +
				8 k p^3 r^2 \\
				&+ 4 p^4 r^2 + 2 s - 6 p r s - 2 k p r s + 4 k^2 p r s -
				4 p^2 r s + 4 k p^2 r s - 2 s^2 + 4 k p s^2 + 4 p^2 s^2 \\
				&+ 2 s^3 -
				2 s^5 - p r (2 r + s) - 2 k p r (2 r + s) - 2 p^2 r (2 r + s) +
				2 p s (2 r + s) + n (2 r\\
				& - 2 p r - 4 k p r - 4 p^2 r + s) -
				r^3 (-s + 2 s^2) - r^2 (2 s - 2 s^3) - r (-1 - 2 s + s^3 - 2 s^4) \\
				&-
				2 p (4 s^2 + 4 r^3 s^2 - 2 s^3 + 4 s^5 + r^2 (2 s - 4 s^3) +
				r (-4 s - 4 s^4));
				\\
				&2\Delta_{1}D_{2,d}(n,p) =6 r s + 4 k r s + 8 p r s - 10 k p r s - 4 k^2 p r s - 12 p^2 r s -
				12 k p^2 r s \\
				&- 8 p^3 r s - 2 r^2 s - 4 p r^2 s - 4 s^2 - 2 k s^2 -
				4 k^2 s^2 - 8 p s^2 - 4 k p s^2 - 4 r^3 s^2 - 8 p r^3 s^2\\
				& + 2 s^3 +
				4 p s^3 + 4 r^2 s^3 + 8 p r^2 s^3 + 4 r s^4 + 8 p r s^4 - 4 s^5 -
				8 p s^5 + n (2 s + 4 k s + 4 p s);
				\\
				&	2\Delta_{1}D_{3,d}(n,p) =-n^2 - 2 r + r^2 - k^2 r^2 - 6 k p r^2 - 4 k^2 p r^2 - 9 p^2 r^2 -
				16 k p^2 r^2 \\
				&- 4 k^2 p^2 r^2 - 12 p^3 r^2 - 8 k p^3 r^2 -
				4 p^4 r^2 + 7 r s + 7 k r s + 17 p r s - 4 k^2 p r s + 6 p^2 r s \\
				&-
				4 k p^2 r s - 2 r^2 s - 4 p r^2 s - 8 s^2 - 2 k s^2 - 4 k^2 s^2 -
				14 p s^2 - 4 k p s^2 - 4 p^2 s^2 - 4 r^3 s^2 \\
				&- 8 p r^3 s^2 + 4 s^3 +
				4 p s^3 + 4 r^2 s^3 + 8 p r^2 s^3 + 4 r s^4 + 8 p r s^4 - 6 s^5 -
				8 p s^5 + n (2 k r+ 6 p r  \\
				&+ 4 k p r + 4 p^2 r - s + 4 k s) -
				r^3 (-s + 2 s^2) - r^2 (2 s - 2 s^3) - r (-1 - 2 s + s^3 - 2 s^4).
			\end{align*}
			\begin{align*}
				&2\Delta_{1}D_{0,d}(t_{1,p},p)=\boldsymbol{p^2 (r (4 k s-4 s)+8 s^2)}+p (r (4 k^2 s-2 k s-2 (-4 s^4-4 s)-6 s)\\&+4 k s^2-\boldsymbol{8 r^3 s^2}-2 r^2 (2 s-4 s^3)-8 s^5+4 s^3-8 s^2);\\
				&2\Delta_{1}D_{1,d}(t_{1,p},p)=\boldsymbol{p^2 (r (4 k s-4 s)+8 s^2)}+p (r (4 k^2 s-2 k s+8 s^4+2 s)+4 k s^2\\&-\boldsymbol{8 r^3 s^2}+r^2 (8 s^3-4 s)-8 s^5+4 s^3-8 s^2)+r^3 (s-2 s^2)+r^2 (2 s^3-2 s)+r (2 s^4\\&-s^3+2 s+3)-2 s^5+2 s^3-2 s^2+2 s-2;\\
				&2\Delta_{1}D_{1,d}(-r + p r + 2 k p r + 2 p^2 r - s/2,p)=\boldsymbol{p^2 (r (4 k s-4 s)+4 s^2)}+p (r (4 k^2 s\\&-2 k s+8 s^4+6 s)+4 k s^2-\boldsymbol{8 r^3 s^2}+r^2 (8 s^3-4 s)-8 s^5+4 s^3-6 s^2)+r^3 (s\\&-2 s^2)+r^2 (2 s^3-2 s-1)+r (2 s^4-s^3+s+3)-2 s^5+2 s^3-(9 s^2)/4+2 s-2;\\
				&2\Delta_{1}D_{2,d}(t_{2,p},p)=\boldsymbol{p^2 (r (4 k s-4 s)+8 s^2)}+p (r (4 k^2 s-2 k s+8 s^4+10 s)+12 k s^2\\&-\boldsymbol{8 r^3 s^2}+r^2 (8 s^3-4 s)-8 s^5+4 s^3)+4 k^2 s^2+r (4 k s+4 s^4+6 s)+6 k s^2-4 r^3 s^2\\&+r^2 (4 s^3-2 s)-4 s^5+2 s^3-2 s^2;\\
				&2\Delta_{1}D_{3,d}(t_{3,p},p)=\boldsymbol{p^2 r(4 k r s+4 r s-8 s^2)}+p r(4 k^2 r s+10 k r s-12 k s^2-\boldsymbol{8 r^3 s^2}\\&+8 r^2 s^3-4 r^2 s+8 r s^4+18 r s-8 s^5+4 s^3-16 s^2)+r(4 k^2 r s-4 k^2 s^2+10 k r s\\&-6 k s^2-4 r^3 s^2-r^3 (2 s^2-s)+4 r^2 s^3-r^2 (2 s-2 s^3)-2 r^2 s+4 r s^4-r (-2 s^4+s^3\\&-2 s-1)+8 r s-2 r-6 s^5+4 s^3-8 s^2);\\
				&2\Delta_{1}D_{3,d}(t_{4,p},p)=\boldsymbol{p^2 r(4 k r s+4 r s-8 s^2)}+p r(4 k^2 r s+10 k r s-4 k s^2-\boldsymbol{8 r^3 s^2}+8 r^2 s^3\\&-4 r^2 s+8 r s^4+10 r s-8 s^5+4 s^3-24 s^2)+r(4 k^2 r s+6 k r s-4 k s^2-4 r^3 s^2-r^3 (2 s^2\\&-s)+4 r^2 s^3-r^2 (2 s-2 s^3)-2 r^2 s+4 r s^4-r (-2 s^4+s^3-2 s-1)+2 r s-2 r-\\&6 s^5+4 s^3-14 s^2).
			\end{align*}
			
			{\noindent $\bullet$} Finally, $E_{0}(n,p)/(2\Delta_{1})$, $E_{1,d}(n,p)/(2\Delta_{1})$, $E_{2,d}(n,p)/(2\Delta_{1})$, $E_{3,d}(n,p)/(2\Delta_{1})$,  $ E_{0,d}(t_{1,p},p)$, $E_{1,d}(t_{1,p},p)$, $E_{1,d}(3 r - 2 p r - 4 k p r - 4 p^2 r - s,p)$,  $E_{2,d}(t_{2,p},p)$, $E_{3,d}(t_{3,p},p) $, $E_{3,d}(t_{4,p},p) $ appeared in   the proof of Theorem  \ref{subthm-2nmids} are given as follows
			\begin{align*}
				&2\Delta_{2}E_{0}(n,p) =-4 n p s - 8 p r s - 2 k p r s + 4 k^2 p r s + 12 k p^2 r s +
				8 p^3 r s + 4 p s^2 + 4 k p s^2 \\&-
				2 p (4 s^2 + 8 r^3 s^2 + 2 s^3 - 4 s^5 + r^2 (4 s - 20 s^3) +
				r (-4 s - 6 s^2 + 16 s^4));\\
				&2\Delta_{2}E_{1,d}(n,p) =-2 + n^2 + 3 r + p^2 r^2 + 4 k p^2 r^2 + 4 k^2 p^2 r^2 + 4 p^3 r^2 +
				8 k p^3 r^2 + 4 p^4 r^2 \\
				&- p r (3 r - s) - 2 k p r (3 r - s) -
				2 p^2 r (3 r - s) + n (3 r - 2 p r - 4 k p r - 4 p^2 r - s) - 2 s \\
				&-
				4 p^2 r s - 8 k p^2 r s - 8 p^3 r s + 2 p (3 r - s) s - 2 s^2 +
				4 p^2 s^2 - 2 s^3 + 2 s^5 +
				2 p s (-4 r - k r \\
				&+ 2 k^2 r + 6 k p r + 4 p^2 r + 2 s + 2 k s) -
				r^3 (-2 s + 4 s^2) - r^2 (4 s + 3 s^2 - 10 s^3) -
				r (-2 - 2 s \\
				&- 6 s^2 - s^3 + 8 s^4) -
				2 p (4 s^2 + 8 r^3 s^2 + 2 s^3 - 4 s^5 + r^2 (4 s - 20 s^3) +
				r (-4 s - 6 s^2 + 16 s^4));
				\\
				&	2\Delta_{2}E_{2,d}(n,p) =7 r s + 6 k r s + 10 p r s - 10 k p r s - 4 k^2 p r s - 12 p^2 r s -
				12 k p^2 r s - 8 p^3 r s \\&- 4 r^2 s - 8 p r^2 s - 6 s^2 - 6 k s^2 -
				4 k^2 s^2 - 12 p s^2 - 4 k p s^2 + 6 r s^2 + 12 p r s^2 -
				8 r^3 s^2 -
				16 p r^3 s^2\\ &- 2 s^3 - 4 p s^3 + 20 r^2 s^3 +
				40 p r^2 s^3 - 16 r s^4 - 32 p r s^4 + 4 s^5 + 8 p s^5 +
				n (2 s + 4 k s + 4 p s);
				\\
				&2\Delta_{2}E_{3,d}(n,p) =-n^2 - 3 r + 2 r^2 + k r^2 - k^2 r^2 + 3 p r^2 - 4 k p r^2 -
				4 k^2 p r^2 - 7 p^2 r^2 -\\
				& 16 k p^2 r^2 - 4 k^2 p^2 r^2 -
				12 p^3 r^2 - 8 k p^3 r^2 - 4 p^4 r^2 + 5 r s + 7 k r s + 13 p r s -
				4 k p r s - 4 k^2 p r s \\&+ 2 p^2 r s - 4 k p^2 r s - 4 r^2 s -
				8 p r^2 s - 8 s^2 - 6 k s^2 - 4 k^2 s^2 - 18 p s^2 - 4 k p s^2 -
				4 p^2 s^2 + 6 r s^2\\& + 12 p r s^2 - 8 r^3 s^2 - 16 p r^3 s^2 -
				4 s^3 - 4 p s^3 + 20 r^2 s^3 + 40 p r^2 s^3 - 16 r s^4 -
				32 p r s^4 + 6 s^5 \\&+ 8 p s^5 +
				n (-r + 2 k r + 6 p r + 4 k p r + 4 p^2 r + s + 4 k s) -
				r^3 (-2 s + 4 s^2) - r^2 (4 s + 3 s^2 \\&- 10 s^3) -
				r (-2 - 2 s - 6 s^2 - s^3 + 8 s^4);
			\end{align*}
			\begin{align*}
				&2\Delta_{2}E_{0,d}(t_{1,p},p)=\boldsymbol{p^2 (-4 r s + 4 k r s + 8 s^2) }+				p (-8 r s - 2 k r s + 4 k^2 r s + 4 s^2 + 4 k s^2 -
				2 (4 s^2 \\&+ \boldsymbol{8 r^3 s^2 }+ 2 s^3 - 4 s^5 + r^2 (4 s - 20 s^3) +
				r (-4 s - 6 s^2 + 16 s^4))); \\
				&2\Delta_{2}E_{1,d}(t_{1,p},p)=p (4 k^2 r s-2 k r s+4 k s^2-\boldsymbol{16 r^3 s^2}+40 r^2 s^3-8 r^2 s-32 r s^4+12 r s^2+8 s^5\\&-4 s^3-4 s^2)+\boldsymbol{p^2 (4 k r s-4 r s+8 s^2)}-4 r^3 s^2+2 r^3 s+10 r^2 s^3-3 r^2 s^2-4 r^2 s-8 r s^4+r s^3\\&+6 r s^2+2 r s+5 r+2 s^5-2 s^3-2 s^2-2 s-2;\\
				&2\Delta_{2}E_{1,d}(-(3 r - 2 p r - 4 k p r - 4 p^2 r - s)/2,p)=p (4 k^2 r s-2 k r s+4 k s^2-\boldsymbol{16 r^3 s^2}+40 r^2 s^3\\&-8 r^2 s-32 r s^4+12 r s^2+6 r s+8 s^5-4 s^3-6 s^2)+\boldsymbol{p^2 (4 k r s-4 r s+4 s^2)}-4 r^3 s^2+2 r^3 s\\&+10 r^2 s^3-3 r^2 s^2-4 r^2 s-\frac{9}{4} r^2-8 r s^4+r s^3+6 r s^2+\frac{7}{2} r s+5 r+2 s^5-2 s^3-\frac{9}{4} s^2-2 s-2;\\
				&2\Delta_{2}E_{2,d}(t_{2,p},p)=\boldsymbol{p^2 (4kr s-4 rs)+8 s^2)}+p (r (4 k^2 s-2 k s-48 s^4+18 s^2+16 s)\\&+12 k s^2-24 r^3 s^2+r^2 (60 s^3-12 s)+12 s^5-6 s^3-8 s^2)+4 k^2 s^2+r (6 k s\\&-16 s^4+6 s^2+7 s)+2 k s^2-\boldsymbol{8 r^3 s^2}+r^2 (20 s^3-4 s)+4 s^5-2 s^3-4 s^2;\\
				&2\Delta_{2}E_{3,d}(t_{3,p},p)=p r(4 k^2 r s+10 k r s-12 k s^2-\boldsymbol{16 r^3 s^2}+40 r^2 s^3-8 r^2 s-32 r s^4+12 r s^2\\&+22 r s+8 s^5-4 s^3-24 s^2)+r(4 k^2 r s-4 k^2 s^2+12 k r s-10 k s^2-12 r^3 s^2+2 r^3 s+30 r^2 s^3\\&-3 r^2 s^2-8 r^2 s-24 r s^4+r s^3+12 r s^2+11 r s-r+6 s^5-4 s^3-10 s^2)+\boldsymbol{p^2 r(4 k r s+4 r s-8 s^2)};\\
				&2\Delta_{2}E_{3,d}(t_{4,p},p)=p r(4 k^2 r s+10 k r s-4 k s^2-\boldsymbol{16 r^3 s^2}+40 r^2 s^3-8 r^2 s-32 r s^4+12 r s^2\\&+10 r s+8 s^5-4 s^3-24 s^2)+r(4 k^2 r s+6 k r s-4 k s^2-12 r^3 s^2+2 r^3 s+30 r^2 s^3-3 r^2 s^2-8 r^2 s\\&-24 r s^4+r s^3+12 r s^2+2 r s-r+6 s^5-4 s^3-10 s^2)+\boldsymbol{p^2 r(4 k r s+4 r s-8 s^2)}.
		\end{align*}}

	\end{document}